\documentclass[letterpaper, 11pt]{amsart}
\usepackage{mathtools}
\usepackage{amsmath}
\usepackage[T1]{fontenc}
\usepackage{amssymb}
\usepackage{yhmath}
\usepackage{comment}
\usepackage{graphicx} 
\usepackage{ mathrsfs }
\usepackage{bbm} 
\usepackage{xcolor} 
\usepackage{tikz-cd}
\usepackage{tikz}
\usetikzlibrary{patterns}
\usepackage{hyperref}

\DeclareMathAlphabet{\mathpzc}{OT1}{pzc}{m}{it}

\usepackage{thmtools}
\usepackage{thm-restate}

\usepackage{caption}

\newtheorem{theorem}{Theorem}[section]

\newtheorem*{claim*}{Claim}

\newtheorem{lemma}[theorem]{Lemma}
\newtheorem{lem}[theorem]{Lemma}
\newtheorem{corollary}[theorem]{Corollary}

\newtheorem{cor}[theorem]{Corollary}

\newtheorem{problem}[theorem]{Problem}

\newtheorem{proposition}[theorem]{Proposition}

\newtheorem{fact}[theorem]{Fact}

\theoremstyle{definition}
\newtheorem{definition}[theorem]{Definition}

\newtheorem{example}[theorem]{Example}

\theoremstyle{remark}
\newtheorem{remark}[theorem]{Remark}

\numberwithin{equation}{section}

\DeclareMathOperator{\Isom}{\mathsf{Isom}}
\DeclareMathOperator{\Hom}{Hom}

\DeclareMathOperator{\diam}{diam}

\DeclareMathOperator{\rank}{rank}

\DeclareMathOperator{\dist}{d}

\DeclareMathOperator{\Dc}{\mathcal{D}}
\DeclareMathOperator{\Fc}{\mathcal{F}}

\DeclareMathOperator{\Hc}{\mathcal{H}}

\DeclareMathOperator{\Nc}{\mathcal{N}}

\DeclareMathOperator{\Rc}{\mathcal{R}}

\DeclareMathOperator{\Sc}{\mathcal{S}}

\DeclareMathOperator{\Hb}{\mathbb{H}}

\DeclareMathOperator{\Nb}{\mathbb{N}}
\DeclareMathOperator{\Pb}{\mathbb{P}}
\DeclareMathOperator{\Rb}{\mathbb{R}}

\DeclareMathOperator{\Zb}{\mathbb{Z}}

\newcommand{\N}{\mathbb{N}}
\newcommand{\R}{\mathbb{R}}
\newcommand{\Z}{\mathbb{Z}}

\DeclareMathOperator{\Gsf}{\mathsf{G}}
\DeclareMathOperator{\Psf}{\mathsf{P}}

\DeclareMathOperator{\PSL}{\mathsf{PSL}}

\DeclareMathOperator{\GL}{\mathsf{GL}}

\newcommand{\abs}[1]{\left|#1\right|}
\newcommand{\norm}[1]{\left\|#1\right\|}

\newcommand{\Ga}{\Gamma}
\newcommand{\ga}{\gamma}
\newcommand{\F}{\mathcal{F}}
\newcommand{\La}{\Lambda}
\newcommand{\la}{\lambda}

\newcommand{\Msf}{\mathsf{M}}
\newcommand{\Asf}{\mathsf{A}}

\newcommand{\Nsf}{\mathsf{N}}
\newcommand{\Ksf}{\mathsf{K}}

\newcommand{\Lsf}{\mathsf{L}}

\newcommand{\Usf}{\mathsf{U}}
\newcommand{\Ssf}{\mathsf{S}}

\newcommand{\fa}{\mathfrak{a}}

\newcommand{\Hsf}{\mathsf{H}}

\renewcommand{\L}{\mathcal{L}}

\newcommand{\ba}{\backslash}

\newcommand{\opp}{\operatorname{i}}

\newcommand{\Lie}{\operatorname{Lie}}

\newcommand{\CAT}{\operatorname{CAT}}

\newcommand{\Spec}{\operatorname{Spec}}

\newcommand{\Hor}{\mathcal{H}}
\newcommand{\E}{\mathcal{E}}
\newcommand{\BR}{\mathrm{BR}}

\newcommand{\inte}{\operatorname{int}}

\begin{document}

\title[Horospherical measures in higher rank: The Full Story]{Classification of horospherical invariant measures in higher rank: The Full Story}

\author{Inhyeok Choi}
\address{School of Mathematics, KIAS, Hoegi-ro 85, Dongdaemun-gu, Seoul 02455, South Korea 
\newline
\indent Simons Laufer Mathematical Sciences Institute, Berkeley, CA 94720}
\email{inhyeokchoi48@gmail.com}

\author{Dongryul M. Kim}
\address{Department of Mathematics, Yale University, New Haven, CT 06511
\newline
\indent Simons Laufer Mathematical Sciences Institute, Berkeley, CA 94720}
\email{dongryul.kim97@gmail.com}

\begin{abstract}
    In this paper, we classify horospherical invariant Radon measures for Anosov subgroups of arbitrary semisimple real algebraic groups. This generalizes the works of Burger and Roblin in rank one to higher ranks. At the same time, this extends the works of Furstenberg, Veech, and Dani, and a special case of Ratner's theorem for finite-volume homogeneous spaces to infinite-volume Anosov homogeneous spaces. 

    Especially, this resolves the open problems proposed by Landesberg--Lee--Lindenstrauss--Oh and by Oh. Our measure classification is in fact for a more general class of discrete subgroups, including relatively Anosov subgroups with respect to any parabolic subgroups, not necessarily minimal. We also obtain results for their normal subgroups. Our method is rather geometric, not relying on continuous flows or ergodic theorems.

    \bigskip

    \bigskip

    \bigskip

\begin{center}
{\bf {\large D}ECLARATION}
\end{center}

\bigskip

Prior to this paper, we wrote a paper \cite{CK_pradak} dealing with the case that the ambient Lie group is a product of rank one groups (in fact for products of $\CAT(-1)$-spaces), which is the setting originally asked by Landesberg--Lee--Lindenstrauss--Oh and Oh. As many parts of the argument in this paper are simplified in such a case, \cite{CK_pradak} might be more accessible to readers who are not familiar with homogeneous dynamics but have some geometry background.

In this regard, we decided to keep \cite{CK_pradak} public on arXiv and our  websites, but {\bf not to  publish \cite{CK_pradak} in any  journal.} The paper \cite{CK_pradak} can be considered as a teaser for this paper.

\end{abstract}

\maketitle

\tableofcontents

\section{Introduction}

Given a dynamical system, classifying invariant measures is a natural and important question with many applications, as also indicated by the celebrated theorem of Ratner \cite{Ratner_measure} for instance. We study this question for dynamical systems given by \emph{horospherical actions}.

Let $\Gsf$ be a connected semisimple real algebraic group and $\Psf < \Gsf$ its minimal parabolic subgroup with a fixed Langlands decomposition $\Psf = \Msf \Asf \Nsf$, where $\Asf$ is a maximal real split torus of $\Gsf$, $\Msf < \Psf$ is a maximal compact subgroup commuting with $\Asf$, and $\Nsf$ is the unipotent radical of $\Psf$.

Let $\Ga < \Gsf$ be a Zariski dense discrete subgroup. The right multiplication of $\Nsf \Msf$ or $\Nsf$  on $\Ga \ba \Gsf$ is called (maximal) \emph{horospherical action}. 
For a uniform lattice $\Ga < \Gsf$, the $\Nsf\Msf$-action on $\Ga \ba \Gsf$ is uniquely ergodic,\footnote{By unique ergodicity, we mean that there exists a unique invariant ergodic Radon measure up to a constant multiple.} with the Haar measure for $\Gsf$ as the ergodic measure. This was first shown for $\Gsf = \PSL(2, \R)$ by Furstenberg \cite{furstenberg1973the-unique}, and by Veech \cite{veech1977unique} in general. When $\Ga < \Gsf$ is a non-uniform lattice, Dani classified all $\Nsf \Msf$-invariant ergodic Radon measures on $\Ga \ba \Gsf$ (\cite{dani1978invariant}, \cite{dani1981invariant}). In fact, all results above hold for the $\Nsf$-action as well, and they are generalized further by Ratner \cite{Ratner_measure} to actions generated by unipotent elements, which is now called Ratner's theorem.

In contrast, the case that $\Ga < \Gsf$ is not a lattice, i.e., $\Ga$ has infinite covolume, has remained mysterious for general $\Gsf$. For a non-lattice Zariski dense discrete $\Ga < \Gsf$, a natural $\Nsf \Msf$-invariant subspace of $\Ga \ba \Gsf$ that has been considered is 
$$
\text{the unique $\Psf$-minimal set } \E_{\Ga} \subset \Ga \ba \Gsf,
$$
whose uniqueness for higher-rank $\Gsf$ is due to Benoist \cite{Benoist1997proprietes}. Indeed, $\E_{\Ga}$ consists of $[g] \in \Ga \ba \Gsf$ whose ``forward endpoint'' $g \Psf$ in the Furstenberg boundary $\Fc := \Gsf / \Psf$ belongs to the limit set $\La(\Ga) \subset \Fc$, the unique $\Ga$-minimal set. When $\Gsf$ is of rank one, 
it is easy to see this picture and that any $\Nsf \Msf$-orbit in $(\Ga \ba \Gsf) \smallsetminus \E_{\Ga}$ is closed.

The first classification of horospherical invariant measure in this infinite-volume setting is due to Burger \cite{Burger_horoc}, who considered the case that $\Gsf = \PSL(2, \R)$ and $\Ga < \Gsf$ is convex cocompact whose limit set $\La(\Ga)$ has Hausdorff dimension strictly bigger than $1/2$. Burger proved the unique ergodicity of the $\Nsf \Msf$-action on $\E_{\Ga}$ in this case.
For a general rank one group $\Gsf$ and $\Ga < \Gsf$ geometrically finite, Roblin classified all $\Nsf \Msf$-invariant ergodic Radon measures on $\E_{\Ga}$ \cite{Roblin2003ergodicite}, by showing that there exists an $\Nsf \Msf$-invariant Radon measure $\mu_{\Ga}$ on $\E_{\Ga}$ so that any $\Nsf \Msf$-invariant ergodic Radon measure on $\E_{\Ga}$ is either supported on a closed $\Nsf \Msf$-orbit, or a constant multiple of $\mu_{\Ga}$. The measure $\mu_{\Ga}$ is now called
\begin{center}
\emph{Burger--Roblin measure} of $\Ga$.
\end{center} 
Later, Winter showed that the Burger--Roblin measure is $\Nsf$-ergodic and extended the Roblin's classification to $\Nsf$-invariant Radon measures \cite{Winter_mixing}. 
For geometrically infinite cases, Babillot and Ledrappier first discovered that there may be continuous family of $\Nsf \Msf$-invariant ergodic Radon measures (\cite{Babillot_nilpotent}, \cite{BL_covers}); see also (\cite{Sarig_abelian}, \cite{Sarig_genus}, \cite{Ledrappier_invariant}, \cite{LS_periodic}, \cite{Winter_mixing}, \cite{OP_local}, \cite{LL_Radon}, \cite{Landesberg_horospherically}, \cite{LLLO_Horospherical})  for partial classification
results in the rank-one case.

We now move to the case that 
\begin{center}
$\Gsf$ is of \emph{higher rank}.
\end{center} Edwards--Lee--Oh extended the notion of Burger--Roblin measure to higher rank, introducing higher-rank Burger--Roblin measures \cite{Edwards2020anosov}. Their ergodicity with respect to horospherical actions was proved for Zariski dense Borel Anosov subgroups, higher-rank analogues of convex cocompact subgroups, by Lee--Oh \cite{LO_invariant}. For a larger class of discrete subgroups including relatively Anosov subgroups, higher-rank analogues of geometrically finite subgroups, the ergodicity was shown by the second author \cite{kim2024conformal}.

On the other hand, even for Zariski dense Borel Anosov subgroups, there was no single specific example that the horospherical invariant measures are classified, except our teaser paper \cite{CK_pradak}. A work related to this direction is by Landesberg--Lee--Lindenstrauss--Oh \cite{LLLO_Horospherical}. 

\subsection{Work of Landesberg--Lee--Lindenstrauss--Oh and Problems}

As we are answering an open problem proposed in \cite{LLLO_Horospherical} which has been motivational about the classification of horospherical invariant measures in higher rank, we briefly describe their work.
They considered a specific group 
$$ 
{ \Gsf := \prod_{i = 1}^r \Gsf_i}
$$
where $\Gsf_i$ is a simple real algebraic group of rank one. In this case, we have $r = \rank \Gsf$. (We will consider \emph{arbitrary $\Gsf$}, in contrast.)

Even for the case that $\Gsf$ is a product of rank one groups as above, the classification of horospherical invariant measure was not known before our teaser paper \cite{CK_pradak}. Instead, in \cite{LLLO_Horospherical}, they considered the directionally recurrent set in $ \Ga \ba \Gsf$ for each \emph{1-dimensional diagonal flow} (or, directional flow). More precisely, denote by $\fa := \Lie \Asf$ and fix a closed Weyl chamber $\fa^+ \subset \fa$ so that for each $n \in \Nsf$, the set $\{ a^{-1} n a : a \in \exp \fa^+\} \subset \Nsf$  is precompact. Then for each $v \in \inte \fa^+$, they considered  the set $\mathcal{R}_{\Ga, v} \subset \E_{\Ga}$ consisting of elements each of whose 1-dimensional $\exp (\R_{>0} v)$-orbit is recurrent to a compact subset. 

Their main result is about the horospherical action restricted on $\mathcal{R}_{\Ga, v} \subset \E_{\Ga}$, for $\Ga$ Zariski dense Borel Anosov. 
In the setting of product of rank-one Lie groups, $\Ga < \Gsf $ is Borel Anosov if 
the projection $\Ga \to \Gsf_i$ has finite kernel and convex cocompact image for all $1 \le i \le r$. Based on the ergodicity results of Lee--Oh (\cite{LO_invariant}, \cite{LO_ergodic}) and Burger--Landesberg--Lee--Oh \cite{BLLO}, the rigidity result of \cite{LLLO_Horospherical} is as follows:

\begin{theorem}[{\cite{LLLO_Horospherical}}] \label{thm:LLLO}
    Suppose that $\Gsf$ is a product of rank one groups. 
    Let $\Ga < \Gsf $ be a Zariski dense Borel Anosov subgroup and $v \in \inte \fa^+$. Let $\L_{\Ga} \subset \fa^+$ denote the limit cone\footnote{The limit cone of $\Ga$ is the asymptotic cone of the Cartan projections of $\Ga$ in $\fa$. We  will revisit this later.} of $\Ga$.
    \begin{enumerate}
        \item If $\rank \Gsf \le 3$ and $v \in \inte \L_{\Ga}$, then the $\Nsf$-action and the $\Nsf \Msf$-action on $\mathcal{R}_{\Ga, v} \subset \E_{\Ga}$ are uniquely ergodic.
        \item If $\rank \Gsf > 3$ or $v \notin \inte \L_{\Ga}$, then there exists no non-zero, $\Nsf$-invariant Radon measure supported on $\mathcal{R}_{\Ga, v} \subset \E_{\Ga}$.
    \end{enumerate}
\end{theorem}

The ergodic measures in (1) above are higher-rank Burger--Roblin measures, whose ergodicity was proved in \cite{LO_ergodic}, and being supported on the directionally recurrent set was proved in \cite{BLLO}. Delaying their definitions, we  note that in contrast to rank-one settings, they come as a family of mutually singular measures, because higher-rank Patterson--Sullivan measures do so. The reason for the rank dichotomy in Theorem \ref{thm:LLLO}(2) is that $\mathcal{R}_{\Ga, v}$ has zero Burger--Roblin measures when $r > 3$ \cite{BLLO}.

In view of Theorem \ref{thm:LLLO}, the following open problem was proposed by Landesberg--Lee--Lindenstrauss--Oh,  towards classifying horospherical invariant measures.

\begin{problem}[{\cite[Open problem 1.8]{LLLO_Horospherical}}] \label{ques:LLLO}
Let $\Ga < \Gsf$ be a Zariski dense Borel Anosov subgroup and suppose $\rank \Gsf \le 3$. Is any $\Nsf$-invariant ergodic Radon measure on $\E_{\Ga}$ supported on $\mathcal{R}_{\Ga, v}$ for some $v \in \inte \fa^+$?
\end{problem}

More generally, in a very recent preprint for the Proceedings of the ICM 2026, Oh asked for horospherical measure classification for Anosov subgroups without any assumption on $\rank \Gsf$.

\begin{problem}[{\cite[Section 8.2]{oh2025dynamics}}] \label{ques:Anosov}
Let $\Ga < \Gsf$ be a Zariski dense Borel Anosov subgroup. Is any $\Nsf$-invariant ergodic Radon measure on $\E_{\Ga}$ a Burger--Roblin measure?
\end{problem}

Main results of this paper are resolvents of Problem \ref{ques:LLLO} and Problem \ref{ques:Anosov}, open problems proposed by Landesberg--Lee--Lindenstrauss--Oh in \cite{LLLO_Horospherical} and by Oh in \cite{oh2025dynamics}. Indeed, we give a complete classification of horospherical invariant measures.

\begin{remark}
    Before stating our main results, we remark the following regarding our setup:
    \begin{itemize}
        \item While products of rank one groups are considered above, we deal with general semisimple real algebraic groups. Hence, in the rest of this paper,
        $\Gsf$ is an \emph{arbitrary semisimple real algebraic group}.
        \item We also consider Anosov and relatively Anosov subgroups with respect to \emph{any parabolic subgroup}, not necessarily Borel. In fact, we consider a more general class of discrete subgroups.
        \item Unless $\Gsf$ is a product of rank one groups, higher-rank Burger--Roblin measures are not $\Nsf$-ergodic in general, as shown by ergodic decomposition results (\cite{LO_ergodic}, \cite{kim2024conformal}). In this perspective, our classification tells that any $\Nsf \Msf$-invariant ergodic Radon measure on $\E_{\Ga}$ is a higher-rank Burger--Roblin measure and any $\Nsf$-invariant ergodic Radon measure is an ergodic component of a higher-rank Burger--Roblin measure, which is explicitly described in (\cite{LO_ergodic}, \cite{kim2024conformal}).
    \end{itemize}
\end{remark}

\subsection{Main results for Borel Anosov subgroups}

From now on,

\medskip

\begin{center}
    let $\Gsf$ be an {\bf arbitrary} connected semisimple real algebraic group.
\end{center}

\medskip

For simplicity, we first present our main results for Borel Anosov subgroups as special cases, which resolve the open problems proposed by Landesberg--Lee--Lindenstrauss--Oh (Problem \ref{ques:LLLO}) and by Oh (Problem \ref{ques:Anosov}).

An \emph{Anosov subgroup} is a higher-rank generalization of convex cocompact subgroups, introduced by Labourie \cite{Labourie2006anosov} for surface groups and generalized by Guichard--Wienhard \cite{Guichard2012anosov} for hyperbolic groups. Delaying some precise definitions, we first state our classification as follows:

\begin{theorem} \label{thm:mainAnosov}
    Let $\Ga < \Gsf$ be a Zariski dense Borel Anosov subgroup.
If $\mu$ is an $\Nsf \Msf$-invariant ergodic Radon measure on $\E_{\Ga}$, then 
\begin{center}
$\mu$ is a constant multiple of a Burger--Roblin measure of $\Ga$.
\end{center}
\end{theorem} 
Our proof in fact gives the measure classification result ({\`a} la Landesberg--Lindenstrauss \cite{LL_Radon} and of Landesberg \cite{Landesberg_horospherically}) for normal subgroups of a more general class of subgroups discussed later. See Theorem \ref{thm:trbytrlengthqi}.

From the classification of $\Nsf \Msf$-invariant ergodic Radon measures on $\E_{\Ga}$ above, one can deduce the classification of $\Nsf$-invariant ergodic Radon measures as well, using the $\Nsf$-ergodic decomposition of Burger--Roblin measures for Borel Anosov subgroups by Lee--Oh \cite{LO_ergodic}.
Let $\Psf^{\circ}$ denote the identity component of $\Psf$.
Denote by 
$$\Dc_{\Ga}$$ 
the (finite) collection of $\Psf^{\circ}$-minimal subsets of $\Ga \ba \Gsf$.

\begin{theorem} \label{thm:mainAnosovN}
    Let $\Ga < \Gsf$ be a Zariski dense Borel Anosov subgroup. If $\mu$ is an $\Nsf$-invariant ergodic Radon measure on $\E_{\Ga}$, then
    $$
    \mu \text{ is a constant multiple of } \mu^{\BR}|_{\E}
    $$
    for a Burger--Roblin measure $\mu^{\BR}$ of $\Ga$ and $\E \in \Dc_{\Ga}$.
    
    In particular, if the $\Psf^{\circ}$-action on $\E_{\Ga}$ is minimal, then $\mu$ is a constant multiple of $\mu^{\BR}$.
\end{theorem}
 We note the following regarding Theorem \ref{thm:mainAnosov} and Theorem \ref{thm:mainAnosovN}.
\begin{remark} 

    As mentioned earlier, the natural $\Nsf \Msf$-invariant subspace $\E_{\Ga} \subset \Ga \ba \Gsf$ is where interesting dynamics occur. Indeed, this space is dual to the limit set of $\Gamma$ in the Furstenberg boundary, which is the unique $\Gamma$-minimal set as shown by Benoist \cite{Benoist1997proprietes}.
    
    When $\Gsf$ is of rank one, it is straightforward that any $\Nsf \Msf$-orbit in the complement $(\Ga \ba \Gsf) \smallsetminus \E_{\Ga}$ is closed, and hence measures supported outside $\E_{\Ga}$ are trivial, for this straightforward reason. 
    This easily extends to the case where $\Gsf$ is a product of rank one groups as well \cite[Section 4.4]{CK_pradak}.
    
    On the other hand, this does not seem to hold for general higher-rank $\Gsf$ (see Remark \ref{rmk:bddIwasawa}). This is the reason why Problem \ref{ques:LLLO} and Problem \ref{ques:Anosov} were asked within $\E_{\Ga}$, besides that $\E_{\Ga}$ is the genuine region for the interesting dynamics.

\end{remark}

We now define Borel Anosov subgroups and Burger--Roblin measures.
A Zariski dense discrete subgroup $\Ga < \Gsf$ is called \emph{Borel Anosov} if
\begin{itemize}
    \item  $\Ga$ is a hyperbolic group, and 
    \item there exists a $\Ga$-equivariant embedding $f : \partial \Ga \to \Fc$ such that for any distinct $x, y \in \partial \Ga$, $f(x)$ and $f(y)$ are in \emph{general position}, i.e., the diagonal $\Gsf$-orbit $\Gsf \cdot (f(x), f(y))$ is open in $\Fc \times \Fc$, where $\Fc = \Gsf / \Psf$ is the Furstenberg boundary.
\end{itemize}

\begin{example} \label{ex:Hitchin}
    A famous example of Borel Anosov subgroup is the so-called \emph{Hitchin subgroup}, as shown by Labourie \cite{Labourie2006anosov}. For a closed connected orientable surface $S$ of genus at least two and $d \ge 3$, a representation 
    $$\pi_1(S) \to \PSL(d, \Rb)$$
     is called a Hitchin representation if it belongs to the same component as the composition $\pi_1(S) \to \PSL(2, \Rb)\to \PSL(d, \Rb)$ in the character variety $\Hom( \pi_1(S), \PSL(d, \Rb))/\sim$, for a cocompact $\pi_1(S) \to \PSL(2, \Rb)$ and irreducible $\PSL(2, \Rb) \to \PSL(d, \Rb)$. Hitchin subgroups are images of Hitchin representations.

\end{example}
 
To define Burger--Roblin measures, recall that $\fa^+$ is a fixed positive Weyl chamber in $\fa = \Lie \Asf$, and fix a maximal compact subgroup $\Ksf < \Gsf$ so that the Cartan decomposition $\Gsf = \Ksf (\exp \fa^+) \Ksf$ holds. Then we have the Furstenberg boundary
$$
\mathcal{F} = \Ksf/ \Msf = \Gsf / \Psf.
$$

Let $\Ga < \Gsf$ be a Zariski dense discrete subgroup.
For $\delta \ge 0$ and a linear form $\psi \in \fa^*$, a Borel probability measure $\nu$ on $\mathcal{F}$ is called a $\delta$-dimensional $\psi$-\emph{Patterson--Sullivan measure} of $\Ga$ if
\begin{equation} \label{eqn:PSmeasintro}
\frac{d g_* \nu}{d \nu}(x) = e^{-\delta \cdot \psi( \sigma(g^{-1}, x))} \quad \text{a.e.}
\end{equation}
where $\sigma$ is the Iwasawa cocycle, i.e., $\sigma(g^{-1}, x) \in \fa$ is the unique element satisfying $g^{-1}k \in \Ksf (\exp \sigma(g^{-1}, x)) \Nsf$, for $k \in \Ksf$ with $x = k \Msf$ in $\Fc$. This notion of Patterson--Sullivan measures was first introduced by Quint \cite{Quint2002Mesures}, generalizing the classical Patterson--Sullivan theory to higher rank. We simply call $\nu$ Patterson--Sullivan measure of $\Ga$ if it is $\delta$-dimensional $\psi$-Patterson--Sullivan measure for some $\delta \ge 0$ and $\psi \in \fa^*$.

In \cite{Edwards2020anosov}, Edwards--Lee--Oh extended the classical Burger--Roblin measure to higher rank. For a $\delta$-dimensional $\psi$-Patterson--Sullivan  measure $\nu$ of $\Ga$ on $\mathcal{F}$, the (higher-rank) \emph{Burger--Roblin measure} associated to $\nu$ is the Radon measure $\mu_{\nu}^{\BR}$ on $\Ga \ba \Gsf$ induced by the $\Ga$-invariant measure $\tilde \mu_{\nu}^{\BR}$ on $\Gsf$ defined as follows: for $g = k (\exp u) n \in \Ksf (\exp \fa) \Nsf$ in Iwasawa decomposition of $\Gsf$,
$$
d \tilde \mu_{\nu}^{\BR}(g) := e^{\delta \cdot \psi(u)} d \tilde \nu (k) du dn
$$
where $\tilde \nu$ is the $\Msf$-invariant lift of $\nu$ to $\Ksf$ and $du$ and $dn$ are Lebesgue measures on $\fa$  and $\Nsf$ respectively. The measure $\mu_{\nu}^{\BR}$ is $\Nsf \Msf$-invariant.

Recall that the limit set $\La(\Ga) \subset \Fc$ of $\Ga$ is the unique $\Ga$-minimal subset \cite{Benoist1997proprietes} and 
$$
\E_{\Ga} = \{ [g] \in \Ga \ba \Gsf : g \Psf \in \La(\Ga) \}.
$$
Hence, $\mu_{\nu}^{\BR}$ is supported on $\E_{\Ga}$ if and only if $\nu$ is supported on $\La(\Ga)$, and in this case, the $\Nsf \Msf$-ergodicity was proved by Lee--Oh \cite{LO_invariant}.
As a corollary of Theorem \ref{thm:mainAnosov} and Theorem \ref{thm:mainAnosovN}, we conclude that Burger--Roblin measures are all such ergodic measures.

\begin{corollary} \label{cor:mainAnosov}
    Let $\Ga < \Gsf$ be a Zariski dense Borel Anosov subgroup.
    Then we have the following classifications of $\Nsf\Msf$-ergodic measures and $\Nsf$-ergodic measures.
    \begin{enumerate}
        \item The following  sets are the same, up to constant multiples:
    \begin{itemize}
        \item $\left\{ \mu_{\nu}^{\BR} : \nu \text{ a Patterson--Sullivan measure of $\Ga$ on $\La(\Ga)$}\right\}$.
        \item the set of all $\Nsf \Msf$-invariant ergodic Radon measures on $\E_{\Ga}$.
    \end{itemize}

    \medskip
        \item The following sets are the same, up to constant multiples:
    \begin{itemize}
        \item $\left\{ \mu_{\nu}^{\BR}|_{\E} : \nu \text{ a Patterson--Sullivan measure of $\Ga$ on $\La(\Ga)$, } \E \in  \Dc_{\Ga}\right\}$.
        \item the set of all $\Nsf$-invariant ergodic Radon measures on $\E_{\Ga}$.
    \end{itemize}
\end{enumerate}
\end{corollary}

The sets of ergodic measures in Corollary \ref{cor:mainAnosov} can be described more explicitly. 
Denote by $\kappa : \Gsf \to \fa^+$ the Cartan projection, defined by the condition $g \in \Ksf (\exp \kappa(g)) \Ksf$ for all $g \in \Gsf$.
The \emph{limit cone} $\L_{\Ga} \subset \fa^+$ of $\Ga$ is the asymptotic cone of Cartan projections $\kappa(\Ga)$. Benoist showed that if $\Ga$ is Zariski dense, $\L_{\Ga}$ is convex and has non-empty interior \cite{Benoist1997proprietes}. For a Zariski dense Borel Anosov subgroup $\Ga < G$, Lee--Oh classified Patterson--Sullivan measures of $\Ga$ on $\La(\Ga)$ in \cite{LO_invariant}, and provided a natural homeomorphism
\begin{equation} \label{eqn:LObijection}
\Pb(\inte \L_{\Ga}) \longleftrightarrow  \{ \mu_{\nu}^{\BR} : \nu \text{ a Patterson--Sullivan measure of $\Ga$ on $\La(\Ga)$}\}
\end{equation}
constructed using tangencies of the growth indicator of $\Ga$,  introduced by Quint \cite{Quint2002divergence}.
Corollary \ref{cor:mainAnosov} is now rephrased as follows:

\begin{corollary}  \label{cor:mainAnosovrephrase}
    Let $\Ga < \Gsf$ be a Zariski dense Borel Anosov subgroup. Then the homeomorphism in Equation \eqref{eqn:LObijection} becomes a homeomorphism
    $$
    \begin{tikzcd}[column sep = huge]
    \inte \L_{\Ga} \arrow[r, <->] & \left\{
        \ \begin{matrix}
            \Nsf \Msf \text{-invariant ergodic}\\
            \text{non-zero Radon measures on } \E_{\Ga}
        \end{matrix} \ \right\}.
    \end{tikzcd}
    $$
    In particular, the above sets are homeomorphic to $\R^{\rank \Gsf}$.
    
    Moreover, the homeomorphism in Equation \eqref{eqn:LObijection} extends to a homeomorphism
    $$
    \begin{tikzcd}[column sep = huge]
    \left( \inte \L_{\Ga} \right) \times \Dc_{\Ga} \arrow[r, <->] & \left\{
        \ \begin{matrix}
            \Nsf \text{-invariant ergodic}\\
            \text{non-zero Radon measures on } \E_{\Ga}
        \end{matrix} \ \right\}.
    \end{tikzcd}
    $$
    In particular, the above sets are homeomorphic to $\R^{\rank \Gsf} \times \{ 1, \dots, \# \Dc_{\Ga} \}$.
\end{corollary}

The role of $\Dc_{\Ga}$ above is to restrict the support of measures.
Combining with the work of Burger--Landesberg--Lee--Oh \cite{BLLO}, our measure classification implies the following.

\begin{corollary} \label{cor:suppondirectional}
    Suppose $\rank \Gsf \le 3$ and let $\Ga < \Gsf$ be a Zariski dense Borel Anosov subgroup. Then the following holds.
    \begin{enumerate}
    \item Any $\Nsf \Msf$-invariant ergodic Radon measure on $\E_{\Ga}$ is supported on $\Rc_{\Ga, v}$ for some $v \in \inte \L_{\Ga}$.
    \item Any $\Nsf$-invariant ergodic Radon measure on $\E_{\Ga}$ is supported on $\Rc_{\Ga, v}$ for some $v \in \inte \L_{\Ga}$.
    \end{enumerate} 
\end{corollary}

Note that $\E_{\Ga}$ is $\Psf^{\circ}$-minimal when $\Gsf$ is a product of rank one groups. Then Corollary \ref{cor:suppondirectional} resolves Problem \ref{ques:LLLO}, the open problem proposed by Landesberg--Lee--Lindenstrauss--Oh \cite[Open problem 1.8]{LLLO_Horospherical}. Similarly, as Corollary \ref{cor:mainAnosov} and Corollary \ref{cor:mainAnosovrephrase} do not have any rank assumption, they resolve Problem \ref{ques:Anosov}, which was proposed by Oh in \cite[Section 8.2]{oh2025dynamics}.

\subsection{Relatively Borel Anosov subgroups}
 
Before discussing the most general version of our measure classification, we state the classification for relatively Borel Anosov subgroups, due to its specific feature distinguished from Anosov subgroups and the most general case.

Relatively Anosov subgroups are higher-rank generalization of geometrically finite subgroups in rank one. A discrete subgroup $\Ga < \Gsf$ is called \emph{relatively Borel Anosov} if
\begin{itemize}
    \item for any infinite sequence $\{g_n\}_{n \in \Nb} \subset \Ga$ and a simple root $\alpha$,
    $$
    \alpha(\kappa(g_n)) \to + \infty \quad \text{as } n \to + \infty,
    $$
    \item $\Ga$ is a relatively hyperbolic group with respect to some peripheral structure, and
    \item there exists a $\Ga$-equivariant embedding $f : \partial_B \Ga \to \Fc$ such that any distinct $x, y \in \partial_B \Ga$, $f(x)$ and $f(y)$ are in general position, where $\partial_B \Ga$ is the associated Bowditch boundary.
\end{itemize}

For Zariski dense relatively Borel Anosov subgroups, the $\Nsf \Msf$-ergodicity of Burger--Roblin measures was proved by the second author \cite{kim2024conformal}. As a special case of our result, we obtain that those are essentially all possible ergodic measures.

\begin{theorem} \label{thm:mainrelAnosov}
Let $\Ga < \Gsf$ be a Zariski dense relatively Borel Anosov subgroup. If $\mu$ is an $\Nsf \Msf$-invariant ergodic Radon measure on $\E_{\Ga}$, then either
\begin{enumerate}
    \item $\mu$ is a constant multiple of a Burger--Roblin measure of $\Ga$, or
    \item $\mu$ is supported on a closed $\Nsf \Msf$-orbit in $\E_{\Ga}$.
\end{enumerate}
\end{theorem}

The second author proved the $\Nsf$-ergodic decomposition for Burger--Roblin measures in \cite{kim2024conformal}. Using this, we deduce the classification of $\Nsf$-invariant ergodic measures.

\begin{theorem} \label{thm:mainrelAnosovN}
Let $\Ga < \Gsf$ be a Zariski dense relatively Borel Anosov subgroup. If $\mu$ is an $\Nsf$-invariant ergodic Radon measure on $\E_{\Ga}$, then either
\begin{enumerate}
    \item $\mu$ is a constant multiple of $\mu^{\BR}|_{\E}$ for a Burger--Roblin measure $\mu^{\BR}$ of $\Ga$ and $\E \in \Dc_{\Ga}$, or
    \item $\mu$ is supported on a closed $\Nsf \Msf$-orbit in $\E_{\Ga}$.
\end{enumerate}
In particular, when $\E_{\Ga}$ is $\Psf^{\circ}$-minimal, we have that $\mu$ is a constant multiple of $\mu^{\BR}$   in case (1) above.
\end{theorem}

As a relative version of Hitchin subgroups in Example \ref{ex:Hitchin}, there is a notion of cusped Hitchin subgroups, and they are relatively Borel Anosov subgroups of $\PSL(d, \Rb)$, $d \ge 3$. See the work of Canary--Zhang--Zimmer \cite{CZZ_cusped}, and Canary's article \cite{Canary_ICM} for the Proceedings of the ICM 2026 for more details and comprehensive expositions.

\subsection{Measure classification in the full generality}

In the rest of the introduction, we present the most general version of our horospherical invariant measure classification result.

\subsubsection{Horospherical foliations}
Let $\Delta$ be the set of all simple roots associated to $\fa^+$. For a non-empty subset $\theta \subset \Delta$, let $\Psf_{\theta} < \Gsf$ be the standard parabolic subgroup corresponding to $\theta$, i.e., $\Psf_{\theta}$ is generated by $\Msf \Asf$ and all root subgroups for all positive roots and any negative root which is a $\Zb$-linear combination of $\Delta \smallsetminus \theta$. Let $\Nsf_{\theta} < \Psf_{\theta}$ be the unipotent radical of $\Psf_{\theta}$ and let $\Lsf_{\theta} < \Psf_{\theta}$ be a Levi subgroup with the Levi decomposition $\Psf_{\theta} = \Lsf_{\theta} \Nsf_{\theta}$. Then $\Lsf_{\theta} = (\exp \fa_{\theta}) \Ssf_{\theta}$, where $\fa_{\theta} := \bigcap_{\alpha \in \Delta \smallsetminus \theta} \ker \alpha$ and $\Ssf_{\theta}$ is an almost direct product of a connected semisimple real algebraic subgroup and a compact subgroup. 

Using them, we define the \emph{$\theta$-horospherical foliation} of $\Gsf$ by
$$
\Hor_{\theta} := \Gsf / \Nsf_{\theta} \Ssf_{\theta}
$$ 
so that each element of $\Hor_{\theta}$ corresponds to a single $\Nsf_{\theta} \Ssf_{\theta}$-orbit in $\Gsf$. Then $\Gsf$ acts on $\Hor_{\theta}$ by left multiplication.

When $\theta = \Delta$, we have $\Psf_{\Delta} = \Psf$, $\Nsf_{\Delta} = \Nsf$, and $\Ssf_{\Delta} = \Msf$. In particular, in this case, the $\Ga$-action on $\Hor_{\Delta}$  for a discrete $\Ga < \Gsf$ is dual to the $\Nsf \Msf$-action on $\Ga \ba \Gsf$.  Hence, the $\theta$-horospherical foliation of $\Gsf$ generalizes the picture of the horospherical action on homogeneous spaces.

\subsubsection{Burger--Roblin measures}

We now define Burger--Roblin measures on $\Hor_{\theta}$, employing the following identification of $\Hor_{\theta}$. Let 
$$
\Fc_{\theta} := \Gsf / \Psf_{\theta}
$$
be the $\theta$-boundary of $\Gsf$.
Then there exists a homeomorphism
$$
\Hor_{\theta} \to \Fc_{\theta} \times \fa_{\theta}
$$
so that the $\Gsf$-action on $\Hor_{\theta}$ descends to the $\Gsf$-action on $\Fc_{\theta} \times \fa_{\theta}$ given by
$$
g \cdot (x, u) := (gx, u + \sigma_{\theta}(g, x)) \quad \text{for } g \in \Gsf, (x, u) \in \Fc_{\theta} \times \fa_{\theta},
$$
where $\sigma_{\theta}$ is the partial Iwasawa cocycle. See Section \ref{subsec:Iwasawa} and Section \ref{subsec:horosphericalfol} for details.

Let $\Ga < \Gsf$ be a discrete subgroup. For $\delta \ge 0$ and a linear form $\psi \in \fa_{\theta}^*$, a $\delta$-dimensional $\psi$-Patterson--Sullivan measure $\nu$ of $\Ga$ on $\Fc_{\theta}$ is defined similarly as in Equation \eqref{eqn:PSmeasintro} (see Eequation \eqref{eqn:PSdefbody} for the precise definition). Then the associated \emph{Burger--Roblin measure} on $\Hor_{\theta}$ is the $\Ga$-invariant measure defined by
$$
d \mu_{\nu}^{\BR} (x, u) : = e^{\delta \cdot \psi(u)} d \nu(x) du
$$
abusing notations, using the identification $\Hor_{\theta} = \Fc_{\theta} \times \fa_{\theta}$. Note also that $\Asf_{\theta}$ acts on $\Hor_{\theta} = \Gsf / \Nsf_{\theta} \Ssf_{\theta}$ by right multiplication. This descends to the translation action on the $\fa_{\theta}$-component of $\Fc_{\theta} \times \fa_{\theta}$, and the measure $\mu_{\nu}^{\BR}$ is quasi-invariant under this action.

In our discussion, Burger--Roblin measures associated to Patterson--Sullivan measure of divergence type are main objects.
The critical exponent of $\psi \in \fa_{\theta}^*$ is defined by
$\delta_{\psi}(\Ga) := \inf \{ s > 0 : \sum_{g \in \Ga} e^{-s \psi(\kappa(g))} < + \infty\}$
where we identify $\fa_{\theta}^*$ with a subspace of $\fa^*$, by precomposing the orthogonal projection $\fa \to \fa_{\theta}$. We call a $\delta$-dimensional $\psi$-Patterson--Sullivan measure $\nu$ on $\Fc_{\theta}$ \emph{divergence-type} if $\delta = \delta_{\psi}(\Ga)$ and $\sum_{g \in \Ga} e^{-\delta_{\psi}(\Ga) \psi(\kappa(g))} = + \infty$.

\subsubsection{Hypertransverse subgroups}

A class of discrete subgroups we consider in this paper is called hypertransverse subgroups, introduced in \cite{kim2024conformal}. This notion extends rank one discrete subgroups and includes Anosov subgroups, relatively Anosov subgroups, and their subgroups. In addition, hypertransverse subgroups are special types of transverse subgroups, which are introduced and studied by Canary--Zhang--Zimmer \cite{CZZ_transverse}.

We denote by $\opp : \Delta \to \Delta$ the opposition involution, i.e., $\alpha(\kappa(g^{-1})) = \opp(\alpha)(\kappa(g))$ for all $g \in \Gsf$ and $\alpha \in \Delta$. A Zariski dense discrete subgroup $\Ga < \Gsf$ is called \emph{$\Psf_{\theta}$-transverse} if
\begin{itemize}
    \item for any infinite sequence $\{g_n\}_{n \in \Nb} \subset \Ga$ and $\alpha \in \theta$,
    $$\alpha(\kappa(g_n)) \to + \infty \quad \text{as } n \to + \infty,$$
    and
    \item any distinct $x, y \in \La_{\theta \cup \opp(\theta)}(\Ga) \subset \Fc_{\theta \cup \opp(\theta)}$ are in general position, i.e., the diagonal $\Gsf$-orbit $\Gsf \cdot (x, y)$ is open in $\Fc_{\theta \cup \opp(\theta)} \times \Fc_{\theta \cup \opp(\theta)}$. 
\end{itemize}
A $\Psf_{\theta}$-transverse $\Ga < \Gsf$ is called \emph{$\Psf_{\theta}$-hypertranseverse} if we further have that
\begin{itemize}
    \item there exists a proper geodesic Gromov hyperbolic space $Z$ on which $\Ga$ acts properly discontinuously by isometries, and
    \item there exists a $\Ga$-equivariant homeomorphsim $\La_Z(\Ga) \to \La_{\theta}(\Ga)$ where $\La_Z(\Ga) \subset \partial Z$ is the limit set of $\Ga$.
\end{itemize}

When $\Ga$ is a hyperbolic group and $Z$ can be taken as the Cayley graph of $\Ga$, $\Ga$ is called \emph{$\Psf_{\theta}$-Anosov}. When $\Ga$ is a relatively hyperbolic group and $Z$ can be taken as the Gromov model of $\Ga$, $\Ga$ is called \emph{relatively $\Psf_{\theta}$-Anosov}.

An important property of $\Psf_{\theta}$-hypertransverse $\Ga$ is that the $\Ga$-action on $\La_{\theta}(\Ga)$ is a convergence action, which also holds for transeverse subgroups \cite{KLP_Anosov}. Hence, the notion of \emph{conical limit set} $\La_{\theta}^{\rm con}(\Ga) \subset \La_{\theta}(\Ga)$ can be defined in terms of the convergence action. If $\Ga$ is $\Psf_{\theta}$-Anosov, then $\La_{\theta}^{\rm con}(\Ga) = \La_{\theta}(\Ga)$. If $\Ga$ is relatively $\Psf_{\theta}$-Anosov, then $\La_{\theta}(\Ga)$ is a disjoint union of $\La_{\theta}^{\rm con}(\Ga)$ and its \emph{parabolic limit set}, which we denote by $\La_{\theta}^{\rm p}(\Ga) \subset \La_{\theta}(\Ga)$.

\subsubsection{Measure classification}

For a Zariski dense $\Psf_{\theta}$-hypertransverse $\Ga < \Gsf$, we define the \emph{recurrence locus} by 
$$
\Rc_{\Ga, \theta} := \left\{ [g] \in \Hor_{\theta} : 
\begin{matrix}
    \exists \text{ infinite sequences } \{g_n\}_{n \in \Nb} \subset \Ga, \{u_n\}_{n \in \Nb} \subset \fa^+ \cap \fa_{\theta} \\
     \text{s.t. the sequence } g_n [g] \exp u_n \in \Hor_{\theta} \text{ is precompact}.
\end{matrix}
\right\}
$$
Under the identification $\Hor_{\theta} = \Fc_{\theta} \times \fa_{\theta}$, we have
$$
\Rc_{\Ga, \theta} = \La_{\theta}^{\rm con}(\Ga) \times \fa_{\theta}
$$
(see \cite[Lemma 5.5]{KOW_SF}), and when $\theta = \Delta$, the recurrence locus $\Rc_{\Ga, \Delta}$  corresponds to the set of elements in $\Ga \ba \Gsf$ whose $\exp \fa^+$-orbits are recurrent to compact subsets, which is also a subset of $\E_{\Ga}$. Note that in this sense, $\Rc_{\Ga, \Delta}$  contains all directional recurrent set $\Rc_{\Ga, v}$, $v \in \inte \fa^+$, considered in \cite{LLLO_Horospherical}.

We now state the most general version of our measure classification result.
By the higher-rank Hopf--Tsuji--Sullivan dichotomy for transverse subgroups (\cite{CZZ_transverse}, \cite{KOW_PD}), the Burger--Roblin measure $\mu_{\nu}^{\BR}$ is supported on $\mathcal{R}_{\Ga, \theta}$ for a divergence-type Patterson--Sullivan measure $\nu$ of $\Ga$ on $\Fc_{\theta}$. Moreover, in this case, the second author showed the $\Ga$-ergodicity of $\mu_{\nu}^{\BR}$ in \cite{kim2024conformal}. It turns out that they are the only ergodic measures.

\begin{theorem} \label{thm:maintransverse}
    Let $\Ga < \Gsf$ be a Zariski dense $\Psf_{\theta}$-hypertransverse subgroup. Then the following sets are the same, up to constant multiples:
    \begin{enumerate}
        \item $\{ \mu_{\nu}^{\BR} : \nu \text{ a divergence-type Patterson--Sullivan measure of $\Ga$ on $\Fc_{\theta}$}\}$.
        \item the set of all $\Ga$-invariant ergodic Radon measures on $\mathcal{R}_{\Ga, \theta}$.
    \end{enumerate}
\end{theorem}

See Theorem \ref{thm:uniqueRadon} for the version without Zariski density, but with the non-arithmeticity of the Jordan spectrum, which is due to the work of Benoist \cite{Benoist2000proprietes}. As a byproduct of our method, we also obtain a strengthened version of the Hopf--Tsuji--Sullivan dichotomy (Corollary \ref{cor:pattersonSqueeze}). See also Theorem \ref{thm:trbytrlengthqi} for normal subgroups.

The subset of $\Hor_{\theta}$ corresponding to the $\Psf$-minimal set $\E_{\Ga}\subset \Ga \ba \Gsf$ defined above is
$$
\E_{\Ga, \theta} := \{ [g] \in \Hor_{\theta} : g \Psf_{\theta} \in \La_{\theta}(\Ga) \}.
$$
We finally state a corollary of  Theorem \ref{thm:maintransverse} for relatively $\Psf_{\theta}$-Anosov subgroups that we classify all horospherical invariant measures $\E_{\Ga, \theta}$.

\begin{corollary} \label{cor:mainrelAnosov}
    Let $\Ga < \Gsf$ be a Zariski dense relatively $\Psf_{\theta}$-Anosov subgroup. If $\mu$ is a $\Ga$-invariant ergodic Radon measure on $\E_{\Ga, \theta} \subset \Hor_{\theta}$, then either 
    \begin{enumerate}
        \item $\mu$ is supported on $\mathcal{R}_{\Ga, \theta}$ and is a constant multiple of a Burger--Roblin measure, or 
        \item $\mu$ is supported on a closed $\Ga$-orbit $\Ga \cdot \xi$ in $\Hor_{\theta}$, for some $\xi \in \E_{\Ga, \theta} \smallsetminus \mathcal{R}_{\Ga, \theta}$.
    \end{enumerate}
\end{corollary}

\begin{remark}
We remark that the ergodic decomposition in \cite{kim2024conformal} is in fact for $\Psf_{\Delta}$-hypertransverse subgroups. Hence, in this case, we also have the same classification of $\Nsf$-invariant ergodic Radon measures on the subset of $\E_{\Ga}$ corresponding to $\Rc_{\Ga, \Delta} \subset \Hc_{\Delta}$, as in Theorem \ref{thm:mainAnosovN} and Theorem \ref{thm:mainrelAnosovN}. See Theorem \ref{thm:hypertransverseNergodic}.
\end{remark}

\subsection{On the proof}

The main part of the proof in this paper is to show that any $\Ga$-invariant ergodic Radon measure on $\Rc_{\Ga, \theta} \subset \Hor_{\theta}$ is quasi-invariant under the $\exp \fa_{\theta}$-action. It is now well-known that such a quasi-invariance implies that the measure should be of the same form as Burger--Roblin measures (\cite[0.1 Basic Lemma]{aaronson2002invariant}, \cite[Lemma 1]{Sarig_abelian}, \cite[Proposition 10.25]{LO_invariant}).

Our proof is rather geometric, and does not make use of any continuous flow on $\Ga \ba G$, such as one-dimensional diagonal flows given by $v \in \inte \fa^+$ or multi-dimensional action of $\exp \fa^+$, not relying on ergodic theorems. This can be compared with the work of Landesberg--Lee--Lindenstrauss--Oh \cite{LLLO_Horospherical} where the open problem in Problem \ref{ques:LLLO} was proposed, and the dynamics of one-dimensional diagonal flow $\exp (\Rb_{> 0} v)$ played a role, and hence the measure under the consideration is restricted to the smaller subset $\Rc_{\Ga, v}$. This is one difference that enables us to classify horospherical invariant measures without restricting the supports of measures to smaller subsets.

We also do not rely on the existence of Besicovitch-type covering, which was used in \cite{LLLO_Horospherical} (and other previous literatures) based on their assumption that $\Gsf$ is a product of rank one Lie groups. This is another difference which allows us to handle arbitrary semisimple real algebraic group $\Gsf$.

Our proof of the measure classification is in a similar spirit to our recent work \cite{CK_ML}, where we considered the action of subgroups of the mapping class group on the space of measured laminations on a surface. On the other hand, there are essential difficulties in higher-rank homogeneous settings, due to their geometric differences from Teichm\"uller spaces of surfaces, and to cocycles being vector-valued.

In both settings, the proof of the quasi-invariance goes through showing that the measure is quasi-invariant under the translations by Jordan projections of loxodromic elements in this paper (Theorem \ref{thm:trbytrlengthqi}), and by Teichm\"uller translation lengths of pseudo-Anosov elements in \cite{CK_ML}. These are eventually done by carefully approximating cocycles by Jordan projections and translation lengths. Major difficulties arise at this step.

In \cite{CK_ML}, it was crucial that the axis of a pseudo-Anosov mapping class in the Teichm\"uller space possesses a $\CAT(-1)$-property, which we called \emph{squeezing}. More precisely, this means that the farther the projections of any two points in the Teichm\"uller space to the axis, the closer the geodesic segment between the two points to the axis. The squeezing property can be regarded as an effective/quantitative version of so-called contracting property, and it was deduced from Minsky's contraction theorem \cite{minsky1996quasi-projections} in \cite{CK_ML}. Such an effective/quantitative property enabled us to do appropriate approximations of cocycles by Teichm\"uller translation lengths.

In contrast, having such squeezing properties in higher-rank homogeneous settings may be impossible in general, due to the presence of higher-rank flats. Indeed, in the teaser paper \cite{CK_pradak} where $\Gsf$ is a product of rank one groups, the fact that each component is $\CAT(-1)$ was helpful, while this is only possible for such a restricted case.

What we have in our setting is only Gromov hyperbolicity associated to $\Ga$, which is much weaker than $\CAT(-1)$. The Gromov hyperbolicity only gives contracting property of geodesics, an ineffective/coarse version of squeezing. That is, it only says that if two points have sufficiently far projection to a geodesic, then a geodesic segment connecting those two points intersects some neighborhood of the geodesic. See Figure \ref{figure.contractingsqueezing} for the comparison.

\begin{figure}[h]
\begin{tikzpicture}[scale=0.45]
    \draw[thick] (-4, 0) -- (4, 0);
    \filldraw (-4, 0) circle(2pt);
    \filldraw (4, 0) circle(2pt);


    \draw[dashed, teal, thick] (-6, 2) .. controls (-5, 2) and (-3.5, 1) .. (-3.5, 0);
    \draw[dashed, teal, thick] (6.5, 3) .. controls (5, 3) and (3.5, 1) .. (3.5, 0);

    \filldraw[teal] (-3.5, 0) circle(2pt);
    \filldraw[teal] (3.5, 0) circle(2pt);

    \draw[red, thick] (-6, 2) .. controls (-3.5, 2)  and (-3.3, 0.3) .. (-2.7, 0.3) .. controls (-2.5, 0.2) and (2.5, 0.2) .. (2.8, 0.5) .. controls (3.3, 0.7) and (3.5, 3) .. (6.5, 3);

    \filldraw (-6, 2) circle(2pt);
    \filldraw (6.5, 3) circle(2pt); 

    \filldraw (0, 0) circle(2pt);
    \filldraw[color=blue!80, opacity=0.3] (0, 0) circle(0.5);

    \draw[blue] (0, 0.8) node {$\epsilon$};

    \draw[<->, teal, thick] (-3.5, -0.25) -- (0, -0.25);
        \draw[<->, teal, thick] (3.5, -0.25) -- (0, -0.25);

    \draw (2, -0.1) node[below] {\color{teal} \tiny $\ge K(\epsilon)$};

    \draw (4, 0) node[right] {\tiny $\ga$}; 

\end{tikzpicture}
\qquad
\begin{tikzpicture}[scale=0.45]
    \draw[thick] (-4, 0) -- (4, 0);
    \filldraw (-4, 0) circle(2pt);
    \filldraw (4, 0) circle(2pt);


    \draw[dashed, teal, thick] (-6, 2) .. controls (-5, 2) and (-3.5, 1) .. (-3.5, 0);
    \draw[dashed, teal, thick] (6.5, 3) .. controls (5, 3) and (3.5, 1) .. (3.5, 0);

    \filldraw[teal] (-3.5, 0) circle(2pt);
    \filldraw[teal] (3.5, 0) circle(2pt);

    \filldraw[teal] (-1, 0) circle(2pt);
    \filldraw[teal] (1, 0) circle(2pt);

    \draw[red, thick] (-6, 2) .. controls (-5, 2)  and (-3.5, 1.5) .. (-3.3, 0.9) .. controls (-3.1, 0.3) and (-2.8, 0.9) .. (-2.5, 0.9) .. controls (-2.4, 0.9) and (-2.2, 0.7) .. (-2.1, 0.6) .. controls (-2, 0.5) and (-1.8, 0.5) .. (-1.7, 0.8) .. controls (-1.6, 1.1) and (-1.5, 1.1) .. (-1.4, 0.8) .. controls (-1.3, 0.5) and (-1, 0.5) .. (-0.9, 0.8) .. controls (-0.8, 1.0) and (-0.4, 1.0) .. (-0.1, 0.7) .. controls (0.1, 0.6) and (0.6, 0.6) .. (1, 0.8) .. controls (1.4, 1) and (2, 1) .. (2.4, 0.6) .. controls (2.6, 0.4) and (3, 0.4) .. (3.2, 0.8) .. controls (4.3, 3) and (6, 3) .. (6.5, 3);

    \filldraw (-6, 2) circle(2pt);
    \filldraw (6.5, 3) circle(2pt);

    \filldraw[color=blue!80, opacity=0.3] (-1, 1) arc (90:270:1) -- (1, -1) arc(-90:90:1) -- (-1, 1);

    \draw[<->, teal, thick] (-3.5, -0.25) -- (-1, -0.25);
    \draw[<->, teal, thick] (3.5, -0.25) -- (1, -0.25);

    \draw (2.55, -0.1) node[below] {\color{teal} \tiny $\ge K$};

    \draw (4, 0) node[right] {\tiny $\ga$};

\end{tikzpicture}
\caption{A squeezing geodesic $\ga$ (left) and a contracting geodesic $\ga$ (right)} \label{figure.contractingsqueezing}
\end{figure}
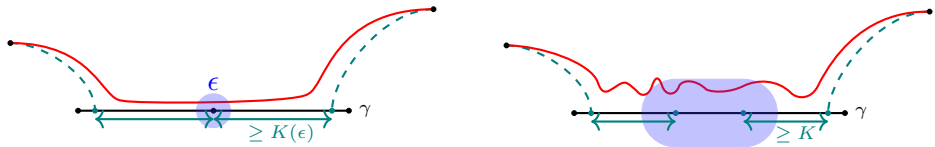

The coarseness of the contracting property does not allow similar approximations of cocycles as in (\cite{CK_ML}, \cite{CK_pradak}) to work. To overcome this obstruction, we establish in this paper a robust connection between alignments of several geodesics in the Gromov hyperbolic space on which $\Ga$ acts and the projective spaces obtained by Tits' representations \cite{Tits_representations}. While this machinery does not produce any squeezing-like property, we could proceed an appropriate approximation of cocycles even for the coarse contracting geodesics, as desired.

The approximation argument above is also based on our observation that almost every point on the limit set is accumulated by synchronically aligned geodesic segments, along translates of an axis of a loxodromic element. We call the set of such points \emph{guided limit set}, which is a subset of the conical limit set. As a byproduct, a part of our proof gives a strengthened version of the Hopf--Tsuji--Sullivan dichotomy that divergence-type Patterson--Sullivan measures are supported on guided limit sets (Corollary \ref{cor:pattersonSqueeze}).

\subsection{Organization}

In Section \ref{section:GH}, we review basic geometry of a Gromov hyperbolic space, and discuss specific properties of alignment of geodesics. Section \ref{section:guided} is devoted to guided limit sets, which are key players in the proof of our measure classification. Horospherical foliations and Iwasawa cocycles are discussed in Section \ref{section:Hor}. In Section \ref{section:Discrete}, we review some dynamical properties of transverse subgroups, and prove an approximation result for Iwasawa cocycles and Jordan projections, with a certain uniform control. Our main measure classification is proved in Section \ref{section:MC}. Statements in the introduction are deduced at the end of Section \ref{section:MC}.

\subsection{Notation}
For $a, b, c \in \Rb$, we write the condition $|a-b| \le c$ by $a=_{c} b$.

\subsection*{Acknowledgements}
The authors would like to thank Dick Canary and Hee Oh for helpful conversations, and Hee Oh for telling us that Theorem \ref{thm:mainAnosovN} follows from Theorem \ref{thm:mainAnosov} when combined with \cite{LO_ergodic}. We also thank Or Landesberg for asking us about normal subgroup version as in Theorem \ref{thm:trbytrlengthqi}. 
This project was the last project started during Kim's Ph.D. study. Kim expresses his special gratitude to his Ph.D. advisor Hee Oh for her encouragement and guidance.

This material is based upon work supported by the National Science Foundation under Grant No. DMS-2424139, while the authors were in residence at the Simons Laufer Mathematical Sciences Institute in Berkeley, California, during the Spring 2026 semester.

Kim thanks KIAS for the hospitality during his visit in 2025 December. Choi was supported by the Mid-Career Researcher Program (RS-2023-00278510) through the National Research Foundation funded by the government of Korea, and by the KIAS individual grant (MG091901) at KIAS.

\section{Geometry of  Gromov hyperbolic spaces}\label{section:GH}

In this section, we review some facts about Gromov hyperbolic spaces. We refer the readers to classical references (\cite{gromov1987hyperbolic}, \cite{coornaert1990geometrie}, \cite{ghys1990bord}, \cite{Bridson1999metric}, etc.) for more comprehensive expositions.

Throughout this section, we fix a proper geodesic metric space $(X, d)$ that  is Gromov hyperbolic. This means that the following holds.
\begin{enumerate}
\item[(Proper)] For any $x \in X$ and $R>0$, the metric ball $\Nc_{R}(x) := \{ z \in X : d(x, z) < R\}$ is precompact.
\item[(Geodesic)] For every pair of points $x, y \in X$, there exists a geodesic denoted by $[x, y]$ connecting $x$ to $y$.
\item[($\delta$-hyperbolic)] There exists $\delta>0$ such that every geodesic triangle in $X$ is $\delta$-thin, i.e., for each $x, y, z \in X$ we have \[
[x, y] \subset \mathcal{N}_{\delta} \big( [x, z] \cup [z, y]\big).
\]
\end{enumerate}
For $w, z \in [x, y]$, we intrinsically assume $[w, z]$ to be a segment of $[x, y]$. Every parametrization of a geodesic is of unit speed.

Let $\partial X$ be the Gromov boundary of $X$. Instead of giving a precise definition, we note that there is a natural topology on $X \sqcup \partial X$ so that it  is a compact metrizable space where $X$ is an open dense subset and $\partial X$ is compact. For any distinct $x, y \in X \cup \partial X$, there exists a geodesic in $X$ with endpoints $x$ and $y$. We also denote by $[x, y] \subset X$ an arbitrarily chosen geodesic connecting $x$ to $y$, and continue the same convention as for geodesics connecting points in $X$.
By enlarging $\delta > 0$ if necessary, every geodesic triangle in $X$ with vertices in $X \cup \partial X$ is also $\delta$-thin.

Given three points $x, y, z \in X$, we define the Gromov product \[
(x | y)_{z} := \frac{1}{2} \left( d(x, z) + d(z, y) - d(x, y) \right).
\]
This definition extends to the case that $x, y \in X \cup \partial X$, by setting
$$
(x | y)_{z} := \sup \liminf_{i, j \to + \infty} (x_i | y_j)_{z}
$$
where the supremum is taken over all sequences $\{x_i\}_{i \in \Nb}, \{ y_j \}_{j \in \Nb} \subset X$ converging to $x, y$, respectively.
The Gromov product $(x | y)_z$ measures the distance from $z$ to $[x, y]$ up to a uniform error, i.e.,
$
\abs{ (x | y)_z - d(z, [x, y])}
$
is bounded by a constant depending only on $\delta$.

By enlarging $\delta>0$ if necessary, we can also guarantee the following Gromov's 4-point inequality: for every $x, y, z, w \in X \cup \partial X$ we have \[
(x | y)_{w} \ge \min \big( (x | z)_{w}, (z | y)_{w} \big) - \delta.
\]

\subsection{Contracting}

Due to the Gromov hyperbolicity of $X$, the nearest-point projection map in $X$ has some usefule geometric properties.
 We note that the discussion below works for a general metric space which is not necessarily Gromov hyperbolic.

For a closed subset $A \subset X$, we denote by $\pi_{A}(\cdot) : X \rightarrow 2^{A}$ the \emph{nearest-point projection}, i.e., 
\[
\pi_{A}(x) := \left\{ a \in A : d(x, a) = \inf_{z \in A} d(x, z)\right\}.
\]

\begin{definition}\label{dfn:contracting}
Let $C \ge 0$. We say that a closed subset $A\subset X$ is \emph{$C$-contracting} if for every geodesic $\gamma \subset X$ with $d(\gamma, A) \ge C$ we have $\diam \pi_{A}(\gamma)  \le C$. We say that $A$ is \emph{(strongly) contracting} if it is $C$-contracting for some $C\ge0$.
\end{definition}

By definition, if $A \subset X$ is $C$-contracting, then $\diam \pi_A(x) \le 2C$ for all $x \in X$. It is also easy to see that if $A \subset X$ is $C$-contracting and $x, y \in X$ are in the $C$-neighborhood of $A$, then $[x, y]$ is contained in the $2.5C$-neighborhood of $A$. Contracting property is relevant to our setting  as every geodesic in the Gromov hyperbolic space $X$ is contracting.

\begin{fact}[{\cite[Proposition 10.2.1]{coornaert1990geometrie}}]\label{fact:deltaHypContracting}
There exists $C = C(\delta) > 0$ such that every geodesic\footnote{More generally, every quasi-convex set is contracting  \cite[Lemma 3.1]{sisto2018contracting}.} in $X$ is $C$-contracting. 
\end{fact}

We now describe some geometric features of contracting geodesics. 
While the constants below can be made uniform depending only on $\delta$ as we will see, we present as follows, to emphasize the power of contracting geodesics in arbitrary metric spaces. We use the  version in \cite{chawla2023genericity} due to its conciseness, but we note that similar results were already observed in (\cite[Proposition 10.2.1]{coornaert1990geometrie}, \cite[Lemma 2.4, 2.5]{sisto2013projections}, \cite[Proposition 2.9, Lemma 2.10, Lemma 2.11]{arzhantseva2015growth}, \cite[Proposition 3.1]{yang2014growth}).

We say that two geodesics $\gamma_1, \gamma_2 \subset X$ are \emph{$K$-equivalent} for $K \ge 0$ if their Hausdorff distance is at most $K$ and if their beginning/ending points are pairwise $K$-close.

\begin{lem}[{\cite[Lemma 2.2]{chawla2023genericity}}]\label{lem:BGIPFellow}
Let $\ga \subset X$ be a $C$-contracting geodesic and $x, y \in X$. Suppose that $\diam \pi_{\ga}([x, y]) > C$. Then there exist points $p, q \in [x, y]$, with $p$ closer to $x$ than $q$, such that
\begin{itemize}
\item  $\pi_{\ga}([x, y])$ and $[p, q]$ are $4C$-equivalent,
\item $\diam(\pi_{\ga}([x, p]) \cup \{p\}) \le 2C$,
\item $\diam(\pi_{\ga}([q, y]) \cup \{q\}) \le 2C$, and
\item for all $x' \in \pi_{\ga}(x)$ and $y' \in \pi_{\ga}(y)$, $[x', y']$  and $[p, q]$ are $10C$-equivalent. 
\end{itemize} 
\end{lem}

The following is immediate from Lemma \ref{lem:BGIPFellow}.

\begin{cor}\label{cor:BGIPFellow}
For a $C$-contracting geodesic $\gamma \subset X$, the following holds.
\begin{enumerate}
\item The map $\pi_{\gamma}(\cdot)$ is $(1, 4C)$-Lipschitz: for each $x, y \in X$,
\[
\diam \pi_{\gamma} (\{x, y\}) \le d(x, y) + 4C.
\]
\item Let $x \in X$ and $\gamma(t) \in \pi_{\gamma}(x)$. Then for every $s \in \R$, we have 
\begin{equation}\label{eqn:approxDist}
d(x, \gamma(s)) =_{4C} d(x, \gamma(t)) + |t-s|.
\end{equation}
\end{enumerate}
\end{cor}

The $\delta$-thinness of geodesic triangles also implies the following:

\begin{fact}\label{fact:CAT(-1)Thin}
Let $M>0$ and let $x, y, x', y' \in X$ be such that $d(x, x'), d(y, y') \le M$. Then the geodesics $[x, y]$ and $[x', y']$ are $(2\delta + M)$-equivalent. The same is true when $x, y$ or $x', y'$ are the same point on $\partial X$.
\end{fact}

\subsection{Alignment}
We now define alignment between geodesics and points. 

\begin{definition}[Alignment]
Let $w, x, y, z \in X$. For a geodesic $[x, y] \subset X$ and $K \ge 0$, we say that the sequence $(w, [x, y])$ is \emph{$K$-aligned} if \[
\diam \pi_{[x, y]}(w) \cup  \{x\}  < K.
\]
We say that the sequence $([x, y], z)$ is \emph{$K$-aligned} if $(z, [y, x])$ is $K$-aligned.

We say that the sequence $(w, [x, y], z)$ is \emph{$K$-aligned} if both sequences $(w, [x, y])$ and $([x, y], z)$ are $K$-aligned. See Figure \ref{fig:alignment}.

\end{definition}

\begin{figure}[h]

\begin{tikzpicture}[scale=0.85]

\draw[very thick] (-2.6, 0) -- (2.6, 0);

\draw[dashed, thick] (-3.5, 3) -- (-1.9, 0.2);
\draw[dashed, thick] (-3.5, 3) -- (-2.4, 0.2);
\draw[thick] (-2.55, 0.3) -- (-2.15, 0) -- (-1.75, 0.3);

\begin{scope}[xscale=-1]

\draw[dashed, thick] (-3.5, 3) -- (-1.9, 0.2);
\draw[dashed, thick] (-3.5, 3) -- (-2.4, 0.2);
\draw[thick] (-2.55, 0.3) -- (-2.15, 0) -- (-1.75, 0.3);
\end{scope}

\draw(-2.9, 0) node {$x$};
\draw(2.9, 0) node {$y$};
\draw (-3.5, 3.3) node {$w$};
\draw (3.5, 3.3) node {$z$};

\fill[opacity=0.1] (-2.6, 0) circle (1);
\fill[opacity=0.1] (2.6, 0) circle (1);
\draw[<->] (-2.6, -1.1) -- (-1.6, -1.1);
\draw (-2.1, -1.3) node {$K$};

\begin{scope}[xscale=-1]
\draw[<->] (-2.6, -1.1) -- (-1.6, -1.1);
\draw (-2.1, -1.3) node {$K$};
\end{scope}
\end{tikzpicture}
\caption{Alignment of geodesics and points.}
\label{fig:alignment}
\end{figure}
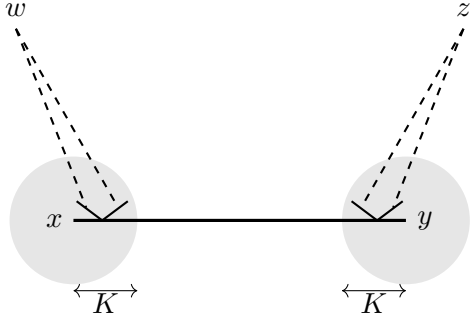

 The following is immediate.
 
\begin{lem}\label{lem:alignDich} 
    Let $\ga \subset X$ be a geodesic of length $L \ge 0$, let $0 \le D \le L$ and let $x \in X$. Then $(\gamma, x)$ is not $D$-aligned or $(x, \gamma)$ is not $(L-D)$-aligned.
\end{lem}

\begin{definition}\label{dfn:alignHoro}

    Let $x \in \partial X$ and $\ga \subset X$ be a compact geodesic. For $K \ge 0$, we say that $(x, \ga)$ is \emph{$K$-aligned} if for every sequence $\{z_{i}\}_{i \in \N} \subset X$ converging to $x$, $(z_{i}, \gamma)$ is $K$-aligned eventually (i.e., for all large $i \in \N$). We define the alignment for $(\ga, x')$ and $(x, \ga, x')$ similarly for $x' \in X \cup \partial X$.
    
\end{definition}

\subsection{Shadows and alignment}

We make a useful elementary observation that the alignment can be interpreted in terms of shadows.

\begin{definition}
    For $x, y \in X$ and $R > 0$, we define the \emph{shadow} $O_R(x, y)$ of a ball of radius $R$ centered at $y$ viewed from $x$, as follows:
    $$
    O_R(x, y) := \{ w \in X \cup \partial X : d([x, w], y) < R \}.
    $$
    In other words, $w \in O_R(x, y)$ if there exists a geodesic $[x, w]$ intersecting the $R$-neighborhood of $y$.
\end{definition}

We now interpret the alignment using shadows. First, note that  one can imagine that if $x, y, z, w \in X$ satisfy
$$
w \in O_R(x, y) \cap O_R(y, z),
$$
then $y$ comes earlier than $z$ along $[x, w]$. Let us make this more precise.

\begin{lemma} \label{lem:shadowandalignment}
    Let $R > 0$ and $x, y, z, w \in X$. Let $C > 0$ be such that $[y, z]$ is $C$-contracting.
\begin{enumerate}
    \item If  $w \in O_R(x, y) \cap O_R(y, z)$, then
    $$
    (x, [y, z], w) \quad \text{is $(2R + 4C)$-aligned.}
    $$

    \item If $
    (x, [y, z], w)$ is $R$-aligned and $d(y, z) > 2R + C$, then 
    $$w \in O_{R + 2C}(x, y)\cap O_{R + 2C}(y, z).$$
\end{enumerate}
\end{lemma}

\begin{proof} 

Let us prove (1). Since $w \in O_R(y, z)$,
$$
d(y, w) \ge d(y, z) + d(z, w) - 2R.
$$
By Corollary \ref{cor:BGIPFellow}(2), if $p \in \pi_{[y, z]}(w)$, then 
$$
\begin{aligned}
    d(y, w) & \le d(y, p) + d(p, w) + 4 C, \\
    d(z, w) & \ge d(z, p) + d(p, w) - 4C.
\end{aligned}
$$
Hence,
$$
d(y, z) + d(z, p) - 2R - 4C \le d(y, p) + 4C,
$$
which implies
$$
d(z, p) \le R + 4C.
$$
Therefore, $([y, z], w)$ is $(R + 4C)$-aligned.

Similarly, since $w \in O_R(x, y)$, we have
$$
d(x, w) \ge d(x, y) + d(y, w) - 2R.
$$
Since $d(x, w) \le d(x, z) + d(z, w)$ and $w \in O_R(y, z)$, we have
$$
d(x, z)  \ge d(x, y) + d(y, z)  - 4R.
$$
If $q \in \pi_{[y, z]}(x)$, then it follows from Corollary \ref{cor:BGIPFellow}(2) that 
$$
\begin{aligned}
d(x, z) \le d(x, q) + d(q, z) + 4C, \\
d(x, y) \ge d(x, q) + d(q, y) - 4C.
\end{aligned}
$$
Hence,
$$
d(q, y) + d(y, z) - 4R - 4C \le d(q, z) + 4C.
$$
This implies
$$
d(q, y) \le 2 R + 4C,
$$
and therefore $(x, [y, z])$ is $(2R + 4C)$-aligned.

We now prove (2). By the assumption,  $\diam \pi_{[y, z]} \left( \{x, w\}\right) > C$. It then follows from Lemma \ref{lem:BGIPFellow} that $d(y, [x, w]) < R + 2C$, and hence $w \in O_{R + 2C}(x, y)$. Since $d(y, z) > R + 2C$, we also have $\diam \pi_{[y, z]} ( \{y, w\} ) > C$. By Lemma \ref{lem:BGIPFellow} again, $d(z, [y, w]) < R + 2C$, and therefore $w \in O_{R + 2C}(y, z)$ as well.
\end{proof}

\subsection{Isometries}
We now turn to isometries of $X$.  The isometries can be classified in terms of their fixed points in $X \cup \partial X$.
A non-trivial isometry $g \in \Isom(X)$ is either \emph{elliptic} (i.e., fixes a point in $X$), \emph{parabolic} (i.e., has a unique fixed point in $\partial X$), or \emph{loxodromic} (i.e., has a unique pair of two fixed points in $\partial X$). If $g \in \Isom(X)$ is of infinite order, it is either parabolic or loxodromic. 

For a loxodromic $g \in \Isom(X)$, its \emph{stable translation length} is defined by
$$
\tau_g := \lim_{n \to +\infty} \frac{d(x, g^n x)}{n} > 0 \quad \text{for any } x \in X.
$$
One can observe that
$$
\tau_{g^k} = |k| \tau_g \le d(x, g^k x) \quad \text{for all } k \in \Zb \text{ and } x \in X.
$$

Let $\ga : \Rb \to X$ be a geodesic such that $\ga(t) \in X$ converges to the attracting and repelling fixed points of $g$ as $t \to + \infty$ and $t \to - \infty$ respectively. Since $g$ fixes its attracting and repelling fixed points, $g \cdot \ga$ is again a geodesic connecting those points. In particular, $\ga$ and $g \cdot \ga$ have Hausdorff distance $2\delta$.

We call $\ga$ an \emph{axis} of $g$, which might not be unique in general. Note that for any $t \in \Rb$,
\begin{equation} \label{eqn:translationaxis}
d( g \cdot \ga(t), \ga(t + \tau_g)) \le c
\end{equation}
for some constant $c = c(\delta) > 0$ that depends only on $\delta$, not on the choice of~$g$.

On the other hand, we often consider points outside its axis. Regarding this, we record useful properties of $\ga$ in the following lemma, whose version for squeezing isometries is proved in \cite[Lemma 5.9]{CK_ML}.

\begin{lem}\label{lem:BGIPHeredi}
Let $g \in \Isom(X)$ be a loxodromic isometry, $\gamma: \R \rightarrow X$ its axis, and $x_{0} \in X$. Then there exists $C=C(g, \gamma, x_{0})>0$ such that the following holds. 
\begin{enumerate}
\item  Any geodesic in $X$ is $C$-contracting. 
\item $d\left(g^{k} x_{0}, \gamma(\tau_{g} k)\right) < C$ for all $k \in \Z$.
\item Let $k \in \N$,  let $x \in X$, and let $K \ge C$. Then 
$$
\left(x, [x_{0}, g^{k} x_{0}]\right) \text{ is not $K$-aligned} \quad \Longrightarrow \quad \pi_{\gamma}(x) \subset \gamma \left( [K-C, +\infty) \right).
$$
\item Let $k\in \N$, let $x \in X$, and let $0 \le K \le \tau_{g} k - C$. Then
$$
\left(x, [x_{0}, g^{k} x_{0}]\right) \text{ is $K$-aligned} \quad \Longrightarrow \quad \pi_{\gamma}(x) \subset \gamma \left( (-\infty,  K + C]\right).
$$
\end{enumerate}
 
Moreover, $C$ can be chosen so that $C(g, \ga, x_0) = C(g^{k}, \ga, x_0)$ for all $k \in \N$ and $C(g^{-1}, \widehat{\ga}, x_0) = C(g, \ga, x_0)$ where $\widehat{\ga}$ is the inversion of $\ga$.
\end{lem}
We often write $C(g) = C(g, \ga, x_0)$ by implicitly choosing its axis $\ga$.

\begin{proof}

Let $C_0  > 0$ be the constant depending only on the Gromov hyperbolicity of $X$ so that any geodesic in $X$ is $C_0$-contracting. Then Item (1) follow when $C \ge C_0$.
For the constant $c > 0$ in Equation \eqref{eqn:translationaxis} which also depends only on the Gromov hyperbolicity, Item (2) holds for $C \ge d(x_0, \ga (0)) + c$.

Now let
$$
D = 10\Big(C_{0} + d\big(x_0, \gamma(0)\big) + \delta + c \Big).$$
We will see that $C = 100 D$ plays the desired role.

\medskip
 We first show Item (3). Let $k \in \N$ and  $K \ge 100D$ and suppose that $(x, [x_{0}, g^{k} x_{0}])$ is not $K$-aligned. Recall that $[x_{0}, g^{k} x_{0}]$ and $\gamma([0, \tau_{g} k])$ are $D$-equivalent. Hence there exist $p \in \pi_{[x_0, g^k x_0]}(x)$ and $t_p \in [0, \tau_g k]$ such that
 $$
d(x_0, p) > K \quad \text{and} \quad d\left(p, \ga(t_p)\right) \le D.
 $$
 Note also that $d(x_0, \ga(0)) \le D$. We then have
 $$
t_p \ge K - 2D.
 $$
 If there exists $t \in (-\infty, K - 6D]$ such that $\ga(t) \in \pi_{\ga}(x)$, then 
 $$
 \begin{aligned}
d(x, \ga(t_p)) & \ge d(x, \ga(K - 6D)) + t_p - (K - 6D) - D \\
& \ge d(x, \ga(K - 6D)) +  3D
 \end{aligned}
 $$
 by  Corollary \ref{cor:BGIPFellow}(2). Since $0 \le K - 6D \le t_p \le \tau_g k$, and since $[x_0, g^k x_0]$ and $\ga([0, \tau_g k])$ are $D$-equivalent, we have
 $$
d(x, p) \ge d(x, \ga(t_p)) - D \ge d(x, \ga(K - 6D)) + 2D \ge d(x, [x_0, g^k x_0]) + D,
 $$
 which contradicts $p \in \pi_{[x_0, g^k x_0]}(x)$. Therefore,
 $$
\pi_{\ga}(x) \subset \ga([K - 6D, + \infty))
 $$
 and hence Item (3) holds for $C = 100 D$.

\medskip

For Item (4), let $k \in \N$ and $0 \le K \le \tau_{g} k - 100D$, and suppose that $(x, [x_{0}, g^{k} x_{0}])$ is $K$-aligned. Then $([x_0, g^k x_0], x)$ is not $(d(x_0, g^k x_0) - K)$-aligned by Lemma \ref{lem:alignDich}. Note that \[
d(x_0, g^{k} x_0) - K \ge \tau_{g} k - K \ge 100D.
\]
As $(g^{-k} x, [x_0, g^{-k} x_0])$ is not $(d(x_0, g^k x_0) - K)$-aligned, we apply Item (3) by reversing the orientation of $\ga$ and get
$$\begin{aligned}
\pi_{\ga}(g^{-k} x) & \subset \ga ( (-\infty, - (d(x_0, g^k x_0) - K) + 6D]) \\
& \subset  \ga ( (-\infty, - \tau_g k + K + 6D]) 
\end{aligned}.
$$
Hence,
$$
\pi_{g^k \ga}( x) \subset  g^k \ga ( (-\infty, - \tau_g k + K + 6D]).
$$
Since $\ga$ and $g^k \ga$ have Hausdorff distance at most $2 \delta$ by Fact \ref{fact:CAT(-1)Thin}, if $y \in \pi_{\ga}(x)$, then there exists $z \in g^k \ga$ such that $d(y, z) \le 2 \delta$.
By Corollary \ref{cor:BGIPFellow}(2),
$$
d(x, z) \ge d(x, \pi_{g^k \ga}(x)) + d(\pi_{g^k \ga}(x), z) - 4C_0,
$$
and hence
$$
d(x, y) + 2 \delta \ge d(x, y) - 2 \delta + d(\pi_{g^k \ga}(x), z) - 4C_0.
$$
We then have 
$$
d(\pi_{g^k \ga}(x), z) \le 4 \delta + 4 C_0,
$$
which implies
$$
z \in  g^k \ga ( (-\infty, - \tau_g k + K + 6D + 4 \delta + 4 C_0]).
$$
Since $d(y, z) \le 2 \delta$,
$$
y \in \ga (( - \infty, K + 6D + 6 \delta + 4 C_0 + c]).
$$

Therefore, Item (4) holds for $C = 100 D$ as well.

\medskip

The ``Moreover'' part is straightforward.
\end{proof}

\subsection{Non-elementary subgroups of isometries}

We denote by $\Isom(X)$ the isometry group of $(X, d)$.
We call a subgroup $\Ga < \Isom(X)$ \emph{discrete} if it acts properly on $X$. The class of subgroups of $\Isom(X)$ we are interested in is as follows:

\begin{definition} \label{def:noneltsubgp}
    A discrete subgroup $\Ga < \Isom(X)$ is called  \emph{non-elementary}~if
    \begin{itemize}
        \item $\Ga$ is not virtually cyclic, and
        \item $\Ga$ contains a loxodromic isometry.
    \end{itemize}
\end{definition}

We can characterize non-elementary subgroups in terms of their limit sets:

\begin{definition}
    Let $\Ga < \Isom(X)$ be a discrete subgroup. Its \emph{limit set} $\La(\Ga) \subset \partial X$ is the set of all accumulation points of $\Ga x \subset X$ on $\partial X$, for any fixed $x \in X$. 
\end{definition}

One can see that $\La(\Ga)$ is compact and $\Ga$-invariant. The $\Ga$-action on $X \cup \partial X$ is a convergence action, and the limit set $\La(\Ga)$ is also the limit set as a convergence group. It is a fact that a discrete subgroup $\Ga < \Isom(X)$ is non-elementary if and only if $\# \La(\Ga) \ge 3$, and in this case the $\Ga$-action on $\La(\Ga)$ is minimal.

We call two loxodromic isometries \emph{independent} if they have disjoint sets of fixed points in $\partial X$. One can see that there are plenty of pairs of independent loxodromic isometries on a non-elementary subgroup.

\begin{lem}\label{lem:indep}
Let $\Gamma < \Isom(X)$ be a non-elementary subgroup. For a loxodromic isometry $g \in \Ga$, there exists $h \in \Ga$ such that $hgh^{-1}$ and $g$ are independent. Moreover,  there are infinitely many pairwise independent loxodromic isometries in $\Ga$.
\end{lem}

The following is a variant of the so-called \emph{extension lemma} of Yang, which can be regarded as the coarse-geometric version of the Anosov closing lemma (cf. \cite[Lemma 3.8]{bowen2008equilibrium}).

\begin{lem}[Extension lemma {\cite[Lemma 1.13]{yang2019statistically}}]\label{lem:extension} \label{lem:extensionHoro}
Let $\Gamma < \Isom(X)$ be a non-elementary subgroup. Then for each loxodromic isometry $\varphi \in \Ga$, there exist $a_{1}, a_{2}, a_{3}\in \Gamma$ and $\beta = \beta(\varphi)>0$  such that  for each $x, y \in X \cup \partial X$, there exists $a\in \{a_{1}, a_{2}, a_{3}\}$ that  makes 
$$
(x, a \cdot [x_{0}, \varphi^{n}x_{0}], a \varphi^{n} a^{-1} \cdot y) \quad \text{$\beta$-aligned for all } n \in \N.
$$
Moreover, $\beta$ can be chosen so that $\beta(\varphi^k) = \beta(\varphi)$ for all $k \in \Z$.
\end{lem}

See also \cite[Lemma 5.12, Lemma 5.15]{CK_ML}. Note that while we stated with $a \varphi^n a$ instead of $a \varphi^n a^{-1}$ in \cite{CK_ML}, the same proof works for $a \varphi^n a^{-1}$.

\subsection{Conical limit sets} 
We define conical limit sets using shadows. Fix a basepoint $x_0 \in X$, while the conical limit sets do not depend on the choice of the basepoint.

\begin{definition} \label{def:conical}
    Let $\Ga < \Isom(X)$ be a discrete subgroup.  A point $x \in \partial X$ is called a \emph{conical limit point} of $\Ga$ if there exist $R > 0$ and an infinite sequence $ \{g_n \}_{n \in \N} \subset \Ga $ such that
    $$
    x \in O_R(x_0, g_n x_0)   \quad \text{for all } n \in \N.
    $$
    We denote the \emph{conical limit set} by $\La^{\rm con}(\Ga) \subset \partial X$.
\end{definition}

The conical limit set $\La^{\rm con}(\Ga)$ is $\Ga$-invariant.

\section{Guided limit sets} \label{section:guided}

We continue the setting of the previous section that $(X, d)$ is a proper geodesic Gromov hyperbolic space, with the same Gromov hyperbolicity constant $\delta > 0$.
In \cite{CK_ML}, we introduced the notion of guided limit sets on the horofunction boundary, which are variants of 
Coulon's contracting limit sets \cite{coulon2024patterson-sullivan} and Yang's $(L, \mathscr{F})$-limit sets \cite{yang2024conformal}.

In this section, we consider the guided limit sets defined on the Gromov boundary and study its properties analogous to the ones we studied in \cite{CK_ML}. They will be a key player in our proof of the main measure classification theorem.

\begin{definition} \label{def:guidedlimitset}
    Let $\Ga < \Isom(X)$ be a non-elementary subgroup. Let $\varphi \in \Ga$ be  a loxodromic isometry and let $C(\varphi) > 0$ be as in Lemma \ref{lem:BGIPHeredi} and fix $K\ge C(\varphi)$. We say that $x \in \partial X$ is a \emph{$(\varphi, K)$-guided limit point} of $\Ga$ 
if for each sufficiently large $n \in \N$, there exists $h \in \Gamma$ such that 
$$
(x_{0}, h[x_{0}, \varphi^{n} x_{0}], x) \quad \text{is $K$-aligned.}
$$
We call the set of $(\varphi, K)$-guided limit points of $\Ga$ the \emph{$(\varphi, K)$-guided limit set} of $\Gamma$. We denote it by $\La^{\varphi, K}(\Ga)$.

\end{definition}

By Lemma \ref{lem:shadowandalignment}, guided limit sets are subsets of conical limit sets.
The role of $K$ in the definition of $(\varphi, K)$-guided limit set is quite flexible. The following was proved in \cite[Lemma 6.4]{CK_ML} in a slightly different setting. We present the proof for completeness.

\begin{lem}\label{lem:squeezedInv}
    Let $\Ga < \Isom(X)$ be a non-elementary subgroup. Let $\varphi \in \Ga$ be  a loxodromic isometry and let $C = C(\varphi) > 0$ be as in Lemma \ref{lem:BGIPHeredi}.  Then for each $K>C$, 
    $$
    \La^{\varphi, K}(\Ga) = \La^{\varphi, C}(\Ga).
    $$
     Moreover, $\La^{\varphi, C}(\Ga)$ is $\Gamma$-invariant.
\end{lem}

\begin{proof} 
Let $\gamma : \R \rightarrow X$ be an axis of $\varphi$ chosen for the constant $C = C(\varphi)$ in Lemma \ref{lem:BGIPHeredi}. Fix $K > C$. We then set \[
N = \lceil (K + 100C + 100 \delta + 100 c)/\tau_{\varphi} \rceil
\]
where  $c > 0$ is the constant given in Equation \eqref{eqn:translationaxis}. 

Now pick an arbitrary $x \in \La^{\varphi, K}(\Ga)$ and let $\{ z_i \}_{i \in \N} \subset X$ be a sequence converging to $x$. Since $x$ is $(\varphi, K)$-guided, for each large enough $n \in \N$ there exists $h \in \Gamma$ such that $(x_0, h[x_{0}, \varphi^{n+2N} x_{0}], z_{i})$ is $K$-aligned for all large $i \in \N$
Since $(n+2N) \tau_\varphi > K +C$, Lemma \ref{lem:BGIPHeredi}(4) tells us that 
$$
\begin{aligned}
\pi_{h \gamma}(x_{0}) &\subset h \gamma \big( (-\infty, K+C] \big) \quad \text{and} \\
\pi_{h \varphi^{n + 2N} \ga}( z_i) & \subset h \varphi^{n + 2N} \ga ([-K - C, + \infty))
\end{aligned}
$$
As in the proof of Lemma \ref{lem:BGIPHeredi}(4), $\pi_{h \ga}(z_i)$ and $\pi_{h \varphi^{n + 2N} \ga}(z_i)$ have Hausdorff distance at most $6 \delta + 4C$.
Hence,
$$
\pi_{h  \ga}( z_i) \subset h  \ga ([(n + 2N) \tau_{\varphi} - c - 6 \delta - 4 C -K - C, + \infty)).
$$

We now show that
$$
(x_0, h \varphi^N [ x_0, \varphi^n x_0], x) \quad \text{is } C\text{-aligned.}
$$
Suppose to the contrary that $(x_{0}, h \varphi^{N} [x_{0}, \varphi^{n} x_{0}])$ is not $C$-aligned. Then by Lemma \ref{lem:BGIPHeredi}(3), we have
$$
\pi_{h \varphi^N \ga}(x_0) \subset h \varphi^N \ga([0, + \infty)).
$$
As above, this implies
$$
\pi_{h \ga}(x_0) \subset h \ga([N \tau_{\varphi} - c - 6 \delta - 4C, + \infty)).
$$
On the other hand, since
$$
K + C < N \tau_{\varphi} - c - 6 \delta - 4C,
$$
this is a contradiction. Therefore, $(x_{0}, h \varphi^{N} [x_{0}, \varphi^{n} x_{0}])$ is $C$-aligned.

Similarly, if $(h \varphi^N [x_0, \varphi^n x_0], z_i)$ is not $C$-aligned, then 
$$
\pi_{h \varphi^{n + N} \ga}(z_i) \subset h \varphi^{n + N} \ga ((-\infty, 0]).
$$
Hence as above,
$$
\pi_{h \ga}(z_i) \subset h \ga ((-\infty, (n + N) \tau_{\varphi} + c + 6\delta + 4C]).
$$
Since 
$$
(n + N) \tau_{\varphi} + c + 6\delta + 4C < (n + 2N) \tau_{\varphi} - c - 6 \delta - 4 C -K - C,
$$
this yields a contradiction, and therefore $(h \varphi^N [x_0, \varphi^n x_0], z_i)$ is  $C$-aligned.
Since this is the case for arbitrary sequence $\{z_{i}\}_{i\in \N} \subset X$ convering to $x$, $(x_{0}, h\varphi^N [x_{0}, \varphi^{n} x_{0}] , x)$ is $C$-aligned. We conclude $x \in \La^{\varphi, C}(\Ga)$, proving the first statement.

\medskip
We now show that $\La^{\varphi, C}(\Ga)$ is  $\Ga$-invariant. Fix  $x \in \La^{\varphi, C}(\Ga)$  and $g \in \Ga$. Then for each sufficiently large $n \in \N$, there exists $h \in \Ga$ such that $(x_0, h [x_0, \varphi^n x_0], x)$ is $C$-aligned. In other words, $(g x_0, g h [x_0, \varphi^n x_0], g x)$ is $C$-aligned. By Corollary \ref{cor:BGIPFellow}(1), this implies that $(x_0, g h [x_0, \varphi^n x_0], g x)$ is $(5C + d(x_0, g x_0))$-aligned. Hence, $g x$ is a $(\varphi, 5C + d(x_0, g x_0))$-guided limit point of $\Ga$, and therefore $g x \in \La^{\varphi, C}(\Ga)$ by the first statement. This shows the desired $\Ga$-invariance.
\end{proof}

 For a non-elementary subgroup $\Ga < \Isom(X)$, isometries $g, \varphi \in \Ga$, constants $C > 0$, and $n \in \N$, we set
$$
U_C ( g; \varphi, n) := \left\{ x \in \La(\Ga) : \textrm{$(x_{0}, g[x_{0}, \varphi^{n} x_{0}], x)$ is $C$-aligned}\right\}.
$$
In \cite[Lemma 7.9]{CK_ML}, we observed that they form a basis for the  topology on the guided limit set, which is defined on the horofunction boundary for a general metric space. We prove the following version of it for the guided limit set defined on the Gromov boundary of $X$.

\begin{lem} \label{lem:nbdBasis}
    Let $\Ga < \Isom(X)$ be a non-elementary subgroup. Let $\varphi \in \Ga$ be loxodromic and  let $C=C(\varphi) > 0$ be as in Lemma \ref{lem:BGIPHeredi}. Then for each  $x \in \La^{\varphi, C}(\Ga)$, for each open set $O \subset \partial X$ with $x \in O$ and for each $N \in \N$, there exist $g \in \Ga$ and  $n > N$ such that
$$
x \in  U_{C}(g; \varphi, n) \subset O.
$$
\end{lem}

As the collection of open sets
$$
\{ y \in \partial X : (x | y)_{x_0} > D \}
$$
for $x \in \partial X$ and $D > 0$ forms a basis of the topology on $\partial X$, Lemma \ref{lem:nbdBasis} is a consequence of the following.

\begin{lemma}
Let $\Ga < \Isom(X)$ be a non-elementary subgroup. Let $\varphi \in \Ga$ be loxodromic and let $C=C(\varphi) > 0$ be as in Lemma \ref{lem:BGIPHeredi}. 
Then for any $ D  > 0$, there exists $N > 0$ such that 
$$
( x | y)_{x_0} \ge D \quad \text{for all } n > N, g \in \Ga, \text{ and } x, y \in U_C(g; \varphi, n).
$$
\end{lemma}

\begin{proof}
    Let $D_0 > 0$ be a variable which will be determined later. Let $N \in \Nb$ be such that $d(x_0, \varphi^n x_0) > 2D_0$ for all $n > N$.
    
    Fix $g \in \Ga$ and $n > N$. Taking $D_0$ large enough, we can apply Lemma \ref{lem:shadowandalignment}, and have some $R  > 0$ depending on $C$ such that 
    $$
    U_C(g; \varphi, n) \subset O_R(x_0, g x_0) \cap O_R( g x_0, g \varphi^{n} x_0).
    $$
    Hence there exists $D' > 0$ depending on $R$ so that for any $x, y \in U_C(g; \varphi, n)$,
    $$
    (x | y)_{x_0} \ge \max \left( d(x_0, g x_0), d(x_0, g \varphi^{n} x_0) \right) - D'.
    $$

    If $d(x_0, g x_0) \ge d(x_0, g \varphi^{n} x_0)$, then 
    $$
    d(x_0, g x_0) \ge d(x_0, \varphi^n x_0) - d(x_0, g x_0) > 2 D_0  - d(x_0, g x_0),$$
     and hence $d(x_0, g x_0) \ge D_0$. This implies
    $$
(x | y)_{x_0} \ge D_0 - D'.
    $$
    Otherwise, 
    $$
    (x | y)_{x_0} \ge  d(x_0, g \varphi^{n} x_0)  - D'.
    $$
    It follows from $O_R(x_0, g x_0) \cap O_R( g x_0, g \varphi^{n} x_0) \neq \emptyset$ that for some $D'' > 0$ depending on $R > 0$, we have
    $$
    d(x_0, g \varphi^{n} x_0) \ge d(x_0, g x_0) + d(x_0, \varphi^{n}x_0) - D'' \ge D_0 - D''.
    $$
    Hence,
    $$
    (x | y)_{x_0} \ge D_0 - D'' - D'.
    $$
    In any case, we can take $D_0 = D + D'' + D'$, and the claim follows.
\end{proof}

\section{Horospherical foliations of Lie groups} \label{section:Hor}

Let $\Gsf$ be a connected semisimple real algebraic group. In this section, we define horospherical foliations of $\Gsf$ and Burger--Roblin measures.

Let $\Psf < \Gsf$ be a minimal parabolic subgroup with a fixed Langlands decomposition $\Psf = \Msf \Asf \Nsf$ where $\Asf$ is a maximal real split torus of $\Gsf$, $\Msf < \Psf$ is the maximal compact subgroup commuting with $\Asf$, and $\Nsf$ is the unipotent radical of $\Psf$. 

Let $\fa := \Lie \Asf$. Fix a positive closed Weyl chamber $\fa^+ \subset \fa$ and set $\Asf^+ := \exp \fa^+$ so that for each $n \in \Nsf$, the set $\{a^{-1} n a 
: a \in \Asf^+ \} \subset \Nsf$  is precompact. Let $\Ksf < \Gsf$ be a maximal compact subgroup such that the Cartan decomposition $\Gsf = \Ksf \Asf^+ \Ksf$ holds.
The \emph{Cartan projection} is the map $\kappa : \Gsf \to \fa^+$ satisfying
$$
g \in \Ksf (\exp \kappa(g)) \Ksf \quad \text{for all } g \in \Gsf.
$$
The bi-$\Ksf$-invariant norm $\norm{\cdot}$ on $\fa$ induced by the Killing form gives a left $\Gsf$-invariant Riemannian metric $d(\cdot, \cdot)$ on $\Gsf$: for $g, h \in \Gsf$, $d(g, h) := \norm{\kappa(g^{-1} h)}$.  This is right $\Ksf$-invariant as well, and hence induces a left $\Gsf$-invariant Riemannian metric on the associated symmetric space $\Gsf / \Ksf$.

Fix an element $w_0 \in \Ksf$ which normalizes $\Asf$ and represents the longest Weyl element so that $\operatorname{Ad}_{w_0} \fa^+ = - \fa^+$. This induces the map
$$
\opp := - \operatorname{Ad}_{w_0} : \fa \to \fa
$$
which is an involution preserving $\fa^+$, called the opposition involution. Then we have $\kappa(g^{-1}) = \opp(\kappa(g))$ for all $g \in \Gsf$.

Let $\Delta$ denote the set of all simple roots associated to $\fa^+$. The opposition involution on $\fa$ induces an involution $\Delta \to \Delta$, $\alpha \mapsto \alpha \circ \opp$. We also denote this involution by $\opp : \Delta \to \Delta$.

For a non-empty subset $\theta \subset \Delta$, we set
$$
\fa_{\theta} := \bigcap_{\alpha \in \Delta \smallsetminus \theta} \ker \alpha \quad \text{and} \quad \fa_{\theta}^+ := \fa_{\theta} \cap \fa^+.
$$
We denote by $p_{\theta} : \fa \to \fa_{\theta}$ the projection invariant under all Weyl elements fixing $\fa_{\theta}$ pointwise. Let $\fa^*$ and $\fa_{\theta}^*$ be spaces of all $\Rb$-linear forms on $\fa$ and $\fa_{\theta}$ respectively. Via $p_{\theta}$, $\fa_{\theta}^*$ can be regarded as the subspace of $\fa^*$ invariant under the precomposition of $p_{\theta}$.

For $g \in \Gsf$, its \emph{Jordan projection} is defined as
$$
\la(g) := \lim_{n \to + \infty} \frac{ \kappa(g^n)}{n} \in \fa^+
$$
and is invariant under conjugations.
We also set
$$
\la_{\theta}(g) := p_{\theta}(\la(g)) \in \fa_{\theta}^+.
$$
For Zariski dense subsemigroups, Benoist proved the non-arithmeticity of Jordan projections.
\begin{theorem}[{\cite{Benoist2000proprietes}}] \label{thm:Benoist_nonarithmetic}
Let $\Hsf \subset \Gsf$ be a Zariski dense subsemigroup. Then $\la(\Hsf) \subset \fa^+$ generates a dense additive subgroup of $\fa$. 
\end{theorem}

\subsection{Furstenberg boundary and $\theta$-boundary}

The \emph{Furstenberg boundary} of $\Gsf$ is defined as the quotient
$$
\Fc := \Gsf / \Psf = \Ksf / \Msf
$$
where the last equality is due to the Iwasawa decomposition $\Gsf = \Ksf \Psf = \Ksf \Asf \Nsf$.

In a similar spirit, each non-empty subset of $\Delta$ gives rise to a boundary of $\Gsf$.
For a non-empty subset $\theta \subset \Delta$, let $\Psf_{\theta} < \Gsf$ be the standard parabolic subgroup corresponding to $\theta$. That is, $\Psf_{\theta}$ is generated by $\Msf \Asf$ and all root subgroups for all positive roots and any negative root which is a $\Zb$-linear combination of $\Delta \smallsetminus \theta$. Then we have $\Psf_{\Delta} = \Psf$ and for $\theta \subset \theta'$, $\Psf_{\theta'} < \Psf_{\theta}$. This is the reason why $\Psf$ is a minimal parabolic subgroup.

Then \emph{$\theta$-boundary} is defined as the quotient
$$
\Fc_{\theta} := \Gsf / \Psf_{\theta}
$$
and the natural projection $\pi_{\theta} : \Fc \to \Fc_{\theta}$, $g \Psf \mapsto g \Psf_{\theta}$, has compact fibers.
Note that $w_0 \Psf_{\opp(\theta)} w_0^{-1}$ is a parabolic subgroup opposite to $\Psf_{\theta}$. Two points $x \in \Fc_{\theta}$ and $y \in \Fc_{\opp(\theta)}$ are said to be  \emph{in general position} if there exists $g \in \Gsf$ such that
$$
x = g \Psf_{\theta} \quad \text{and} \quad y = g w_0 \Psf_{\opp(\theta)}.
$$
Equivalently, a pair $(x, y) \in \Fc_{\theta} \times \Fc_{\opp(\theta)}$ is in general position if its $\Gsf$-orbit $\Gsf \cdot (x, y) \subset \Fc_{\theta} \times \Fc_{\opp(\theta)}$ under the diagonal action is open, which turns out to be a unique open $\Gsf$-orbit.

The convergence of a sequence in $\Gsf$ to $\Fc_{\theta}$ is defined as follows: given a sequence $\{g_n\}_{n \in \Nb} \subset \Gsf$, let $\{k_n\}_{n \in \Nb} \subset \Ksf$ be a sequence such that $g_n \in k_n \Asf^+ \Ksf$ in Cartan decomposition, for each $n \in \Nb$. Then we have the convergence $g_n \to x \in \Fc_{\theta}$ if, as $n \to + \infty$,
\begin{equation} \label{eqn:convergence_G}
k_n \Psf_{\theta} \to x \text{ in } \Fc_{\theta} \quad \text{and} \quad \alpha(\kappa(g_n)) \to +\infty \text{ for each } \alpha \in \theta.
\end{equation}
While $k_n \in \Ksf$ may not be well-defined, the point $k_n \Psf_{\theta} \in \Fc_{\theta}$ is well-defined, independent of the choice of $k_n \in \Ksf$, for each $n \in \Nb$.

Moreover, if
$
g_n = k_n (\exp \kappa(g_n)) \ell_n \in \Ksf \Asf^+ \Ksf
$
in Cartan decomposition for $n \in \Nb$, then
$$
g_n^{-1} = (\ell_n^{-1} w_0) (\exp \opp(\kappa(g_n))) (w_0^{-1} k_n^{-1}) \in \Ksf \Asf^+ \Ksf
$$
is a Cartan decomposition as well. Hence, if $\alpha(\kappa(g_n)) \to + \infty$ for each $\alpha \in \theta$ and $\ell_n \to \ell \in \Ksf$ as $n \to + \infty$, then the sequence $\{ g_n^{-1} \}_{n \in \Nb}$ converges to $\ell^{-1} w_0 \Psf_{\opp(\theta)} \in \Fc_{\opp(\theta)}$.

For $\alpha \in \Delta$, we simply write $\Psf_{\alpha} := \Psf_{\{\alpha\}}$ and $\Fc_{\alpha} := \Fc_{\{\alpha\}}$, and similarly to other places.

\subsection{Iwasawa cocycles} \label{subsec:Iwasawa}
The Iwasawa decomposition is the diffeomorphism
$$
\begin{aligned}
\Ksf \times \Asf \times \Nsf & \to \Gsf \\
(k, a, n) & \mapsto kan.
\end{aligned}
$$
The \emph{Iwasawa cocycle} $\sigma : \Gsf \times \Fc \to \fa$ is defined as follows: for $g \in \Gsf$ and $x \in \Fc$, let $\sigma(g, x) \in \fa$ be the unique element such that
$$
gk \in \Ksf ( \exp \sigma(g, x)) \Nsf
$$
where $k \in \Ksf$ be such that $x = k \Msf \in \Fc$.

For  non-empty $\theta \subset \Delta$, the \emph{partial Iwasawa cocycle} $\sigma_{\theta} : \Gsf \times \Fc_{\theta} \to \fa_{\theta}$ is defined as
$$
\sigma_{\theta}(g, x) := p_{\theta} ( \sigma(g, \tilde x))
$$
for some (any) $\tilde x \in \pi_{\theta}^{-1}(x) \in \Fc$. This does not depend on the choice of $\tilde x$ \cite[Lemma 6.1]{Quint2002Mesures}. Then for $g, h \in \Gsf$ and $x \in \Fc_{\theta}$, the cocycle relation holds:
\begin{equation} \label{eqn:Iwasawa_cocycle}
\sigma_{\theta}(g h, x) = \sigma_{\theta}(g, hx) + \sigma_{\theta}(h, x).
\end{equation}
See \cite[Lemma 6.2]{Quint2002Mesures}. In particular, 
\begin{equation}\label{eqn:Iwasawa_inv}
    \sigma_{\theta}(g^{-1}, x) = - \sigma_{\theta}(g, g^{-1} x).
\end{equation}

\subsection{Horospherical foliations} \label{subsec:horosphericalfol}
We now define the horospherical foliation of $\Gsf$ associated to $\theta \subset \Delta$. Let $\Lsf_{\theta}$ be the centralizer of $\Asf_{\theta} := \exp \fa_{\theta}$. Then $\Lsf_{\theta}$ is a Levi subgroup of $\Psf_{\theta}$ with the Levi decomposition $\Psf_{\theta} = \Lsf_{\theta} \Nsf_{\theta}$, where $\Nsf_{\theta}$ is the unipotent radical of $\Psf_{\theta}$. We also have $\Lsf_{\theta} = \Asf_{\theta} \Ssf_{\theta}$ where $\Ssf_{\theta}$ is an almost direct product of a connected semisimple real algebraic subgroup and a compact subgroup. Then $\Msf_{\theta} := \Ksf \cap \Psf_{\theta}$ is a maximal compact subgroup of $\Ssf_{\theta}$. Note that $\Msf_{\Delta} = \Msf = \Ssf_{\Delta}$ and $\Nsf_{\Delta} = \Nsf$.

The \emph{$\theta$-horospherical foliation} of $\Gsf$ is the foliation whose leaves are right $\Nsf_{\theta}\Ssf_{\theta}$-orbits. Hence, it can be represented as the quotient space $\Gsf / \Nsf_{\theta}\Ssf_{\theta}$, on which $\Gsf$ acts by left multiplications.

Set
$$
\Hor_{\theta} := \Fc_{\theta} \times \fa_{\theta},
$$
and define the $\Gsf$-action on $\Hor_{\theta}$ by
$$
g \cdot (x, u) := (gx, u + \sigma_{\theta}(g, x))
$$
for $g \in \Gsf$ and $(x, u) \in \Hor_{\theta}$.
Then the map
$$
\begin{tikzcd}[%
    ,row sep = 0ex
    ]
\Gsf / \Nsf_{\theta}\Ssf_{\theta} \arrow[r] & \Hor_{\theta} = \Fc_{\theta} \times \fa_{\theta} \\
g \Nsf_{\theta} \Ssf_{\theta} \arrow[r, mapsto] & (g \Psf_{\theta}, \sigma_{\theta}(g, \Psf_{\theta}))
\end{tikzcd}
$$
is a $\Gsf$-equivariant homeomorphism. Therefore, we identify
$$
\Hor_{\theta} = \Fc_{\theta} \times \fa_{\theta} = \Gsf / \Nsf_{\theta} \Ssf_{\theta}
$$
and regard $\Hor_{\theta}$ as the $\theta$-horospherical foliation of $\Gsf$ as well.

Since $\Asf_{\theta}$ normalizes $\Nsf_{\theta} \Ssf_{\theta}$, we have the $\Asf_{\theta}$-action on $\Gsf / \Nsf_{\theta} \Ssf_{\theta}$ given by the right multiplication. This $\Asf_{\theta}$-action commutes with the left $\Gsf$-action, and induces an action on $\Hor_{\theta}$ given by the translation on the $\fa_{\theta}$-component: for $a \in \Asf_{\theta}$ and $(x, u) \in \Hor_{\theta}$,
$$
(x, u) \cdot a := (x, u + \log a).
$$

\begin{remark} \label{rmk:bddIwasawa}
When $\rank \Gsf = 1$, if $\{ g_n \}_{n \in \Nb} \subset \Gsf$ is a sequence such that $\kappa(g_n)$ diverges while $\sigma(g_n, x)$ is bounded for some $x \in \Fc$, then $g_n^{-1} \to x$ in the sense of Equation \eqref{eqn:convergence_G}. This can be easily seen especially when $\Gsf = \PSL(2, \Rb) = \Isom^+ \Hb^2$, by interpreting $\kappa( \cdot)$ as the displacement and $\sigma( \cdot, \cdot)$ as the Busemann function in $\Hb^2$. In particular, if $\{ g_n \}_{n \in \Nb} \subset \Gsf$ is a sequence such that $\kappa(g_n)$ diverges and $g_n^{-1}$ does not converge to $x \in \Fc$, then for any $u \in \fa$, the sequence $g_n \cdot (x, u)$ diverges in $\Hor_{\Delta}$.

On the other hand, this does not hold in higher-rank $\Gsf$ in general, unless $\Gsf$ is a product of rank one groups. For example, consider the case that $\Gsf = \PSL(3, \Rb)$ where $\Psf < \Gsf$ is the subgroup of upper triangular matrices, so that $\Fc$ is the full flag variety $\{ (V_1, V_2) : 0 < V_1 < V_2 < \Rb^3, \dim V_i = i \}$. Then set  $x = \Psf \in \Fc$ which correponds to the flag $( (\Rb, 0, 0), (\Rb, \Rb, 0))$ and let $\{ g_n \}_{n \in \Nb} \subset \Gsf$ be the sequence given by 
$$
g_n := \begin{pmatrix}
1 & n & 0 \\
0 & 1 & 0 \\
0 & 0 & 1
\end{pmatrix} \in \Nsf \quad \text{for all }n \in \Nb.
$$
One can see that $\sigma_{\Delta}(g_n^{\pm}, x) = 0$ for all $n \in \Nb$ and $\alpha(\kappa(g_n^{\pm})) \to + \infty$ for all $\alpha \in \Delta$. 

On the other hand, $\{ g_n \}_{n \in \Nb}$ or $\{ g_n^{-1} \}_{n \in \Nb}$ cannot converge to $x$. Indeed, if both sequences converge to $x$, then for any flag $y \in \Fc$ transverse to $x$, $g_n^{\pm} y \to x$ as $n \to + \infty$. Now taking $y$ to be the flag $((0, 0, \Rb), (0, \Rb, \Rb))$ yields a contradiction.

If one seeks a sequence of loxodromic elements, the sequence
$$
g_n := \begin{pmatrix}
1 & n & 0 \\
0 & 1 & 0 \\
0 & 0 & 1
\end{pmatrix}
\begin{pmatrix}
2 & 0 & 0 \\
0 & 1 & 0 \\
0 & 0 & 1/2
\end{pmatrix}
\begin{pmatrix}
1 & n & 0 \\
0 & 1 & 0 \\
0 & 0 & 1
\end{pmatrix}^{-1}
$$
also works.
\end{remark}

\subsection{Patterson--Sullivan measures and Burger--Roblin measures}
In \cite{Quint2002Mesures}, Quint introduced higher-rank Patterson--Sullivan measures on $\Fc_{\theta}$  using the partial Iwasawa cocycle. For $\delta \ge 0$, $\psi \in \fa_{\theta}^*$, and $\Ga < \Gsf$, a Borel probability measure $\nu$ on $\Fc_{\theta}$ is called a \emph{$\delta$-dimensional $\psi$-Patterson--Sullivan measure} of $\Ga$ if for any $g \in \Ga$,
\begin{equation} \label{eqn:PSdefbody}
\frac{d g_* \nu}{d \nu}(x) = e^{-\delta \cdot \psi(\sigma_{\theta}(g^{-1}, x))} \quad \text{for } \nu\text{-a.e. } x \in \Fc_{\theta}.
\end{equation}

For a $\delta$-dimensional $\psi$-Patterson--Sullivan measure $\nu$ of $\Ga$, the associated \emph{Burger--Roblin measure} $\mu_{\nu}^{\BR}$ of $\Ga$ on $\Hor_{\theta}$ is defined by
$$
d \mu_{\nu}^{\BR}(x, u) := e^{\delta \cdot \psi(u)} d \nu(x) du
$$
where $du$ is the Lebesgue measure on $\fa_{\theta}$. The measure $\mu_{\nu}^{\BR}$ is a $\Ga$-invariant Radon measure on $\Hor_{\theta}$.

In a special case that $\theta = \Delta$, we have $\Nsf_{\Delta} = \Nsf$ and $\Ssf_{\Delta} = \Msf$, and the Burger--Roblin measure of $\Ga$ associated to $\nu$ is also defined on $\Ga \ba \Gsf$. Recalling that $\Fc = \Ksf / \Msf$, we denote by $\widehat{\nu}$ the $\Msf$-invariant lift of $\nu$ to $\Ksf$. Using the Iwasawa decomposition $\Gsf = \Ksf \Asf \Nsf$, we define the measure $\tilde \mu_{\nu}^{\BR}$ on $\Gsf$ as follows: for $g = k (\exp u) n \in \Ksf \Asf \Nsf$,
$$
d \tilde \mu_{\nu}^{\BR}(g) := e^{\delta \cdot \psi(u)} d \widehat{\nu}(k) du dn
$$
where $dn$ is the Haar measure on $\Nsf$. Then the Radon measure $\tilde \mu_{\nu}^{\BR}$ is left $\Ga$-invariant and right $\Nsf \Msf$-invariant, and therefore it induces the $\Nsf\Msf$-invariant Radon measure
\begin{equation} \label{def:BRonhomogeneous}
\mu_{\nu}^{\BR} \quad \text{on} \quad \Ga \ba \Gsf,
\end{equation}
abusing notations, which is also called the \emph{Burger--Roblin measure} of $\Ga$ associated to $\nu$.

\subsection{Tits representations and shadows} \label{section:Tits}

To deduce a property of Iwasawa cocycles, which is an important player in this paper, we employ Tits representations and associated shadows. We refer to (\cite{Quint2002Mesures}, \cite[Section 6]{benoist2016random}) for a comprehensive discussion on this topic.

For each $\alpha \in \Delta$, let $\omega_{\alpha} \in \fa^*$ be the (restricted) fundamental weight associated to $\alpha$. Then for a non-empty subset $\theta \subset \Delta$, $\{ \omega_{\alpha} : \alpha \in \theta \}$ is a basis of $\fa_{\theta}^*$.
The following is a theorem of Tits \cite{Tits_representations}, and representations given below are referred to as \emph{Tits representations}.

\begin{theorem}[{\cite[Theorem 7.2]{Tits_representations}}]
  For each $\alpha\in \Delta$, there exists an irreducible $\Rb$-representation $\rho_{\alpha} : \Gsf \to \GL ( V_{\alpha}) $ whose  highest (restricted) weight 
    $\chi_{\alpha}$ is equal to $k_\alpha \omega_\alpha$
    for some $k_{\alpha} \in \Nb$ and whose highest weight space is one-dimensional.
\end{theorem}

For each $\alpha \in \Delta$, we fix a Tits representation $\rho_{\alpha} : \Gsf \to \GL(V_{\alpha})$ throughout the paper. We denote by $V_{\alpha}^+$ the highest weight space of $\rho_{\alpha}$ which is one-dimensional, and by $V_{\alpha}^{<}$ its unique complementary $\rho_{\alpha}(\Asf)$-invariant subspace in $V_{\alpha}$.
Let $\langle \cdot, \cdot \rangle_{\alpha}$ be a $\rho_{\alpha}(\Ksf)$-invariant inner product on $V_{\alpha}$ with respect to which $\rho_{\alpha}(\Asf)$ is symmetric, so that $V_{\alpha}^+$ is perpendicular to $V_{\alpha}^{<}$. We denote by $\norm{\cdot}_{\alpha}$ the norm on $V_{\alpha}$ induced by $\langle \cdot, \cdot \rangle_{\alpha}$, and use the same notation $\norm{\cdot}_{\alpha}$ for the operator norm of elements of $\GL(V_{\alpha})$. Then we have
$$
\log \norm{\rho_{\alpha}(g)}_{\alpha} = \chi_{\alpha}(\kappa(g)) \quad \text{for all } g \in \Gsf.
$$
The norm $\norm{\cdot}_{\alpha}$ induces a metric on the projective space $\Pb(V_{\alpha})$: for $v, w \in V_{\alpha}$ with $\norm{v}_{\alpha} = \norm{w}_{\alpha} = 1$,
$$
d_{\Pb(V_{\alpha})}([v], [w]) := \min \left\{ \norm{v - w}_{\alpha},  \norm{v + w}_{\alpha} \right\}.
$$

Through $\rho_{\alpha}$, the stabilizers of $V_{\alpha}^+, V_{\alpha}^{<} \subset V_{\alpha}$ in $\Gsf$ are parabolic subgroups $\Psf_{\alpha}, w_0 \Psf_{\opp(\alpha)} w_0^{-1} < \Gsf$, respectively. Hence, denoting by $V_{\alpha}^*$ the dual of $V_{\alpha}$ (i.e., $\Pb(V_{\alpha}^*)$ consists of hyperplanes in $V_{\alpha}$),  maps $\Gsf \to \Pb(V_{\alpha})$, $g \mapsto \rho_{\alpha}(g) V_{\alpha}^+$, and $\Gsf \to \Pb(V_{\alpha}^*)$, $g \mapsto \rho_{\alpha}(g) V_{\alpha}^<$, induce $\rho_{\alpha}$-equivariant embeddings:
$$
\begin{tikzcd}[%
    ,row sep = 0ex
    ]
\Phi_{\alpha} : \Fc_{\alpha} \arrow[r] & \Pb(V_{\alpha}) & [-2em] \text{and} & [-2em] \Phi_{\alpha}^* :  \Fc_{\opp(\alpha)} \arrow[r] & \Pb(V_{\alpha}^*)\\
\quad \quad g \Psf_{\alpha} \arrow[r, mapsto] &  \rho_{\alpha}(g) V_{\alpha}^+ & & \quad \quad g w_0 \Psf_{\opp(\alpha)} \arrow[r, mapsto] &  \rho_{\alpha}(g) V_{\alpha}^<
\end{tikzcd}
$$
The induced maps have the property that $x \in \Fc_{\alpha}$ and $y \in \Fc_{\opp(\alpha)}$ are in general position if and only if the line $\Phi_{\alpha}(x)$ and hyperplane $\Phi_{\alpha}^*(y)$ are transverse as subspaces of $V_{\alpha}$, i.e.,
$$
\Phi_{\alpha}(x) \oplus \Phi_{\alpha}^*(y) = V_{\alpha}.
$$
See \cite[Section 3.2]{GGKW2017} for details.

Now more generally, fix a non-empty subset $\theta \subset \Delta$. Similar to the above, the map $\Gsf \to \prod_{\alpha \in \theta} \Pb(V_{\alpha})$, $g \mapsto (\rho_{\alpha}(g)V_{\alpha}^+)_{\alpha \in \theta}$, factors through a $(\rho_{\alpha})_{\alpha \in \theta}$-equivariant embedding
$$
\begin{tikzcd}
\Phi_{\theta} : \F_{\theta} \arrow[r] & \prod_{\alpha \in \theta} \Pb(V_{\alpha}).
\end{tikzcd}
$$
Then the metric on $\prod_{\alpha \in \theta} \Pb(V_{\alpha})$ induced by $d_{\Pb(V_{\alpha})}$'s equips $\Fc_{\theta}$ with a $\Ksf$-invariant metric via the above map.
Note that the following diagram commutes for each $\alpha \in \theta$:
$$
\begin{tikzcd}
\Fc_{\theta} \arrow[r, "\Phi_{\theta}"] \arrow[d] &  \prod\limits_{\alpha' \in \theta} \Pb(V_{\alpha'}) \arrow[d]\\
\Fc_{\alpha} \arrow[r, "\Phi_{\alpha}"'] & \Pb(V_{\alpha})
\end{tikzcd}
$$
where the vertical maps are canonical projections. For this reason, we also denote by $\Phi_{\alpha}(x)$ the $\alpha$-component of $\Phi_{\theta}(x)$ for $x \in \Fc_{\theta}$.

Tits representations are useful to study Iwasawa cocycles, as in the following lemma due to Quint.

\begin{lemma}[{\cite[Lemma 6.4]{Quint2002Mesures}}] \label{lem:Iwasawa_log}
    Let $\alpha \in \theta$. For any $g \in \Gsf$ and $x \in \Fc_{\theta}$, if $v \in \Phi_{\alpha}(x) \smallsetminus \{0\}$, then
    $$
    \chi_{\alpha}(\sigma_{\theta}(g, x)) = \log \frac{ \norm{\rho_{\alpha}(g) v}_{\alpha}}{{\norm{v}_{\alpha}}}.
    $$
\end{lemma}

Quint defined shadows in $\Fc_{\theta}$ using Tits representations as follows: for $g \in \Gsf$, write its Cartan decomposition
$$
g = k_g (\exp \kappa(g)) \ell_g \in \Ksf \Asf^+ \Ksf.
$$
Note that while $k_g, \ell_g \in \Ksf$ may not be well-defined, the corresponding points $k_g \Psf_{\theta} \in \Fc_{\theta}$ and $\ell_{g}^{-1} w_0 \Psf_{\opp(\theta)} \in \Fc_{\opp(\theta)}$ are well-defined. For $0 < \epsilon \le 1$, let
$$
B_{\theta, g}^{\epsilon} := \left\{ x \in \Fc_{\theta} : d_{\Pb(V_{\alpha})} \left( \Phi_{\alpha}(x), \Pb \left(\rho_{\alpha} \left(\ell_g^{-1} \right) V_{\alpha}^{<} \right) \right) \ge \epsilon \quad \text{for all } \alpha \in \theta \right\}.
$$
In other words, $B_{\theta, g}^{\epsilon}$ consists of points $x \in \Fc_{\theta}$ such that the line $\Phi_{\alpha}(x) \subset V_{\alpha}$ has  uniform transversality to the hyperplane $\Phi_{\alpha}^* \left(\ell_g^{-1} w_0 \Psf_{\opp(\alpha)} \right) \subset V_{\alpha}$ for all $\alpha \in \theta$.

For $\epsilon > 0$ and $ g \in \Gsf$,  the \emph{$\epsilon$-shadow of $g$} is defined as the set
$$
g \cdot B_{\theta, g}^{\epsilon} \subset \Fc_{\theta}.
$$
Within this shadow, Iwasawa cocycles behave nicely due to the following estimate.

\begin{lemma}[{\cite[Lemma 4.2]{Quint2002Mesures}}] \label{lem:norminshadow}
    Let $0 < \epsilon \le 1$ and $g \in \Gsf$.  For any $x \in B_{\theta, g}^{\epsilon}$ and $\alpha \in \theta$, if $v \in \Phi_{\alpha}(x) \smallsetminus \{0\}$, then
    $$
    \frac{ \norm{\rho_{\alpha}(g)v}_{\alpha}}{\norm{v}_{\alpha}} \ge \epsilon \cdot \norm{\rho_{\alpha}(g)}_{\alpha}.
    $$
\end{lemma}

Indeed, together with Lemma \ref{lem:Iwasawa_log}, this implies the compatibility of Cartan projections and Iwasawa cocycles in shadows.

\begin{lemma}[{\cite[Lemma 6.5]{Quint2002Mesures}}] \label{lem:cartanIwasawa}
    For any $0 < \epsilon \le 1$ and $g \in \Gsf$, 
    if $x \in  B_{\theta, g}^{\epsilon}$, then
    $$
    \chi_{\alpha}(\kappa(g)) + \log \epsilon \le \chi_{\alpha}(\sigma_{\theta}(g, x)) \le \chi_{\alpha}(\kappa(g)) \quad \text{for all } \alpha \in \theta
    $$
\end{lemma}

\subsection{Local regularity of Iwasawa cocycles}

We deduce locally Lipschitz property of Iwasawa cocycles. We emphasize that the Lipschitz constant below is independent of the choice of a group element.

\begin{proposition} \label{prop:locLipIwasawa}
    Let $0 < \epsilon \le 1$ and $g \in \Gsf$. Then for any $x, y \in g \cdot B_{\theta, g}^{\epsilon}$ and $\alpha \in \theta$,
    $$
    \abs{\chi_{\alpha} \left( \sigma_{\theta}(g^{-1}, x) - \sigma_{\theta}(g^{-1}, y) \right)} \le \frac{1}{\epsilon} \cdot d_{\Pb(V_{\alpha})} \left( \Phi_{\alpha} \left(g^{-1} x \right), \Phi_{\alpha} \left(g^{-1} y \right)\right).
    $$
\end{proposition}

\begin{proof}
    By Equation \eqref{eqn:Iwasawa_inv}, it suffices to show that for $x', y' \in B_{\theta, g}^{\epsilon}$ and $\alpha \in \theta$,
    $$
    \abs{\chi_{\alpha} \left( \sigma_{\theta}(g, x') - \sigma_{\theta}(g, y') \right)} \le \frac{1}{\epsilon} \cdot d_{\Pb(V_{\alpha})} \left( \Phi_{\alpha} \left(x' \right), \Phi_{\alpha} \left(y' \right)\right).
    $$
    Fix vectors $v \in \Phi_{\alpha}(x')$ and $w \in \Phi_{\alpha}(y')$ such that $\norm{v}_{\alpha} = \norm{w}_{\alpha} = 1$ and $d_{\Pb(V_{\alpha})} \left( \Phi_{\alpha} \left(x' \right), \Phi_{\alpha} \left(y' \right)\right) = \norm{v - w}_{\alpha}$.

    Now by Lemma \ref{lem:Iwasawa_log},
    $$
    \abs{\chi_{\alpha} \left( \sigma_{\theta}(g, x') - \sigma_{\theta}(g, y') \right)} = \abs{ \log \norm{\rho_{\alpha}(g) v}_{\alpha} - \log \norm{\rho_{\alpha}(g) w}_{\alpha} }.
    $$
    Then by the mean value theorem for $\log : \Rb_{>0} \to \Rb$ and Lemma \ref{lem:norminshadow},
    $$\begin{aligned}
 \abs{ \log \norm{\rho_{\alpha}(g) v}_{\alpha} - \log \norm{\rho_{\alpha}(g) w}_{\alpha} } & \le \frac{\abs{ \norm{\rho_{\alpha}(g) v}_{\alpha} - \norm{\rho_{\alpha}(g) w}_{\alpha} }}{\epsilon \cdot \norm{\rho_{\alpha}(g)}_{\alpha}} \\
        & \le \frac{ \norm{ \rho_{\alpha}(g) ( v - w) }_{\alpha}}{ \epsilon \cdot \norm{\rho_{\alpha}(g)}_{\alpha}} \\
        & \le \frac{1}{\epsilon} \cdot \norm{v - w}_{\alpha} .
    \end{aligned}
    $$
    Combining altogether finishes the proof.
\end{proof}

\section{Dynamics of discrete subgroups} \label{section:Discrete}

In this section, let $\Ga < \Gsf$ be a discrete subgroup and fix a non-empty subset $\theta \subset \Delta$. 
Major objects in this paper are transverse subgroups, or Anosov or relatively Anosov subgroups, which we define in this section. An important observation in this section is a local regularity of Iwasawa cocycles for transverse subgroups, with a certain uniform control (Theorem \ref{thm:localregularIwasawa}).
We continue the notations and terminologies introduced in Section \ref{section:Hor}.

The \emph{limit set} $\La_{\theta}(\Ga) \subset \Fc_{\theta}$ of $\Ga$ is defined as the set of all accumulation points of $\Ga$ in the sense of Equation \eqref{eqn:convergence_G}. With respect to the natural $\Ga$-action on $\Fc_{\theta} = \Gsf / \Psf_{\theta}$ given by the left multiplication, $\La_{\theta}(\Ga)$ is $\Ga$-invariant. When $\Ga < \Gsf$ is Zariski dense, $\La_{\theta}(\Ga)$ is the unique $\Ga$-minimal set and is Zariski dense subset of $\Fc_{\theta}$ by the work of Benoist \cite{Benoist1997proprietes}.

\begin{definition}
    A discrete subgroup $\Ga < \Gsf$ is called \emph{$\Psf_{\theta}$-transverse} if 
    \begin{itemize}
        \item for any infinite sequence $\{ g_n \}_{n \in \Nb} \subset \Ga$ and $\alpha \in \theta$,
        $$
        \alpha(\kappa(g_n)) \to + \infty \quad \text{as } n \to + \infty
        $$
        and
        \item for any two distinct $x, y \in \La_{\theta \cup \opp(\theta)}(\Ga)$,
        $$
        x \text{ and } y \text{ are in general position.}
        $$
    \end{itemize}
    We call a $\Psf_{\theta}$-transverse subgroup \emph{non-elementary} if $\# \La_{\theta \cup \opp(\theta)}(\Ga) \ge 3$.
\end{definition}
Note that points in general position are also said to be transverse, and hence the name.

Since $\opp(\kappa(g)) = \kappa(g^{-1})$ for all $g \in \Gsf$, $\Psf_{\theta}$-transversality of $\Ga$ is equivalent to $\Psf_{\opp(\theta)}$-transversality, and to $\Psf_{\theta \cup \opp(\theta)}$-transversality. For a $\Psf_{\theta}$-transverse subgroup $\Ga < \Gsf$ and for any $\theta' \subset \theta \cup \opp(\theta)$, the canonical projection $\Fc_{\theta \cup \opp(\theta)} \to \Fc_{\theta'}$ is restricted to a $\Ga$-equivariant homeomorphism $\La_{\theta \cup \opp(\theta)}(\Ga) \to \La_{\theta'}(\Ga)$ (cf. \cite[Lemma 9.5]{KOW_PD}).

An important property of transverse subgroups is that their actions on limit sets are convergence actions, which was proved by Kapovich--Leeb--Porti.

\begin{theorem}[{\cite[Theorem 4.16]{KLP_Anosov}}]
    Let $\Ga < \Gsf$ be a $\Psf_{\theta}$-transverse subgroup. Then the $\Ga$-action on $\La_{\theta}(\Ga)$ is a convergence action.
\end{theorem}

Hence, the notion of being non-elementary coincides with the one for convergence groups. In addition, we employ the notion of \emph{conical limit set} and the classification of elements of a convergence group in our setting.

\subsection{Shadows for transveres subgroups} \label{subsec:shadowstransverse}

In \cite{BCZZ_PS}, Blayac--Canary--Zhu--Zimmer developed Patterson--Sullivan theory for a general convergence group. We apply their framework to a $\Psf_{\theta}$-transverse subgroup $\Ga < \Gsf$ and its action on $\La_{\theta}(\Ga)$, which was discussed in further details in \cite[Section 10]{BCZZ_2}.

For a non-elementary $\Psf_{\theta}$-transverse subgroup $\Ga < \Gsf$, they showed that the compactification $\Ga \sqcup \La_{\theta}(\Ga)$ in the sense of Equation \eqref{eqn:convergence_G} is metrizable by a metric $\dist_{\rm BCZZ}$ on $\Ga \sqcup \La_{\theta}(\Ga)$ so that the natural $\Ga$-action on $\Ga \sqcup \La_{\theta}(\Ga)$ by left multiplications is a convergence action  \cite[Proposition 2.3]{BCZZ_PS} (see also \cite[Section 10]{BCZZ_2}).

Using this metric, shadows were constructed in \cite{BCZZ_PS} as follows: for $g \in \Ga$ and $\epsilon > 0$, let $\Nc_{\epsilon}(g^{-1}) \subset \Ga \sqcup \La_{\theta}(\Ga)$ denote the open ball of radius $\epsilon$ centered at $g^{-1}$ with respect to the metric $\dist_{\rm BCZZ}$. Then the \emph{$\epsilon$-shadow of $g$} is defined as the set
$$
\Sc_{\epsilon}(g) := g \cdot \left( \La_{\theta}(\Ga) \smallsetminus \Nc_{\epsilon}(g^{-1}) \right) \subset \La_{\theta}(\Ga).
$$

The \emph{conical limit set} 
$$\La_{\theta}^{\rm con}(\Ga) \subset \Fc_{\theta}$$
 then can be defined as the set of points $x \in \La_{\theta}(\Ga)$ such that there exist $\epsilon > 0$ and an infinite sequence $\{g_n\}_{n \in \Nb} \subset \Ga$ such that $x \in \Sc_{\epsilon}(g_n)$ for all $n \in \Nb$. This notion of conical limit set coincides with the conical limit set for the convergence action $\Ga \curvearrowright \La_{\theta}(\Ga)$, as shown in \cite[Lemma 5.4]{BCZZ_PS}.

The shadows defined as above are in fact comparable to the shadows defined in terms of Tits representations in Section \ref{section:Tits}. 
 Note that in item (1) below, excluding finitely many elements might be necessary, because for a fixed $g \in \Ga$, $\Sc_{\epsilon}(g)$ can be the whole limit set $\La_{\theta}(\Ga)$ when $\epsilon$ is small enough while $g \cdot B_{\theta, g}^{\epsilon}$ always excludes hyperplanes in $\Pb(V_{\alpha})$'s.

\begin{proposition} \label{prop:shadowcomparetits}
    Let $\Ga < \Gsf$ be a non-elementary $\Psf_{\theta}$-transverse subgroup.
    \begin{enumerate}
        \item For any $\epsilon > 0$, there exists $\epsilon' = \epsilon'(\epsilon) > 0$ such that
        $$
        \Sc_{\epsilon}(g) \subset g \cdot B_{\theta, g}^{\epsilon'} \quad \text{for all but finitely many } g \in \Ga.
        $$
        \item For any $\epsilon > 0$, there exists $\epsilon' = \epsilon'(\epsilon) > 0$ such that 
        $$
        g \cdot B_{\theta, g}^{\epsilon} \cap \La_{\theta}(\Ga) \subset \Sc_{\epsilon'}(g) \quad \text{for all } g \in \Ga.
        $$
    \end{enumerate}
\end{proposition}

\begin{proof}
We first prove (1). Suppose to the contrary that for some $\epsilon > 0$, there exists an infinite sequence $\{g_n\}_{n \in \Nb} \subset \Ga$ such that
$$
\Sc_{\epsilon}(g_n) \not\subset g_n \cdot B_{\theta, g_n}^{1/n} \quad \text{for all } n \in \Nb.
$$
In other words,
$$
\La_{\theta}(\Ga) \smallsetminus \Nc_{\epsilon}(g_n^{-1}) \not \subset B_{\theta, g_n}^{1/n} \quad \text{for all } n \in \Nb.
$$

Let $\{ \ell_n\}_{n \in \Nb} \subset \Ksf$ be a sequence such that $g_n \in \Ksf \Asf^+ \ell_n$ in Cartan decomposition, for each $n \in \Nb$.
Then there  exists a sequence $\{ x_n \}_{n \in \Nb} \subset \La_{\theta}(\Ga)$ such that for each $n \in \Nb$, 
\begin{enumerate}
    \item[(a)] $\dist_{\rm BCZZ} (x_n, g_n^{-1}) \ge \epsilon$ and 
    \item[(b)] $d_{\Pb(V_{\alpha_n})} \left( \Phi_{\alpha_n}(x_n), \Phi_{\alpha_n}^* \left(\ell_n^{-1} w_0 \Psf_{\opp(\alpha_n)} \right) \right) < \frac{1}{n}$ for some $\alpha_n \in \theta$.
\end{enumerate}
After passing to a subsequence, we may assume that $x_n \to x \in \La_{\theta}(\Ga)$ and $\ell_n \to \ell \in \Ksf$ as $n \to + \infty$, and $\alpha_n = \alpha \in \theta$ for all $n \in \Nb$.

By the $\Psf_{\theta}$-transversality of $\Ga$ and that $\{g_{n}\}_{n \in \Nb}$ is infinite,  $g_n^{-1} \in \Ga$ converges to $\ell^{-1} w_0 \Psf_{\theta} \in \Fc_{\theta}$. Hence, it follows from (a) that $x \neq \ell^{-1} w_0 \Psf_{\theta}$. Recall that we have a $\Ga$-equivariant homeomorphism $\La_{\theta \cup \opp(\theta)}(\Ga) \to \La_{\theta}(\Ga)$ which is a restriction of the projection $\Fc_{\theta \cup \opp(\theta)} \to \Fc_{\alpha}$. Since we also have a convergence $g_n^{-1} \to \ell^{-1} w_0 \Psf_{\theta \cup \opp(\theta)} \in \Fc_{\theta \cup \opp(\theta)}$ as $n \to + \infty$, if $\tilde x \in \La_{\theta \cup \opp(\theta)}(\Ga)$ is the point mapped to $x \in \La_{\theta}(\Ga)$ by the above homeomorphism, then $\tilde x \neq \ell^{-1} w_0 \Psf_{\theta \cup \opp(\theta)}$. Therefore, 
\begin{equation} \label{eqn:shadowcomparetits}
    \tilde x \in \Fc_{\theta \cup \opp(\theta)} \quad \text{and} \quad \ell^{-1} w_0 \Psf_{\theta \cup \opp(\theta)} \in \Fc_{\theta \cup \opp(\theta)} \text{ are in general position}
\end{equation}
 by the $\Psf_{\theta}$-transversality of $\Ga$.

On the other hand, by (b), the line $\Phi_{\alpha}(x) \subset V_{\alpha}$ is contained in the hyperplane $\Phi_{\alpha}^*(\ell^{-1} w_0 \Psf_{\opp(\alpha)}) \subset V_{\alpha}$. This implies that if $x_{\alpha} \in \La_{\alpha}(\Ga)$ is the projection of $x \in \La_{\theta}(\Ga)$, then $x_{\alpha} \in \Fc_{\alpha}$ and $\ell^{-1} w_0 \Psf_{\opp(\alpha)} \in \Fc_{\opp(\alpha)}$ are not in general position. This contradicts Equation \eqref{eqn:shadowcomparetits}, and (1) follows.

\medskip

We now show (2). Suppose to the contrary that for some $\epsilon > 0$, there exists a sequence $\{g_n\}_{n \in \Nb} \subset \Ga$ such that
$$
g_n \cdot B_{\theta, g_n}^{\epsilon} \cap \La_{\theta}(\Ga) \not\subset \Sc_{1/n}(g_n)  \quad \text{for all } n \in \Nb.
$$
In particular,
$$
B_{\theta, g_n}^{\epsilon} \cap \La_{\theta}(\Ga) \not \subset \La_{\theta}(\Ga) \smallsetminus \Nc_{1/n}(g_n^{-1})  \quad \text{for all } n \in \Nb.
$$
If $\{ g_n\}_{n \in \Nb}$ is a finite sequence, then the right hand side above eventually becomes $\La_{\theta}(\Ga)$, yielding a contradiction.

Hence, the sequence $\{ g_n \}_{n \in \Nb}$ is infinite. 
Let $\{ \ell_n\}_{n \in \Nb} \subset \Ksf$ be a sequence such that $g_n \in \Ksf \Asf^+ \ell_n$ in Cartan decomposition, for each $n \in \Nb$.
Then there  exists a sequence $\{ x_n \}_{n \in \Nb} \subset \La_{\theta}(\Ga)$ such that for each $n \in \Nb$,
$$
x_n \in \Nc_{1/n}(g_n^{-1}) \quad \text{and} \quad d_{\Pb(V_{\alpha})} \left( \Phi_{\alpha}(x_n), \Phi_{\alpha}^* \left(\ell_n^{-1} w_0 \Psf_{\opp(\alpha)} \right) \right) \ge \epsilon  \text{ for all } \alpha \in \theta.
$$
After passing to a subsequence, we may assume that $x_n \to x \in \La_{\theta}(\Ga)$ and $\ell_n \to \ell \in \Ksf$ as $n \to + \infty$. Then as in the proof of (1), $g_n^{-1} \in \Ga$ converges to $\ell^{-1} w_0 \Psf_{\theta \cup \opp(\theta)} \in \Fc_{\theta \cup \opp(\theta)}$ as $n \to + \infty$, and if $\tilde x \in \La_{\theta \cup \opp(\theta)}(\Ga)$ projects to $x \in \La_{\theta}(\Ga)$, then 
$$
\tilde x = \ell^{-1} w_0 \Psf_{\theta \cup \opp(\theta)}.
$$
In particular, for each $\alpha \in \theta$, the line $\Phi_{\alpha}(x) = \rho_{\alpha}(\ell^{-1} w_0) V_{\alpha}^+$ is contained in the hyperplane $\rho_{\alpha}(\ell^{-1}) V_{\alpha}^< = \Phi_{\alpha}^* \left( \ell^{-1} w_0 \Psf_{\opp(\alpha)}\right)$.

On the other hand, we have
$$
d_{\Pb(V_{\alpha})} \left( \Phi_{\alpha}(x), \Phi_{\alpha}^* \left(\ell^{-1} w_0 \Psf_{\opp(\alpha)} \right) \right) \ge \epsilon  \text{ for all } \alpha \in \theta,
$$
which is a contradiction.
\end{proof}

\subsection{Iwasawa cocycles for transverse subgroups}

For a $\Psf_{\theta}$-transverse subgroup $\Ga < \Gsf$, we call $\varphi \in \Ga$ \emph{loxodromic} if it is a loxodromic element for a convergence action $\Ga \curvearrowright \La_{\theta}(\Ga)$. For a loxodromic $\varphi \in \Ga$, we denote by $\varphi^+ \in \La_{\theta}(\Ga)$ its attracting fixed point.

\begin{lemma} \label{lem:JordanIwasawafixed}
    Let $\Ga < \Gsf$ be a non-elementary $\Psf_{\theta}$-transverse subgroup. For any loxodromic $\varphi \in \Ga$,
    $$
    \sigma_{\theta}(\varphi, \varphi^+) = \la_{\theta}(\varphi).
    $$
\end{lemma}

\begin{proof}
Fix $\epsilon > 0$ such that $\varphi^+ \in \Sc_{\epsilon}(\varphi^n)$ for all $n \in \Nb$. Then by Proposition \ref{prop:shadowcomparetits} and Lemma \ref{lem:cartanIwasawa}, there exists $C > 0$ such that  
$$
\norm{\sigma_{\theta}(\varphi^n, \varphi^+) - p_{\theta}(\kappa(\varphi^n))} \le C \quad \text{for all large } n \in \Nb.
$$
By Equation \eqref{eqn:Iwasawa_cocycle}, $\sigma_{\theta}(\varphi^n, \varphi^+) = n \cdot \sigma_{\theta}(\varphi, \varphi^+)$ for all $n \in \Nb$. Therefore,
$$
\sigma_{\theta}(\varphi, \varphi^+) = \lim_{n \to + \infty} \frac{ \sigma_{\theta}(\varphi^n, \varphi^+)}{n} = \lim_{n \to + \infty} \frac{ p_{\theta}(\kappa(\varphi^n))}{n} = \la_{\theta}(\varphi)
$$
as desired.
\end{proof}

The following is one of the key observations on transverse subgroups in this paper. It is important that $N$ below does not depend on $g \in \Ga$.

\begin{theorem} \label{thm:localregularIwasawa}
Let $\Ga < \Gsf$ be a non-elementary $\Psf_{\theta}$-transverse subgroup. Let $\varphi \in \Ga$ be loxodromic and $\epsilon > 0$. Then there exists $N = N(\varphi, \epsilon) \in \Nb$ such that for any $g \in \Ga$ and $n \ge N$, if $x, g \varphi g^{-1} x \in \Sc_{\epsilon}(g) \cap g \cdot \Sc_{\epsilon}(\varphi^n)$, then 
$$
\norm{ \sigma_{\theta}( g \varphi g^{-1}, x) - \la_{\theta}(\varphi) } < \epsilon.
$$
\end{theorem}

\begin{proof}
Suppose to the contrary that there exist sequences $\{ g_n \}_{n \in \Nb} \subset \Ga$, $\{ x_n \}_{n \in \Nb} \subset \La_{\theta}(\Ga)$, and $\{k_n\}_{n \in \Nb} \subset \Nb$ such that $k_n \to + \infty$ as $n \to + \infty$ and for all large  $n \in \Nb$,
$$
x_n, g_n \varphi g_n^{-1} x_n \in \Sc_{\epsilon}(g_n) \cap g_n \cdot \Sc_{\epsilon}(\varphi^{k_n}) \quad \text{and} \quad \norm{ \sigma_{\theta}(g_n \varphi g_n^{-1}, x_n) - \la_{\theta}(\varphi) } \ge \epsilon.
$$

If the sequence $\{g_n\}_{n \in \Nb}$ is finite, then after passing to a subsequence, $g_n = g \in  \Ga$ for all $n \in \Nb$, and hence $x_n$ converges to $g \varphi^+$. By Lemma \ref{lem:JordanIwasawafixed}, this is a contradiction.

Hence, the sequence $\{ g_n \}_{n \in \Nb}$ is infinite. By Equation \eqref{eqn:Iwasawa_cocycle} and Equation \eqref{eqn:Iwasawa_inv}, we have for each $n \in \Nb$ that 
$$\begin{aligned}
\sigma_{\theta}( g_n \varphi g_n^{-1}, x_n) & = \sigma_{\theta}( g_n , \varphi g_n^{-1} x_n) + \sigma_{\theta} (\varphi g_n^{-1}, x_n) \\
& = \sigma_{\theta}( g_n , \varphi g_n^{-1} x_n) + \sigma_{\theta} (\varphi, g_n^{-1} x_n) + \sigma_{\theta} (g_n^{-1}, x_n) \\
& = \sigma_{\theta}( g_n , \varphi g_n^{-1} x_n) + \sigma_{\theta} (\varphi, g_n^{-1} x_n) - \sigma_{\theta} (g_n, g_n^{-1} x_n).
\end{aligned}
$$
Since $g_n^{-1} x_n \in \Sc_{\epsilon}(\varphi^{k_n})$ for all $n \in \Nb$, we have $g_n^{-1} x_n \to \varphi^+$ as $n \to + \infty$. This implies $\sigma_{\theta} (\varphi, g_n^{-1} x_n) \to \la_{\theta}(\varphi)$ as $n \to + \infty$, by Lemma \ref{lem:JordanIwasawafixed}. Therefore, for all large $n \in \Nb$, we have
\begin{equation} \label{eqn:squeezingreplace}
\norm{\sigma_{\theta}( g_n , \varphi g_n^{-1} x_n) - \sigma_{\theta} (g_n, g_n^{-1} x_n)} > \frac{\epsilon}{2}.
\end{equation}

On the other hand, since $\{g_n\}_{n \in \Nb}$ is an infinite sequence it follows from Proposition \ref{prop:shadowcomparetits} that for some $\epsilon' = \epsilon'(\epsilon)$, we have after passing to a subsequence that
$$
\Sc_{\epsilon}(g_n) \subset g_n \cdot B_{\theta, g_n}^{\epsilon'} \quad \text{for all } n \in \Nb.
$$
Since $x_n, g_n \varphi g_n^{-1} x_n \in \Sc_{\epsilon}(g_n)$, we have for each $\alpha \in \theta$ and all $n \in \Nb$ that
$$\begin{aligned}
\left| \chi_{\alpha} \left( \sigma_{\theta}(g_n^{-1}, x_n) \right.\right. & - \left.\left. \sigma_{\theta}(g_n^{-1}, g_n \varphi g_n^{-1} x_n) \right) \right| \\
&  \le \frac{1}{\epsilon'} \cdot d_{\Pb(V_{\alpha})} \left( \Phi_{\alpha} \left( g_n^{-1} x_n \right), \Phi_{\alpha} \left( \varphi g_n^{-1} x_n \right) \right)
\end{aligned}
$$
by Proposition \ref{prop:locLipIwasawa}.
By Equation \eqref{eqn:Iwasawa_inv}, this is same ase
$$\begin{aligned}
\left| \chi_{\alpha} \left( \sigma_{\theta}(g_n, g_n^{-1} x_n) \right.\right. & - \left.\left. \sigma_{\theta}(g_n, \varphi g_n^{-1} x_n) \right) \right| \\
&  \le \frac{1}{\epsilon'} \cdot d_{\Pb(V_{\alpha})} \left( \Phi_{\alpha} \left( g_n^{-1} x_n \right), \Phi_{\alpha} \left( \varphi g_n^{-1} x_n \right) \right). 
\end{aligned}
$$

Since $\chi_{\alpha} = k_{\alpha} \omega_{\alpha}$ for some $k_{\alpha} \in \Nb$ and $\{ \omega_{\alpha} : \alpha \in \theta \}$ is a basis of $\fa_{\theta}^*$, there exists $C > 0$ such that 
$$
\norm{ \cdot} \le C \sum_{\alpha \in \theta} \abs{\chi_{\alpha}(\cdot)}.
$$
Together with Equation \eqref{eqn:squeezingreplace},
$$
\frac{\epsilon}{2} < \frac{C}{\epsilon'} \sum_{\alpha \in \theta} d_{\Pb(V_{\alpha})} \left( \Phi_{\alpha} \left( g_n^{-1} x_n \right), \Phi_{\alpha} \left( \varphi g_n^{-1} x_n \right) \right)
$$
for all $n \in \Nb$.

On the other hand, since $x_n, g_n \varphi g_n^{-1} x_n \in g_n \cdot \Sc_{\epsilon}(\varphi^{k_n})$ and hence  we have $g_n^{-1} x_n, \varphi g_n^{-1} x_n \in \Sc_{\epsilon}(\varphi^{k_n})$ for all $n \in \Nb$, both $g_n^{-1} x_n$ and $\varphi g_n^{-1} x_n$ converge to $\varphi^+$ as $n \to + \infty$. Therefore,
$$
\frac{\epsilon}{2} < 0,
$$
contradiction.
\end{proof}

\subsection{Patterson--Sullivan measures of transverse subgroups}

Canary--Zhang--Zimmer initiated the study of Patterson--Sullivan theory for transverse subgroups in \cite{CZZ_transverse}. Later, this was extended further to general convergence groups with coarse cocycles by Blayac--Canary--Zhu--Zimmer \cite{BCZZ_PS}. We briefly review some of their results.

Let $\Ga < \Gsf$ be a $\Psf_{\theta}$-transverse subgroup and $\psi \in \fa_{\theta}^*$. The associated Poincar\'e series is defined by
$$
\mathcal{P}_{\Ga, \psi}(s) := \sum_{g \in \Ga} e^{-s \psi(\kappa(g))} \quad \text{for } s \in \Rb,
$$
and the \emph{critical exponent} is its abscissa of convergence:
$$
\delta_{\psi}(\Ga) := \inf \{ s > 0 : \mathcal{P}_{\Ga, \psi}(s) < + \infty \} \in [0, + \infty].
$$
At the dimension of critical exponent, Patterson--Sullivan measure always exists.
The following is a special case of \cite[Proposition 3.2]{CZZ_transverse}.

\begin{theorem}[{\cite[Proposition 3.2]{CZZ_transverse}}] 
    Let $\Ga < \Gsf$ be a non-elementary $\Psf_{\theta}$-transverse subgroup and $\psi \in \fa_{\theta}^*$. If $\delta_{\psi}(\Ga) < + \infty$, then there exists a $\delta_{\psi}(\Ga)$-dimensional $\psi$-Patterson--Sullivan measure of $\Ga$ on $\La_{\theta}(\Ga)$.
\end{theorem}

Indeed, existence of Patterson--Sullivan measure is equivalent to finiteness of the critical exponent. Under an extra assumption that $\delta_\psi(\Ga) < + \infty$, the following was proved earlier in \cite{CZZ_transverse}. Later, under $\delta_{\psi}(\Ga) < + \infty$ but for measures not necessarily supported on limit sets, it was shown in \cite{KOW_PD}.

\begin{theorem}[{\cite[Proposition 10.1]{BCZZ_2}}] \label{thm:BCZZ_proper}
Let $\Ga < \Gsf$ be a non-elementary $\Psf_{\theta}$-transverse subgroup and  $\psi \in \fa_{\theta}^*$. If there exists a $\delta$-dimensional $\psi$-Patterson--Sullivan measure of $\Ga$ on $\La_{\theta}(\Ga)$, then 
$$
\delta_{\psi}(\Ga) \le \delta.
$$
In particular, $\delta_{\psi}(\Ga) < + \infty$.
\end{theorem}

Whether a Poincar\'e series diverges or not at the critical exponent turns out to be important in the study of related dynamical systems.

\begin{definition} \label{def:divtypegpandmeasure}
    For a $\Psf_{\theta}$-transverse $\Ga < \Gsf$, we say that $\Ga$ is of \emph{$\psi$-divergence type} for $\psi \in \fa_{\theta}^*$ if $\delta_{\psi}(\Ga) < + \infty$ and $\mathcal{P}_{\Ga, \psi}(\delta_{\psi}(\Ga)) = + \infty$. 
    A Patterson--Sullivan measure $\nu$ of $\Ga$ on $\Fc_{\theta}$ is also said to be of \emph{divergence type} if $\nu$ is a $\delta_{\psi}(\Ga)$-dimensional $\psi$-Patterson--Sullivan measure of $\Ga$ such that $\Ga$ is of $\psi$-divergence type.
\end{definition}

As part of their generalization of Hopf--Tsuji--Sullivan dichotomy, Canary--Zhang--Zimmer proved the following for Patterson--Sullivan measures supported on limit sets, which was extended further by Kim--Oh--Wang, removing an assumption on their supports:

\begin{theorem}[{\cite{CZZ_transverse}, \cite{KOW_PD}}] \label{thm:CZZHTS}
    Let $\Ga < \Gsf$ be a non-elementary $\Psf_{\theta}$-transverse subgroup and $\psi \in \fa_{\theta}^*$. Let $\nu$ be a $\delta$-dimensional $\psi$-Patterson--Sullivan measure of $\Ga$ on $\Fc_{\theta}$.

    Then the following are equivalent:
    \begin{enumerate}
    \item $\delta = \delta_{\psi}(\Ga) < + \infty$ and $\Ga$ is of $\psi$-divergence type.
    \item $\nu(\La_{\theta}^{\rm con}(\Ga)) = 1$.
\end{enumerate}
    Moreover, in this case, the $\Ga$-action on $(\La_{\theta}(\Ga), \nu)$ is ergodic and $\nu$ is the unique $\delta$-dimensional $\psi$-Patterson--Sullivan measure on $\La_{\theta}(\Ga)$.
\end{theorem}

\subsection{Hypertransverse subgroups and (relatively) Anosov subgroups}

We consider subclasses of transverse subgroups, whose action on limit sets can be realized by nice actions on Gromov hyperbolic spaces. More precisely, a $\Psf_{\theta}$-tranvserse subgroup $\Ga < \Gsf$ is called \emph{$\Psf_{\theta}$-hypertransverse} if there exists a proper geodesic Gromov hyperbolic space $(Z, d_Z)$ such that
\begin{itemize}
    \item $\Ga$ acts on $Z$ properly discontinuously by isometries and
    \item there exists a $\Ga$-equivariant homeomorphism 
    \begin{equation} \label{eqn:boundary_hypertransverse}
        \Psi : \La_Z(\Ga) \to \La_{\theta}(\Ga),
    \end{equation}
    where $\La_Z(\Ga) \subset \partial Z$ is the limit set of $\Ga$.
\end{itemize}
We call $(Z, d_Z)$ a model space and the map in Equation \eqref{eqn:boundary_hypertransverse} boundary map.
The notion of hypertransverse subgroups was introduced in the work of the second author \cite{kim2024conformal} to study ergodicity of horospherical foliations and its relation to rigidity theory.

\begin{remark}
   Canary--Zhang--Zimmer showed that a general transverse subgroup is realized as a projectively visible subgroup of isometries on a properly convex domain equipped with the Hilbert metric, with an equivariant homeomorphisms between limit sets \cite{CZZ_transverse}. It is unknown whether this Hilbert metric can be chosen to be Gromov hyperbolic on the convex hull of the limit set. It is also unknown whether there exists a transverse subgroup which is not hypertransverse.
\end{remark}

One can see that the shadows for $\Ga$-orbits in $(Z, d_Z)$ are comparable to the shadows defined in Section \ref{subsec:shadowstransverse}, by the same argument as in \cite[Lemma 3.17]{CK_pradak}. Hence, the conical limit set of $\Ga$ in $\partial Z$ is identified with the conical limit set of $\Ga$ in $\La_{\theta}(\Ga) \subset \Fc_{\theta}$, via the boundary map $\Psi$ in Equation \eqref{eqn:boundary_hypertransverse}.

\begin{proposition} \label{prop:componentwiseshadow}
    Let $\Ga < \Gsf$ be a non-elementary $\Psf_{\theta}$-hypertransverse subgroup with a model space $(Z, d_Z)$ and boundary map $\Psi : \La_Z(\Ga) \to \La_{\theta}(\Ga)$.
    \begin{enumerate}
        \item For any $R > 0$ and $z \in Z$, there exists $\epsilon = \epsilon (R, z) > 0$ such that
        $$
        \Psi(O_R(z, g z) \cap \La_Z(\Ga)) \subset \Sc_{\epsilon}(g) \quad \text{for all } g \in \Ga.
        $$
        \item For any $\epsilon > 0$ and $z \in Z$, there exists $R = R(\epsilon, z) > 0$ such that
        $$
        \Sc_{\epsilon}(g) \subset \Psi(O_R(z, g z) \cap \La_Z(\Ga))  \quad \text{for all } g \in \Ga.
        $$
    \end{enumerate}
\end{proposition}

Together with Proposition \ref{prop:shadowcomparetits}(1), Lemma \ref{lem:cartanIwasawa}, and Theorem \ref{thm:localregularIwasawa}, we have the following estimates on Iwasawa cocycles.

\begin{corollary} \label{cor:shadow_use}
    Let $\Ga < \Gsf$ be a non-elementary $\Psf_{\theta}$-hypertransverse subgroup with a model space $(Z, d_Z)$ and boundary map $\Psi : \La_Z(\Ga) \to \La_{\theta}(\Ga)$. Fix $z \in Z$.
    \begin{enumerate}
        \item For any $R > 0$, there exists $c = c(R, z) > 0$ such that for any $g \in \Ga$ and $x \in O_R(z, g z) \cap \La_Z(\Ga)$,
        $$
        \chi_{\alpha}(\kappa(g)) - c \le \chi_{\alpha}( \sigma_{\theta}(g, \Psi(g^{-1}x))) \le \chi_{\alpha}(\kappa(g)) \quad \text{for all } \alpha \in \theta.
        $$
        \item Let $\varphi \in \Ga$ be loxodromic and $R > 0$. For any $\epsilon > 0$, there exists $N = N(\varphi, R, \epsilon, z) \in \Nb$ such that for any  $g \in \Ga$ and $n \ge N$, if $x, g \varphi g^{-1} x \in O_R(z, g z) \cap O_R(g z, g \varphi^n z) \cap \La_Z(\Ga)$, then
        $$
        \norm{ \sigma_{\theta}( g \varphi g^{-1}, \Psi(x)) - \la_{\theta}(\varphi)} < \epsilon.
        $$
    \end{enumerate}
\end{corollary}

Important examples of hypertransverse subgroups are Anosov subgroups and relatively Anosov subgroups. 

\begin{definition} \label{def:Anosovdefbody}
    Let $\Ga < G$ be a non-elementary $\Psf_{\theta}$-transverse subgroup.
    \begin{itemize}
        \item $\Ga$ is \emph{$\Psf_{\theta}$-Anosov} if $\Ga$ is a hyperbolic group and there exists a $\Ga$-equivariant homeomorphism $\partial \Ga \to \La_{\theta}(\Ga)$, where $\partial \Ga$ is the Gromov boundary of $\Ga$.
    \item $\Ga$ is \emph{relatively $\Psf_{\theta}$-Anosov} if $\Ga$ is a relatively hyperbolic group with some choice of a peripheral structure and there exists a $\Ga$-equivariant homeomorphism $\partial_B \Ga \to \La_{\theta}(\Ga)$, where $\partial_B \Ga$ is the Bowditch boundary of $\Ga$ with respect to the chosen peripheral structure.
    \end{itemize}
\end{definition}
Hence, Anosov subgroups and relatively Anosov subgroups are hypertransverse with their Cayley graphs and Gromov models as model spaces, respectively.
Note that Anosov subgroups are also special cases of relatively Anosov subgroups.

Importantly, Canary--Zhang--Zimmer proved that relatively Anosov subgroups are of divergence type. Therefore, in view of Theorem \ref{thm:CZZHTS}, they possess interesting dynamical properties.

\begin{theorem}[{\cite[Theorem 1.1]{CZZ_relative}}]
Let $\Ga < \Gsf$ be a non-elementary relatively $\Psf_{\theta}$-Anosov subgroup and $\psi \in \fa_{\theta}^*$. If $\delta_{\psi}(\Ga) < + \infty$, then $\Ga$ is of $\psi$-divergence type.
\end{theorem}

They also characterized finiteness of critical exponents in terms of limit cones.

\begin{theorem}[{\cite[Theorem 10.1]{CZZ_relative}}]
    Let $\Ga < \Gsf$ be a non-elementary relatively $\Psf_{\theta}$-Anosov subgroup and $\psi \in \fa_{\theta}^*$. Then $\delta_{\psi}(\Ga) < + \infty$ if and only if $\psi > 0$ on $\L_{\Ga} \smallsetminus \{0\}$.
\end{theorem}

\subsection{Closed orbits in $\Hor_{\theta}$}

For a relatively $\Psf_{\theta}$-Anosov subgroup $\Ga < \Gsf$, we denote by $\La_{\theta}^{\rm p}(\Ga) \subset \La_{\theta}(\Ga)$ the \emph{parabolic limit set} of $\Ga$, which is defined as the image of parabolic limit points of $\Ga$ on its Bowditch boundary $\partial_B \Ga$ under the equivariant homeomorphism $\partial_B \Ga \to \La_{\theta}(\Ga)$.

Noting that $\La_{\theta}^{\rm p}(\Ga)$ is $\Ga$-invariant and $\La_{\theta}(\Ga) = \La_{\theta}^{\rm con}(\Ga) \sqcup \La_{\theta}^{\rm p}(\Ga)$, we record that $\Ga$-orbits of horospheres based at $\La_{\theta}^{\rm p}(\Ga)$ are closed in $\Hor_{\theta}$.

\begin{proposition} \label{prop:closedorbitsparabolic}
    Let $\Ga < \Gsf$ be a non-elementary relatively $\Psf_{\theta}$-Anosov subgroup. For any $(x, u) \in \La_{\theta}^{\rm p}(\Ga) \times \fa_{\theta} \subset \Hor_{\theta}$, its orbit $\Ga \cdot (x, u)$ is closed in $\Hor_{\theta}$, i.e.,
    $$
    \overline{\Ga \cdot (x, u)} = \Ga \cdot (x, u) \quad \text{in } \Hor_{\theta}.
    $$
\end{proposition}

To prove Proposition \ref{prop:closedorbitsparabolic}, we introduce a result of the second author joint with Oh and Wang \cite{KOW_PD}. For a non-elementary $\Psf_{\theta}$-transverse subgroup $\Ga < \Gsf$, set
$$
\La_{\theta}^{(2)}(\Ga) := \{ (x, y) \in \La_{\theta}(\Ga) \times \La_{\opp(\theta)}(\Ga) : x \text{ and } y \text{ are in general position} \},
$$
noting that $x \in \La_{\theta}(\Ga)$ and $y \in \La_{\opp(\theta)}(\Ga)$ are in general position if and only if $y$ is not the image of $x$ under the $\Ga$-equivariant homeomorphism $\La_{\theta}(\Ga) \to \La_{\opp(\theta)}(\Ga)$, as $\Ga$ is $\Psf_{\theta}$-transverse.

Now consider the space
$$
\Omega_{\theta}(\Ga) := \La_{\theta}^{(2)}(\Ga) \times \fa_{\theta}
$$
which can be regarded as a subsspace of $\Gsf / \Ssf_{\theta}$. Then the $\Gsf$-action on $\Gsf / \Ssf_\theta$ given by left multiplication induces the $\Ga$-action on  $\Omega_{\theta}(\Ga)$. The induced $\Ga$-action can be written as follows: for $g \in \Ga$ and $(x, y, u) \in \Omega_{\theta}(\Ga)$,
$$
g \cdot (x, y, u) = (gx, gy, u + \sigma_{\theta}(g, x)).
$$
Note that the action on the first and the third components are the same as the action on $\Hor_{\theta}$.

It is not true in general that the $\Ga$-action on $\Gsf / \Ssf_{\theta}$ is properly discontinuous, due to the non-compactness of $\Ssf_{\theta}$. On the other hand, its restriction on $\Omega_{\theta}(\Ga)$ is indeed properly discontinuous.

\begin{theorem}[{\cite[Theorem 1.7]{KOW_PD}}] \label{thm:KOWPD}
    Let $\Ga < \Gsf$ be a non-elementary $\Psf_{\theta}$-transverse subgroup. Then the $\Ga$-action on $\Omega_{\theta}(\Ga)$ is properly discontinuous.
\end{theorem}

Using Theorem \ref{thm:KOWPD}, we prove Proposition \ref{prop:closedorbitsparabolic}.

\begin{proof}[Proof of Proposition \ref{prop:closedorbitsparabolic}]
Let $(x, u) \in \La_{\theta}^{\rm p}(\Ga) \times \fa_{\theta}$. Suppose to the contrary that there exists a sequence  $\{ g_n \}_{n \in \Nb} \subset \Ga$ such that the sequence $g_n (x, u)$ converges in $\Hor_{\theta}$, with the limit in $\Hor_{\theta} \smallsetminus \Ga \cdot (x, u)$.

Let $\Ga_x < \Ga$ denote the stabilizer of $x$. By the relative hyperbolicity of $\Ga$, the $\Ga_x$-action on $\La_{\theta}(\Ga) \smallsetminus \{x\}$ is properly discontinuous and cocompact. Let $\iota : \La_{\theta}(\Ga) \to \La_{\opp(\theta)}(\Ga)$ be the $\Ga$-equivariant homeomorphism. Then the $\Ga_x$-action on $\La_{\opp(\theta)}(\Ga) \smallsetminus \{\iota(x)\}$ is properly discontinuous and cocompact as well. Hence, there exists a compact subset $K \subset \La_{\opp(\theta)}(\Ga) \smallsetminus \{\iota(x)\}$ such that $\Ga_x \cdot K =  \La_{\opp(\theta)}(\Ga) \smallsetminus \{\iota(x)\}$.

Let $y := \lim_{n \to + \infty} g_n \iota(x)$ and let $U \subset \La_{\opp(\theta)}(\Ga)$ be an open neighborhood of $y$ such that $\# (\La_{\opp(\theta)}(\Ga) \smallsetminus U) \ge 2$. Then for each $n \in \Nb$, we have that $(\La_{\opp(\theta)}(\Ga) \smallsetminus \{ \iota(x)\}) \smallsetminus g_n^{-1}U \neq \emptyset$, and hence there exist $w_n \in K$ and $h_n \in \Ga_x$ such that $h_n w_n \notin g_n^{-1}U$. In particular, $g_n h_n w_n \notin U$ for all $n \in \Nb$. After passing to a subsequence, we may assume that the sequence $g_n h_n w_n$ converges in $\La_{\opp(\theta)}(\Ga)$, whose limit is different from $y$.
Hence, the sequence
$$
g_n h_n \cdot (x, w_n)  = (g_n x, g_n h_n w_n) \in \La_{\theta}^{(2)}(\Ga)
$$
converges in $\La_{\theta}^{(2)}(\Ga)$.

We now consider the sequence
\begin{equation} \label{eqn:sequenceinflowspace}
g_n h_n \cdot (x, w_n, u) = (g_n x, g_n h_n w_n, u + \sigma_{\theta}(g_n h_n, x)) \in \Omega_{\theta}(\Ga).
\end{equation}
By Equation \eqref{eqn:Iwasawa_cocycle} and that $h_n \in \Ga_x$, we have for all $n \in \Nb$ that 
$$
\sigma_{\theta}(g_n h_n, x) = \sigma_{\theta}(g_n, x) + \sigma_{\theta}(h_n, x).
$$
Since the sequence $g_n \cdot (x, u) = (g_n x, u + \sigma_{\theta}(g_n, x))$ is convergent in $\Hor_{\theta}$, we also have that $\sigma_{\theta}(g_n, x) \in \fa_{\theta}$ is bounded. Moreover, by \cite[Theorem 4.4, Section 4.4]{CZZ_relative}, we may assume that $\Ga_x$ is a cocompact lattice in a closed Lie subgroup $\Hsf < \Gsf$ with finitely many components such that $\Hsf = \Lsf \ltimes \Usf$ and $\Hsf^\circ = \Lsf^\circ \times \Usf^\circ$, where $\Lsf$ is compact, $\Usf$ is the unipotent radical of $\Hsf$, and the superscript ${}^\circ$ denotes the identity component. Therefore, we have $\sigma_{\theta}(g_n h_n, x) = \sigma_{\theta}(g_n, x)$ for all $n \in \Nb$, and in particular the sequence $\sigma_{\theta}(g_n h_n, x) = \sigma_{\theta}(g_n, x) \in \fa_{\theta}$ is bounded as well.

Combining altogether, we have that the sequence
$$
g_n h_n \cdot (x, w_n, u) \in \Omega_{\theta}(\Ga)
$$
in Equation \eqref{eqn:sequenceinflowspace} converges in $\Omega_{\theta}(\Ga)$, after passing to a subsequence. On the other hand, the sequence $(x, w_n, u) \in \Omega_{\theta}(\Ga)$ converges in $\Omega_{\theta}(\Ga)$ as well after passing to a subsequence, since $w_n \in K$ and $K \subset \La_{\opp(\theta)}(\Ga)$ is a compact subset disjoint from $\iota(x)$.

By the proper discontinuity of the $\Ga$-action on $\Omega_{\theta}(\Ga)$ (Theorem \ref{thm:KOWPD}), this implies that the sequence $\{ g_n h_n \}_{n \in \Nb} \subset \Ga$ is a finite sequence. After passing to a subsequence, we may assume that for some $g \in \Ga$,
$$
g_n h_n = g  \quad \text{for all } n \in \Nb.
$$
Then we have
$$
g_n \cdot (x, u) = g h_n^{-1} \cdot (x, u) = g \cdot (x, u) \quad \text{for all } n \in \Nb.
$$
On the other hand, the sequence $g_n \cdot (x, u) $ is assumed to converge to a point in $\Hor_{\theta} \smallsetminus \Ga \cdot (x, u)$, which is a contradiction.
This finishes the proof.
\end{proof}

\section{Measure classifications} \label{section:MC}

In this section, we prove the main theorem of this paper. For $\Ga < \Gsf$, we set
$$
\Spec_{\theta}(\Ga) := \{ \la_{\theta}(\ga) : \ga \in \Ga \text{ loxodromic} \}.
$$
We call $\Spec_{\theta}(\Ga)$  \emph{non-arithmetic} if it generates a dense additive subgroup of $\fa_{\theta}$. The non-arithmeticity holds when $\Ga$ is Zariski dense by the work of Benoist \cite{Benoist2000proprietes} (Theorem \ref{thm:Benoist_nonarithmetic}).

Recalling the spaces
$$\begin{tikzcd}
\Rc_{\Ga, \theta} \arrow[r, phantom, "\subset"] & \E_{\Ga, \theta} \arrow[r, phantom, "\subset"]& \Hor_{\theta} \\
\La_{\theta}^{\rm con}(\Ga) \times \fa_{\theta} \arrow[u, phantom, sloped, "="] & \La_{\theta}(\Ga) \times \fa_{\theta} \arrow[u, phantom, sloped, "="] & \Fc_{\theta} \times \fa_{\theta} \arrow[u, phantom, sloped, "="]
\end{tikzcd}
$$
we now state our main theorem.

\begin{theorem} \label{thm:uniqueRadon}
Let $\Ga < \Gsf$ be a non-elementary $\Psf_{\theta}$-hypertransverse subgroup with non-arithmetic $\Spec_{\theta}(\Ga)$.
If $\mu$ is a $\Ga$-invariant ergodic Radon measure on $\mathcal{H}_{\theta}$ supported on $\Rc_{\Ga, \theta}$, then 
    $$
    \mu \text{ is a constant multiple of $\mu_{\nu}^{\BR}$}
    $$
    for some divergence-type Patterson--Sullivan measure $\nu$ of $\Ga$. 

\end{theorem}

See also Theorem \ref{thm:trbytrlengthqi} below for normal subgroups.
Together with Proposition \ref{prop:closedorbitsparabolic}, we have the following corollary of Theorem \ref{thm:uniqueRadon} for relatively Anosov subgroups.

\begin{corollary} \label{cor:relAnosovcasemeasureclassification}
Let $\Ga < \Gsf$ be a non-elementary relatively $\Psf_{\theta}$-Anosov subgroup with non-arithmetic $\Spec_{\theta}(\Ga)$. If $\mu$ is a $\Ga$-invariant ergodic Radon measure on $\E_{\Ga, \theta} \subset \Hor_{\theta}$, then either
\begin{enumerate}
    \item $\mu$ is supported on a closed $\Ga$-orbit in $\Hor_{\theta} \smallsetminus \Rc_{\Ga, \theta}$, or 
    \item $\mu$ is a constant multiple of $\mu_{\nu}^{\BR}$ for some divergence-type Patterson--Sullivan measure $\nu$ of $\Ga$. 
\end{enumerate}
\end{corollary}

The rest of this section is devoted to the proof of Theorem \ref{thm:uniqueRadon}. We deduce statements in the introduction from Theorem \ref{thm:uniqueRadon} and Corollary \ref{cor:relAnosovcasemeasureclassification} at the end of this section. 
We prove the theorem by establishing a robust relation between invariant Radon measures and guided limit sets. Note that due to ergodic decompositions, Theorem \ref{thm:uniqueRadon} can be regarded as the classification of $\Ga$-invariant Radon measures on $\Rc_{\Ga, \theta} \subset \mathcal{H}_{\theta}$. 

\subsection{Concentration on guided limit sets}
 
In the rest of this section, let $\Ga_0 < \Gsf$ be a non-elementary $\Psf_{\theta}$-hypertransverse subgroup with a model space $(Z, d_{Z})$ and the boundary map $\Psi : \La_Z(\Ga_0) \to \La_{\theta}(\Ga_0)$. Since $\Psi$ is a $\Ga_0$-equivariant homeomorphism, we identify $\La_Z(\Ga_0)$ and $\La_{\theta}(\Ga_0)$ via $\Psi$ for simplicity.  We also fix a basepoint $z_0 \in Z$.

For a loxodromic $\varphi \in \Ga_0$ and $C > 0$, denote by $\La_{Z}^{\varphi, C}(\Ga_0) \subset \La_{Z}(\Ga_0)$ the $(\varphi, C)$-guided limit set of $\Ga_0$. We set

$$
\La_{\theta}^{\varphi, C}(\Ga_0) := \Psi \left( \La_{Z}^{\varphi, C}(\Ga_0) \right)
$$
and call it the \emph{$(\varphi, C)$-guided limit set} of $\Ga_0$ as well.

For the purpose of Theorem \ref{thm:trbytrlengthqi}, we consider a non-elementary normal subgroup $\Gamma \triangleleft \Gamma_0$. Due to the minimality, $\La_Z(\Ga) = \La_Z(\Ga_0)$ and $\La_{\theta}(\Ga) = \La_{\theta}(\Ga_0)$. In particular, $\Ga$ is a non-elementary $\Psf_{\theta}$-hypertransverse subgroup with the same model space and the boundary map as well. Nevertheless, note that $\Gamma$ may have strictly smaller conical limit set and guided limit set than $\Ga_0$. Hence, we have
$$
\E_{\Ga, \theta} = \E_{\Ga_0, \theta} \quad \text{while} \quad \Rc_{\Ga, \theta} \subset \Rc_{\Ga_0, \theta}
$$
which may be a strict inclusion in general. We first show that  invariant ergodic Radon measures are charged on guided limit sets. 

\begin{theorem}\label{thm:radonCharge}

Let $\varphi \in \Ga < \Isom(Z)$ be a loxodromic element, and let $C = C(\varphi) > 0$ be as in Lemma \ref{lem:BGIPHeredi}.
Let $\mu$ be a $\Ga$-invariant ergodic Radon measure on $\mathcal{H}_{\theta}$ supported on $\Rc_{\Ga_0, \theta}$. Then $\mu$ is supported on
$$
 \La_{\theta}^{\varphi, C}(\Ga_0) \times \fa_{\theta} \subset \mathcal{H}_{\theta}.
$$
\end{theorem}
 
\begin{proof} 
    Applying Lemma \ref{lem:extension} to $\varphi \in \Ga <  \Isom(Z)$, we get $\beta(\varphi) > 0$ and $a_{1}, a_{2}, a_{3} \in \Gamma$ satisfying the conclusion of Lemma \ref{lem:extension} as isometries on  $Z$. Let $ C(\varphi) > 0$ be as in Lemma \ref{lem:BGIPHeredi} for $g=\varphi$. We set $C_0 := 10(\beta(\varphi) +C(\varphi))$. 

For each $K>0$ let 
\[
\La_{K} :=  \left\{ x \in \La_{\theta}(\Ga) : \begin{matrix}
\exists \text{ an infinite sequence } \{g_j\}_{j \in \N} \subset \Ga_0 \text{ s.t.}\\
x \in O_K(z_0, g_j z_0) \text{ for all } j \in \N
\end{matrix}  \right\}.
\]
Then  we have  \[
\Lambda_{\theta}^{\rm con} (\Gamma_0) \times \fa_{\theta} = \bigcup_{K > 0} \La_K \times \fa_{\theta}
\]
by Proposition \ref{prop:componentwiseshadow}.
Since $\mu$ is supported on $\Rc_{\Ga_0, \theta} =  \Lambda_{\theta}^{\rm con} (\Gamma_0) \times \fa_{\theta} $, 
$$\mu  (\La_K \times \fa_{\theta}) > 0 \text{ for all large $K > 0$.}
$$
We fix such $K > 100 C_{0} + 2 \sum_{j = 1}^3 d_Z(z_0, a_j z_0)$.

We identify $\fa_{\theta} = \Rb^{\# \theta}$ via the map $u \mapsto (\chi_{\alpha} (u))_{\alpha \in \theta}$.
For each $R > 0$, we set 
$$\mathcal{H}_{K, R} := \La_K \times [-R, R]^{\# \theta}.$$ 
Since $\La_K \times \fa_{\theta} = \bigcup_{R=1}^{\infty} \mathcal{H}_{K, R}$, $$
\mu(\mathcal{H}_{K, R}) > 0 \quad \text{for all large } R > 0.
$$
We fix such $R>0$. 

Now we fix $n>\frac{1000(C_{0}+K+1)}{\tau_{\varphi}}$ and $k > 0$. We define a map
$$
F = F_{n, k}  : \mathcal{H}_{K, R} \to \mathcal{H}_{\theta}
$$
as follows. For each $\xi = (x, u) \in \mathcal{H}_{K, R}$, there exists $g \in \Gamma_0$ such that 
\begin{equation} \label{eqn:defofgXi}
d_Z(z_0, gz_0) > k \quad \text{and} \quad x \in O_K(z_0, g z_0) .
\end{equation} 
Among many such $g$'s, take the one with minimal  $d_Z(z_0, g z_0)$ and call it $g_{\xi}$.\footnote{There exists a technicality when several candidates tie. An easy rescue is to first enumerate $\Gamma_0 = \{g^{(1)}, g^{(2)}, \ldots\}$, and we choose the earliest whenever there is a tie.} Then the map $\xi \in \mathcal{H}_{K, R} \mapsto g_{\xi}$ is Borel measurable.
By Lemma \ref{lem:extensionHoro}, there exists $a_{\xi} \in \{a_{1}, a_{2}, a_{3}\}$ such that \footnote{Again, when more than one of $\{a_1, a_2, a_3\}$ do the job we choose the earliest.}

\begin{equation} \label{eqn:defofFmap}
        \left(z_0, g_{\xi} \cdot a_{\xi} [z_0, \varphi^{n} z_{0}], g_{\xi} \cdot a_{\xi} \varphi^{n} a_{\xi}^{-1} \cdot g_{\xi}^{-1} x \right)
        \text{ is $C_{0}$-aligned in }Z \cup \partial Z
\end{equation}
where $x \in \La_{\theta}(\Ga)$ is considered as a point $\La_Z(\Ga) \subset \partial Z$ via $\Psi$.
 This map $\xi \mapsto a_{\xi}$ is also Borel measurable. 
We now set 
 \[
F (\xi) := g_{\xi} \cdot a_{\xi} \varphi^{n} a_{\xi}^{-1} \cdot g_{\xi}^{-1} \xi.
\]

Let \[
D := 100\left(C_{0} +  d_Z(z_0, \varphi^n z_0) + \sum_{j=1}^{3} d_Z(z_0, a_j z_0) \right).
\]

\begin{claim*}
We have that 
\begin{equation} \label{eqn:finitetoone}
F \text{ is at most } 3 \cdot \# \{ g \in \Ga_0 : d_Z(z_0, gz_0) \le D\}\text{-to-one}.
\end{equation}
\end{claim*}

To prove this claim, suppose that we have $\xi, \xi' \in \mathcal{H}_{K, R}$ with the same image $F(\xi) = F(\xi') =: (x_0, v_0) \in \La_{\theta}(\Ga) \times \fa_{\theta}$.

Regarding $x_0$ as a point in $\partial Z$, let $\{u_i\}_{i \in \N} \subset Z$ be a sequence converging to $x_0$. Up to a subsequence, 
$$
\left( z_0, g_{\xi} a_{\xi} [z_0, \varphi^{n} z_0], u_i \right) \quad \text{is $C_{0}$-aligned for all $i \in \N$}.
$$
Here, note that $d_Z(z_0, \varphi^n z_0) \ge \tau_{\varphi} n \ge 100C_{0}$. It follows from Lemma \ref{lem:BGIPFellow} that for each $i \in \N$, there exists $P \in [z_0, u_i]$ such that
$d_Z(g_{\xi} a_{\xi} z_0, P) \le 3C_{0}$, and similarly, there exists $Q \in [z_0, u_i]$ such that $d_Z(g_{\xi'} a_{\xi'} z_0, Q) \le 3C_{0}$.
We have three cases. 
\begin{enumerate}
\item If $d_Z(z_0, g_\xi z_0) < d_Z(z_0, g_{\xi'} z_0) - 0.5D$: In this case, along $[z_0, u_i]$, $P$ is closer to $z_0$ than $Q$ is, and $d_Z(P, Q) \ge 0.5D - 6C_{0}$. 
For convenience we write $\eta := g_{\xi} a_{\xi} [z_0, \varphi^{n} z_0]$. Note that 
$$
\begin{aligned}
d_Z(u_i, \eta) & \ge  d_Z(u_i, g_{\xi} a_{\xi} z_0) - d_Z(g_{\xi} a_{\xi}z_0, g_{\xi} a_{\xi}\varphi^n z_0) \\
& \ge d_Z(u_i, P) - 3C_{0} - d_Z(z_0, \varphi^n z_0) \\
& \ge d_Z(u_i, Q) + 0.2 D \\
& > d_Z(u_i, g_{\xi'} a_{\xi'} z_0) + d_Z(a_{\xi'} z_0, z_0) + C_{0}\\
& \ge d_Z(u_i, g_{\xi'} z_0) + C_{0}.
\end{aligned}
$$
This implies that $d_Z(\eta, [u_i, g_{\xi'} z_0]) \ge C_{0}$. Since $\eta$ is $C_{0}$-contracting, $\diam \pi_{\eta} ([u_i, g_{\xi'}z_0]) \le C_{0}$. Since $(z_0, \eta, u_i)$ is $C_{0}$-aligned, we also have that
\begin{equation} \label{eqn:case1}
(z_0, \eta, g_{\xi'} z_0) \quad \text{is $2C_{0}$-aligned.}
\end{equation}

We denote by $x' \in \partial Z$ the first component of $\xi'$. Since $x' \in O_K(z_0, g_{\xi'} z_0)$, there exists  a geodesic ray $[z_0, x'] \subset Z$ that intersects $\Nc_K(g_{\xi'} z_0)$. Let $\{u_i'\}_{i \in \N} \subset [z_0, x']$ be a sequence converging to $x'$. We claim that
\begin{equation} \label{eqn:case1-2}
(\eta, u_i') \quad \text{is $(K + 30C_{0})$-aligned for all large $i \in \N$.}
\end{equation}
Suppose to the contrary that, passing to a subsequence, $(\eta, u'_i)$ is never $(K+30C_{0})$-aligned. We choose \[
p \in \pi_{\eta}(z_{0}), \quad q \in \pi_{\eta}(u_{i}'), \quad \text{and} \quad r \in \pi_{\eta}(g_{\xi'} z_{0}).
\]
By Equation \eqref{eqn:case1} and the assumption that $(\eta, u_i')$ is not $(K + 30C_{0})$-aligned, $r$ is closer to $g_{\xi}a_{\xi} \varphi^{n}z_0$ than $p$ and $q$ are. Moreover, we have $d_Z(p, r), d_Z(q, r) \ge K + 25C_{0}$. Now, it follows from Lemma \ref{lem:BGIPFellow}  that 
 \[\begin{aligned}
d_Z(z_0, u'_i) & \le d_Z(z_0, \eta) + d_Z(p, q) + d_Z(\eta, u'_i), \\
d_Z(g_{\xi'} z_0, u'_i) & =_{8C_{0}} d_Z(g_{\xi'}z_0, \eta) + d_Z(r, q) + d_Z(\eta, u'_i), \quad \text{and} \\
d_Z(z_0, g_{\xi'} z_0) & =_{8C_{0}} d_Z(z_0, \eta) + d_Z(p, r) + d_Z(\eta, g_{\xi'} z_0).
\end{aligned}
\]
Then we have
$$\begin{aligned}
d_Z(z_0, g_{\xi'} z_0) & + d_Z(g_{\xi'} z_0, u'_i) - d_Z(z_0, u'_i) \\
& \ge 2 d_Z(g_{\xi'} z_0, \eta) + 2 \min ( d_Z(p, r), d_Z(q, r)) - 16 C_{0} \\
& \ge 2K + 34 C_{0}
\end{aligned}
$$
In particular, 
$d_Z(z_0, u'_i) - d_Z(g_{\xi'} z_0, u'_i) \le d_Z(z_0, g_{\xi'} z_0) - 2K - C_{0}$. This contradicts that $[z_0, u_i']$ intersects $\Nc_K(g_{\xi'} z_0)$ for all large $i \in \Nb$,  and therefore Equation \eqref{eqn:case1-2} follows.

\medskip

By Lemma \ref{lem:BGIPFellow}, it follows from Equation \eqref{eqn:case1} and Equation \eqref{eqn:case1-2} that $[z_0, u_{i}']$ passes through the $4C_{0}$-neighborhood of $g_{\xi} a_{\xi} z_0$, for all large $i \in \N$. This implies
$$
x' \in O_{4 C_0 + d_Z(z_0, a_{\xi} z_0)}(z_0, g_{\xi} z_0) \subset O_K(z_0, g_{\xi} z_0).
$$
Meanwhile, we also have $k < d_Z(z_0, g_{\xi} z_0) < d_Z(z_0, g_{\xi'}z_0) - 0.5 D$. This contradicts  the definition of $g_{\xi'}$ that $d_Z(z_0, g_{\xi'} z_0)$ is minimal among the elements of $\Ga$ satisfying Equation \eqref{eqn:defofgXi}.

\item  If $d_Z(z_0, g_{\xi'} z_0) < d_Z(z_0, g_{\xi} z_0) - 0.5D$: In this case, one can obtain a similar contradiction as in (1).

\item If $d_Z(z_0, g_{\xi}z_0) =_{0.5D} d_Z(z_0, g_{\xi'} z_0)$: Recall that for each fixed $i \in \N$, we have $P, Q \in [z_0, u_i]$ such that $d_Z(g_{\xi}a_{\xi} z_0, P), d_Z(g_{\xi'}a_{\xi'} z_0, Q)\le 3C_{0}$. Hence, we have
$$
d_Z(g_{\xi} z_0, P), d_Z(g_{\xi'} z_0, Q) \le 0.1 D.
$$
Since both $P$ and $Q$ belong to the geodesic $[z_0, u_i]$, this, together with $d_Z(z_0, g_{\xi} z_0) =_{0.5D} d_Z(z_0, g_{\xi'} z_0)$, implies 
$$
d_Z(g_{\xi} z_0, g_{\xi'} z_0) \le D.
$$

Now when $\xi$ is given (and hence $a_{\xi}, g_{\xi}$ are given as well),
\[
\xi' = g_{\xi'} a_{\xi'} \varphi^{-n} a_{\xi'}^{-1} (g_{\xi'}^{-1} g_{\xi}) a_{\xi} \varphi^{n} a_{\xi}^{-1} g_{\xi}^{-1} \xi
\]
is determined by $g_{\xi'}^{-1} g_{\xi}$ and $a_{\xi'} \in \{a_1, a_2, a_3\}$. The number of these choices is at most $3 \cdot \# \{ g \in \Ga_0 : d_Z(z_0, gz_0) \le D\}$. 

\end{enumerate}
Therefore,  Equation \eqref{eqn:finitetoone} follows.

\medskip

We simply write $M := 3 \cdot \# \{ g \in \Ga_0 : d_Z(z_0, g z_0) \le D\}$, which is finite. Then we have
\[\begin{aligned}
\mu(F(\mathcal{H}_{K, R})) &= \mu \left( \bigcup_{g \in \Ga_0}, a \in \{a_1, a_2, a_3\} F \left( \{\xi \in \mathcal{H}_{K, R} : g_{\xi} = g, a_{\xi} = a\} \right) \right) \\
&\ge \frac{1}{M}\sum_{g \in \Ga_0}, a \in \{a_1, a_2, a_3\} \mu \left( F \left( \{\xi \in \mathcal{H}_{K, R} : g_{\xi} = g, a_{\xi} = a\} \right) \right) \\
&=  \frac{1}{M}\sum_{g \in \Ga_0, a \in \{a_1, a_2, a_3\}} \mu \left( ga\varphi^n a^{-1} g^{-1} \{\xi \in \mathcal{H}_{K, R} : g_{\xi} = g, a_{\xi} = a\} \right) \\
&= \frac{1}{M}\sum_{g \in \Ga_0, a \in \{a_1, a_2, a_3\}} \mu \left( \{\xi \in \mathcal{H}_{K, R} : g_{\xi} = g, a_{\xi} = a\} \right) \\
&= \frac{1}{M}\mu(\mathcal{H}_{K, R}).
\end{aligned}
\]

Now to see the image of $F$, let $\xi = (x, u) \in \mathcal{H}_{K, R}$. For simplicity, write $g := g_{\xi}$ and $a := a_{\xi}$. Then
$$
F(\xi) = ( ga \varphi^n a^{-1} g^{-1} x, u + \sigma_{\theta}( g a \varphi^n a^{-1}  g^{-1}, x)).
$$
We first estimate the $\fa_{\theta}$-component. By Equation \eqref{eqn:Iwasawa_cocycle}, 
$$
\begin{aligned}
\sigma_{\theta}( g a \varphi^n a^{-1} g^{-1}, x) & = \sigma_{\theta}(g a \varphi^n, a^{-1} g^{-1} x) + \sigma_{\theta}( a^{-1} g^{-1}, x) \\
& = \sigma_{\theta}( g a, \varphi^n a^{-1} g^{-1} x) + \sigma_{\theta}(\varphi^n, a^{-1} g^{-1} x) \\
& \quad + \sigma_{\theta}(a^{-1}, g^{-1} x) + \sigma_{\theta}(g^{-1}, x).
\end{aligned}
$$

Recall from Equation \eqref{eqn:defofFmap} that $(z_0, g a [ z_0, \varphi^n z_0], g a \varphi^n a^{-1} g^{-1} x)$ is $C_0$-aligned.
By Lemma \ref{lem:shadowandalignment}(2), there exists $r > 0$ depending on $C_0$ such that
$$
g a \varphi^n a^{-1} g^{-1} x \in O_r( z_0, g a z_0) \cap O_r( g a z_0, g a \varphi^n z_0).
$$
By Corollary \ref{cor:shadow_use}, there exists $c > 0$ depending on $r$ so that for each $\alpha \in \theta$,
$$
\begin{aligned}
\chi_{\alpha}(\kappa(ga)) - c & \le \chi_{\alpha} ( \sigma_{\theta} (g a, \varphi^n a^{-1} g^{-1} x))  \le \chi_{\alpha}(\kappa(ga))  \\
\chi_{\alpha}(\kappa(\varphi^n)) - c & \le \chi_{\alpha}(\sigma_{\theta}(\varphi^n, a^{-1} g^{-1} x)) \le \chi_{\alpha}(\kappa(\varphi^n)).
\end{aligned}
$$

Note also that by Lemma \ref{lem:Iwasawa_log},
$$
- \chi_{\alpha}(\kappa(a)) \le \chi_{\alpha}( \sigma_{\theta}(a^{-1}, g^{-1} x)) \le \chi_{\alpha}(\kappa(a^{-1})).
$$
Hence,
$$\begin{aligned}
\chi_{\alpha}(\sigma_{\theta} & (ga \varphi^n a^{-1} g^{-1}, x)) \\
&   =_{2c + \chi_{\alpha}(\kappa(a)) + \chi_{\alpha}(\kappa(a^{-1})) } \chi_{\alpha}(\kappa(ga)) + \chi_{\alpha}(\kappa(\varphi^n)) + \chi_{\alpha}(\sigma_{\theta}(g^{-1}, x)) \\
& =_{ \chi_{\alpha}(\kappa(a)) + \chi_{\alpha}(\kappa(a^{-1})) }\chi_{\alpha}(\kappa(g)) + \chi_{\alpha}(\kappa(\varphi^n)) - \chi_{\alpha}(\sigma_{\theta}(g, g^{-1}x)) 
\end{aligned}
$$
where the second equality is due to 
$$
\chi_{\alpha}(\kappa(g)) - \chi_{\alpha}(\kappa(a^{-1})) \le \chi_{\alpha}(\kappa(ga)) \le \chi_{\alpha}(\kappa(g)) + \chi_{\alpha}(\kappa(a)).
$$
Moreover, since $x \in O_K(z_0, g z_0)$, it follows from Corollary \ref{cor:shadow_use}(1) that there exists $c' > 0$ depending on $K$ such that
$$
\abs{ \chi_{\alpha}(\kappa(g)) - \chi_{\alpha}(\sigma_{\theta}(g, g^{-1} x))} < c'.
$$
Hence, setting
$$
\widehat{D} := 2c + 2 \sum_{j = 1}^3 \sum_{\alpha \in \theta} \chi_{\alpha}(\kappa(a_j)) + \chi_{\alpha}(\kappa(a_j^{-1})) + \chi_{\alpha}(\kappa(\varphi^n)) + c',
$$
we have that
$$
u + \sigma_{\theta}( g a \varphi^n a^{-1} g^{-1}, x) \in \left[ -R - \widehat{D}, R + \widehat{D} \right]^{\# \theta}.
$$
We note that $\widehat{D}$ is independent of $g$.

In addition, by Equation \eqref{eqn:defofgXi}, we have $d_Z(z_0, ga z_0) > k - \sum_{j = 1}^3 d_Z(z_0, a_j z_0)$ and that $(z_0, ga [z_0, \varphi^n z_0], g a \varphi^n a^{-1} g^{-1}x)$ is $C_{0}$-aligned by Equation \eqref{eqn:defofFmap}.

This implies that $F(\mathcal{H}_{K, R})$ is contained in 
 \[
B_{k;n} := \left\{ (y, v) \in \mathcal{H}_{\theta} : 
\begin{matrix}
v \in \left[-R - \widehat{D}, R+ \widehat{D} \right]^{\# \theta} \text{ and } \exists h \in \Ga_0 \text{ such that}\\
d_Z(z_0, h z_0) > k - \sum_{j = 1}^3 d_Z(z_0, a_j z_0) \text{ and} \\ 
 (z_0, h [z_0, \varphi^n z_0], y) \text{ is $C_{0}$-aligned}
\end{matrix}\right\}.
\]
Hence, we have
$$
\mu (B_{k;n}) \ge \mu(\mathcal{H}_{K, R})/M > 0.
$$
Note that the set $B_{k;n}$ is decreasing in $k$. Since $\mu$ is a Radon measure and $B_{k;n} \subset \partial Z \times \left[-R - \widehat{D}, R + \widehat{D}\right]^{\# \theta}$ which is \emph{compact}, we have $\mu(B_{k;n}) < + \infty$. Therefore, setting
\begin{equation} \label{eqn:defofBn}
B_n := \bigcap_{k > 0} B_{k;n},
\end{equation}
we have
\begin{equation} \label{eqn:positiveBn}
\mu(B_n) = \lim_{k \to + \infty} \mu(B_{k;n}) \ge \mu(\mathcal{H}_{K, R})/M > 0,
\end{equation}
noting that $M$ does not depend on $k$.

Now, $\Gamma \cdot B_{n}$ is a $\Gamma$-invariant set of positive $\mu$-measure. Hence, by the $\Ga$-ergodicity of $\mu$, we have that $\Ga \cdot B_n$ is $\mu$-conull, and therefore
$$
\bigcap_{n} \Ga \cdot B_n \quad \text{is $\mu$-conull}.
$$
We then show that for each $(y, v) \in \bigcap_{n} \Ga \cdot B_n $, we have $y \in \La_{\theta}^{\varphi, 2C_{0}}(\Ga_0)$. This finishes the proof by Lemma \ref{lem:squeezedInv}.

Let $(y, v) \in \bigcap_{n} \Ga \cdot B_n$. Then for each large enough $n \in \N$, there exists  $h_0 \in \Ga$ so that
$(z_0, h [z_0, \varphi^n z_0], h_0^{-1}y)$ is $C_{0}$-aligned for infinitely many $h \in \Ga_0$.
In other words,
$$\begin{matrix}
    ( h_0 z_0, h_0 h [z_0, \varphi^n z_0], y)\\
\text{is $C_{0}$-aligned for infinitely many } h \in \Ga.
\end{matrix}
$$
Among infinitely many such $h \in \Ga_0$, we can choose one such that
$$d_Z(h_0 z_0, h_0 h [z_0, \varphi^n z_0]) > d_Z(z_0, h_0 z_0) + C_0$$
and hence 
$$
d_Z([h_0 z_0, z_0], h_0 h [z_0, \varphi^n z_0]) > C_0.
$$
Since $h_0 h [z_0, \varphi^n z_0]$ is $C_0$-contracting, $\pi_{h_0 h [z_0, \varphi^n z_0 ]}([h_0 z_0, z_0 ])$ has diameter at most $C_0$. Therefore,
$$
( z_0, h_0 h [z_0, \varphi^n z_0], y) \quad \text{is $2C_0$-aligned.}
$$
Since this holds for all large $n \in \N$, we conclude $y \in \La_{\theta}^{\varphi, 2C_0}(\Ga_0)$.
\end{proof}

\subsection{Quasi-invariance under translations} \label{subsec:trqi}

We rephrase the $\Asf_{\theta}$-action on $\Hor_{\theta}$ in terms of the $\fa_{\theta}$-action as follows: for $u \in \fa_{\theta}$, consider a map $T_u : \mathcal{H}_{\theta} \to \mathcal{H}_{\theta}$ given by 
$
(x, v) \mapsto (x, v + u)
$.
For a Radon measure $\mu$ on $\mathcal{H}_{\theta}$, we consider its pullback  measure $T_u^*\mu$, which is given by 
$$T_u^* \mu (E) := \mu(T_u E)$$
 for each Borel subset $E \subset \mathcal{H}_{\theta}$.
For a loxodromic  $\varphi \in \Ga$, we simply write $T_{\varphi} := T_{\la_{\theta}(\varphi)}$.
We show that $\Ga$-invariant ergodic measures on $\mathcal{H}_{\theta}$ are quasi-invariant under this translation.

\begin{theorem} \label{thm:trbytrlengthqi}
    Let $\Ga_0 < \Gsf$ be a non-elementary $\Psf_{\theta}$-hypertransverse subgroup and let $\Ga \triangleleft \Ga_0$ be a non-elementary normal subgroup. Let $\mu$ be a $\Ga$-invariant ergodic Radon measure on $\mathcal{H}_{\theta}$ supported on $\Rc_{\Ga_0, \theta}$. Then for a loxodromic $\varphi \in \Ga$, there exists $s \ge 0$ such that
    $$
    \frac{ d T_{\varphi}^* \mu}{d \mu} = e^{s} \quad \text{a.e.}
    $$
    In particular, if $\Spec_{\theta}(\Ga)$ is non-arithmetic (e.g. $\Ga$ is Zariski dense), then $\mu$~is quasi-invariant under the $\fa_{\theta}$-action.
\end{theorem}

\begin{proof}
Let $\varphi \in \Ga$ be a loxodromic element and let $C= C(\varphi) > 0$ be the constant satisfying Lemma \ref{lem:BGIPHeredi} regarding $\varphi$ as an isometry on $Z$, with the choice of axis $\ga : \R \to X$. 
As in \cite[Proof of Theorem 7.10]{CK_ML}, we may assume that 
$$
\tau_{\varphi} > 100 ( \delta + C + c)
$$
where $\delta > 0$ is the Gromov hyperbolicity constant of $Z$ and $c > 0$ is the constant in Equation \eqref{eqn:translationaxis}. 

It is a standard ergodicity argument that it suffices to show 
\begin{equation} \label{eqn:trabscont}
(T_{\varphi}^{*} \mu)(E) \ge \mu(E)
\end{equation}
for each Borel subset $E \subset \mathcal{H}_{\theta}$ (see \cite[Proof of Theorem 7.10]{CK_ML}, for instance). 

By Theorem \ref{thm:radonCharge}, $\mu$ is supported on $\La_{\theta}^{\varphi, C}(\Ga_0) \times \fa_{\theta}$. Hence, it suffices to consider the case
$$E = K \times I$$
for a compact subset $K \subset \La_{\theta}^{\varphi, C}(\Ga_0)$ and a compact box $I \subset \fa_{\theta}$, where $\fa_{\theta}$ is identified with $\Rb^{\# \theta}$ via $u \mapsto (\chi_{\alpha}(u))_{\alpha \in \theta}$. One can refer to \cite[Proof of Theorem 4.6]{CK_pradak} or \cite[Proof of Theorem 7.10]{CK_ML} for detailed deductions.

As before, we identify two limit sets $\La_{Z}(\Ga_0) \subset \partial Z$ and $\La_{\theta}(\Ga_0)$ via the $\Ga$-equivariant homeomorphism $\Psi : \La_{Z}(\Ga_0) \to \La_{\theta}(\Ga_0)$.
We fix some open subset $O \subset \La_Z(\Ga_0)$ such that $K \subset O$ and $\epsilon > 0$.
Let $N = N(\varphi, 4C, \epsilon, z_0) \in \Nb$ be given in Corollary \ref{cor:shadow_use}(2).
By Lemma \ref{lem:nbdBasis}, for each $x \in K$, there exist  $g(x) \in \Gamma_0$ and $n(x) > 100 \left(\frac{C}{\tau_{\varphi}} + N\right)$ 
such that 
  \[
x \in U_{C} \left(g(x); \varphi, n(x) \right) \subset O.
\]
Let $\mathcal{U} := \left\{ U_C \left(g(x); \varphi, n(x)\right) : x \in K \right\}$, which is a countable collection of sets. For convenience, let us enumerate $\mathcal{U}$ based on their $d_Z$-distances from $z_0$, i.e, let \[
\mathcal{U} = \{U_{1}, U_{2}, \ldots\}
\]
where 
$
U_j := U_C (g_j; \varphi, n_j)
$
for each $j \in \N$
so that 
 \[
d_Z(z_0, g_1 \varphi^{n_1} z_0) \le 
d_Z(z_0, g_2 \varphi^{n_2} z_0) \le \cdots.
\]

We will now define a subcollection 
$$\mathcal{V} := \{U_{i(1)}, U_{i(2)}, \ldots\} \subset \mathcal{U}$$ by inductively defining $i(1), i(2), \ldots$ as follows: let $i(1) = 1$. Having defined $i(1), \ldots, i(k)$, define $i(k+1)$ as the smallest $j \in \N$ such that $U_{j}$ is disjoint from $U_{i(1)} \cup \cdots \cup U_{i(k)}$.

For each $l \in \N$, we set 
\begin{equation} \label{eqn:defofCl}
C_{l} := U_{i(l)} \cup \bigcup \left\{ U_{k} : k \ge i(l), U_{k} \cap U_{i(l)} \neq \emptyset \right\}.
\end{equation}
Then $\{C_{l} : l \in \N \}$ is a covering of $K$ contained in $O$.

A very similar argument as in \cite[First claim in the proof of Theorem 7.10]{CK_ML} gives that 
for each $l \in \N$, 
\begin{equation} \label{eqn:trqiclaimClinU}
    C_{l} \subset U_C \left(g_{i(l)}; \varphi, n_{i(l)} - 1\right).
\end{equation}
Since some modifications are required, we present its proof as follows, for the sake of completeness.

Let $k \ge i(l)$ be such that $U_{k} = U_C (g_k; \varphi, n_k)$ and $U_{i(l)}$ have a common element $x$. As $\La_{\theta}$ and $\La_{Z}(\Ga)$ are identified,  $x$ is also regarded as a point in $\partial Z$. Then for every $z \in Z$ close enough to $x$ in $Z \cup \partial Z$, 
\begin{equation} \label{eqn:trqiclaim}
\left( z_0, g_{i(l)} [z_0, \varphi^{n_{i(l)}} z_0], z\right) \quad \text{is $C$-aligned.}
\end{equation}
 By Lemma \ref{lem:BGIPFellow}, there exists $q \in [z_0, z]$ that is $3C$-close to $g_{i(l)} \varphi^{n_{i(l)}}z_0$. Similarly, by the condition $x \in U_{k}$, $[z_0, z]$ contains a point $p$ that is $3C$-close to $g_{k} \varphi^{n_k} z_0$. Since $k \ge i(l)$,  our enumerating convention  tells us that \[
d(z_0, p) \ge d(z_0, g_k \varphi^{n_k} z_0) - 3C \ge d(z_0, g_{i(l)} \varphi^{n_{i(l)}} z_0 ) - 3C \ge d(z_0, q) - 6C.
\]
In other words, 
$$d(z, p) \le d(z, q) + 6C.$$

If $d(z, p) \le 10C$, then we have $d(z, g_{k} \varphi^{n_k} z_0) \le 13C$. By the $(1, 4C)$-Lipschitzness of $\pi_{g_{i(l)} [z_0, \varphi^n_{i(l)} z_0 ]}(\cdot)$ in Corollar \ref{cor:BGIPFellow}(1), we have\[
\diam  \pi_{g_{i(l)}[z_0, \varphi^{n_{i(l)}} z_0]} (\{z, g_{k} \varphi^{n_{k}} z_0\}) \le 17C.
 \]

 If $d(z, p) > 10C$, then we take a point $p^{\dagger} \in [z, p]$ such that $d(p, p^{\dagger}) = 10C$. Then $\diam [p^{\dagger}, z] = d(p,z) - 10C$, and hence
 $$\begin{aligned}
 d\left( g_{i(l)}[z_0, \varphi^{n_{i(l)}} z_0], [p^{\dagger}, z] \right) & \ge d\left(g_{i(l)}[z_0, \varphi^{n_{i(l)}}z_0], z\right) - \diam  [p^{\dagger}, z]  \\
 & \ge d(z, q) - 3C - d(z, p) + 10C \\
 & \ge C.
 \end{aligned}
$$
 By the $C$-contracting property of $g_{i(l)}[z_0, \varphi^{n_{i(l)}}z_0]$, we have \[
\diam \pi_{g_{i(l)}[z_0, \varphi^{n_{i(l)}} z_0]} (\{z, p^{\dagger}\})  \le C.
 \]
 Since $d(p, g_k \varphi^{n_k} z_0) \le 3C$, we have $d(p^{\dagger}, g_k \varphi^{n_k} z_0) \le 13C$, and hence
  \[
\diam \pi_{g_{i(l)}[z_0, \varphi^{n_{i(l)}} z_0]} (\{z, g_{k} \varphi^{n_{k}} z_0\}) \le 18C
 \]
 by Corollary \ref{cor:BGIPFellow}(1).

 Hence, in any case, together with Equation \eqref{eqn:trqiclaim}, we have that 
 \begin{equation} \label{eqn:secondineq19C}
(z_0, g_{i(l)}[z_0, \varphi^{n_{i(l)}} z_0], g_k \varphi^{n_k} z_0) \quad \text{is $19C$-aligned.}
 \end{equation}

Now let $y \in U_k$ be arbitrary. Then for every $z'\in Z$ close to $y$ in $Z \cup \partial Z$,
$$
(z_0, g_k [z_0, \varphi^{n_k} z_0], z') \quad \text{is $C$-aligned}
$$
and hence 
$[z_0, z']$ passes through the $3C$-neighborhood of $g_k \varphi^{n_k} z_0$ as before. Hence, we have 
\[\begin{aligned}
\diam   \pi_{g_{i(l)}[z_0, \varphi^{n_{i(l)}} z_0]}  ([z_{0}, z'])  
& \ge d(g_{i(l)}z_0, g_{i(l)}\varphi^{n_{i(l)}} z_0 ) \\
& \quad - \diam \left( \pi_{g_{i(l)}[z_0, \varphi^{n_{i(l)}} z_0]} (z_{0}) \cup \{ g_{i(l)} z_{0}  \} \right)\\
&\quad - \left( d([z_0, z'], g_k \varphi^{n_k} z_0 ) + 4C \right)\\
& \quad - \diam \left( \pi_{g_{i(l)}[z_0, \varphi^{n_{i(l)}} z_0]} (g_k \varphi^{n_k} z_0) \cup \{ g_{i(l)} \varphi^{n_{i(l)}} z_{0} \} \right) \\
&\ge 100C - C - ( 3C+ 4C) - 19C \\
& \ge 70C
\end{aligned}
\]
where we applied Corollary \ref{cor:BGIPFellow}(1) in the first inequality, and Equation \eqref{eqn:trqiclaim}, Equation \eqref{eqn:secondineq19C}, and that $d([z_0, z'], g_k \varphi^{n_k} z_0) \le 3C$ in the second.

We then apply Lemma \ref{lem:BGIPFellow} and obtain two points $u, v \in [z_0, z']$ such that 
\begin{enumerate}
\item $[u, v]$ and $\pi_{g_{i(l)}[z_0, \varphi^{n_{i(l)}} z_0]}([z_0, z'])$ are  within Hausdorff distance $4C$,
\item  $\diam \left( \pi_{g_{i(l)}[z_0, \varphi^{n_{i(l)}} z_0]} ([z_0, u]) \cup \{u\} \right) \le 2C$, 
\item $\diam \left( \pi_{g_{i(l)}[z_0, \varphi^{n_{i(l)}} z_0]} ([v, z']) \cup \{v\} \right) \le 2C$, and
\item for each $u' \in \pi_{g_{i(l)}[z_0, \varphi^{n_{i(l)}} z_0]}(z_0)$ and $v' \in \pi_{g_{i(l)}[z_0, \varphi^{n_{i(l)}} z_0]}(z')$, $[u', v']$ and $[u, v]$ are within Hausdorff distance $10C$.
\end{enumerate}
Since $d([z_0, z'], g_k \varphi^{n_k} z_0) \le 3C$, it follows from Equation \eqref{eqn:secondineq19C} and Corollary \ref{cor:BGIPFellow}(1) that there exists $w \in [z_0, z']$ such that 
\begin{equation} \label{eqn:usingnewitem}
\diam \left( \pi_{g_{i(l)}[z_0, \varphi^{n_{i(l)}} z_0]}(w) \cup \{g_{i(l)} \varphi^{n_{i(l)}} z_0\} \right) \le 26 C.
\end{equation}
By the condition (2) above, Equation \eqref{eqn:secondineq19C}, and $d(z_0, \varphi^{n_i(l)}z_0) > 100 C$, we must have  $w \notin [z_0, u]$, and hence  $w \in [v, z']$ or $w \in [u, v]$.

\begin{itemize}
\item If $w \in [v, z']$, then it follows from (3) above that 
$$\diam \left( \pi_{g_{i(l)}[z_0, \varphi^{n_{i(l)}} z_0]}(z') \cup \{g_{i(l)} \varphi^{n_{i(l)}} z_0\} \right) \le 28 C.
$$

\item If $w \in [u, v]$, then it follows from (4) above that for each pair of two points $u' \in \pi_{g_{i(l)}[z_0, \varphi^{n_{i(l)}} z_0]}(z_0)$ and $v' \in \pi_{g_{i(l)}[z_0, \varphi^{n_{i(l)}} z_0]}(z')$, there exists $w' \in [u', v']$ such that $d(w, w') \le 10C$. We then have
$$
\diam \left( \pi_{g_{i(l)}[z_0, \varphi^{n_{i(l)}} z_0]}(w) \cup \{w'\} \right) \le 20 C.
$$
Together with Equation \eqref{eqn:usingnewitem}, $\diam \{w', g_{i(l)} \varphi^{n_{i(l)}} z_0\} \le 46 C$. Since $v'$ is inbetween $w'$ and $g_{i(l)} \varphi^{n_{i(l)}} z_0$, and $v'  \in \pi_{g_{i(l)}[z_0, \varphi^{n_{i(l)}} z_0]}(z')$ is arbitrary, we have
$$
\diam \left( \pi_{g_{i(l)}[z_0, \varphi^{n_{i(l)}} z_0]}(z') \cup \{ g_{i(l)} \varphi^{n_{i(l)}} z_0 \} \right) \le 46 C.
$$
\end{itemize}
Therefore, $(g_{i(l)}[z_0, \varphi^{n_{i(l)}} z_0], z')$ is $46C$-aligned.

We now apply Lemma \ref{lem:BGIPHeredi}(4) to  $( \varphi^{-n_{i(l)}} g_{i(l)}^{-1} z', [z_0, \varphi^{- n_{i(l)}} z_0])$ and have 
$$
\pi_{ g_{i(l)} \varphi^{n_{i(l)}} \ga}(z') \subset g_{i(l)} \varphi^{n_{i(l)}} \ga \left([ - 47 C, + \infty) \right).
$$
As in the proof of Lemma \ref{lem:BGIPHeredi}(4), this implies
\begin{equation} \label{eqn:Clclaimcontradiction}
\begin{aligned}
\pi_{ g_{i(l)}  \ga}(z') & \subset g_{i(l)}  \ga \left([ n_{i(l)} \tau_{\varphi} - 6 \delta - 4 C - c - 47 C, + \infty) \right) \\
& \subset g_{i(l)}  \ga \left([ (n_{i(l)} - 1) \tau_{\varphi} + 6 \delta + 4 C + c + 1, + \infty) \right)
\end{aligned}
\end{equation}
Hence, $\left(g_{i(l)}[z_0, \varphi^{n_{i(l)}-1} z_0], z' \right)$ is $C$-aligned. To see this, suppose to the contrary that $\left( \varphi^{-(n_{i(l)} - 1)} g_{i(l)}^{-1} z', [z_0, \varphi^{-(n_{i(l)} - 1)} z_0] \right)$ is not $C$-aligned. Then it follows from Lemma \ref{lem:BGIPHeredi}(3) and the similar argument as above that
$$
\pi_{g_{i(l)} \ga}(z') \subset g_{i(l)} \ga \left( ( - \infty, (n_{i(l)} - 1) \tau_{\varphi} + 6 \delta + 4 C + c] \right)
$$
which is a contradiction.

Therefore, $\left(g_{i(l)}[z_0, \varphi^{n_{i(l)}-1} z_0], z' \right)$ is $C$-aligned.
Since this holds for every $z'$ close to $y \in U_k$, we have 
$$y \in U_{C} \left(g_{i(l)}; \varphi, n_{i(l)} - 1\right).$$ 
Equation \eqref{eqn:trqiclaimClinU} is  now established.

\medskip

Now for each  $l \in \N$, we  define a map $F_{l} : C_{l} \times I \rightarrow  \mathcal{H}_{\theta}$ as follows: for $g= g_{i(l)}$, we set 
\begin{equation} \label{eqn:defofFl}
F_{l} : \xi \mapsto g \varphi g^{-1} \xi.
\end{equation}
Then we have $\mu \left( F_{l}(C_{l} \times I) \right) = \mu(C_l \times I)$ as $\mu$ is $\Gamma$-invariant. 

\begin{claim*}
    We have 
    \begin{equation} \label{eqn:claimforFl}
    F_{l} (C_{l}\times I) \subset  U_{i(l)} \times \text{$(\epsilon$-neighborhood of $I + \la_{\theta}(\varphi))$}.
    \end{equation}
\end{claim*}
To see this, we simply write $g = g_{i(l)}$ and $n = n_{i(l)} - 1$. We then fix $\xi = (x, u) \in C_l \times I$. Note that
$$
F_l (\xi) = (g \varphi g^{-1} x, u + \sigma_{\theta}( g \varphi g^{-1}, x)).
$$
The inclusion for the first component can be shown by a very similar argument to  \cite[Second claim in the proof of Theorem 7.10]{CK_ML}. Again, as this requires  modifications of involved constants in our setting, we present the proof. 

Let $z  \in Z$ be a close point to $x$ in $Z \cup \partial Z$. It suffices to show that 
$$
(z_0, g[z_0, \varphi^{n+1} z_0], g \varphi g^{-1} z) \quad \text{is $C$-aligned for all large $i \in \N$,}
$$
Since $U_{i(l)} \neq \emptyset$, we already have that $(z_0, g[z_0, \varphi^{n+1} z_0] )$ is $C$-aligned. Now suppose to the contrary that $(g[z_0, \varphi^{n+1} z_0], g \varphi g^{-1} z)$ is not $C$-aligned.  In other words,
$( \varphi^{-n} g^{-1} z, [ z_0, \varphi^{-(n+1)} z_0] ) $ is not $C$-aligned. Then by Lemma \ref{lem:BGIPHeredi}(3) and as in the proof of Lemma \ref{lem:BGIPHeredi}(4), we have 
$$
\pi_{g \ga}(z) \subset g \ga ((-\infty, n \tau_{\varphi} + 6 \delta  + 4 C + c]).
$$
This contradicts Equation \eqref{eqn:Clclaimcontradiction}. Therefore, $g \varphi g^{-1} x \in U_{i(l)}$.

We now show the inclusion for the second component.
Since $U_{i(l)} \subset C_l$, it follows from the inclusion for the first component that  $g \varphi g^{-1} x \in C_l$. Hence, we have
$$
x, g \varphi g^{-1} x \in C_l \subset  U_C \left(g; \varphi, n\right)
$$
where the last inclusion is due to Equation \eqref{eqn:trqiclaimClinU}. By Lemma \ref{lem:shadowandalignment}(2), 
$$
x, g \varphi g^{-1} x \in O_{4C} (z_0, g z_0) \cap O_{4C}(g z_0, g \varphi^n z_0).$$
Now by Corollary \ref{cor:shadow_use},
$$
\norm{ \sigma_{\theta}(g \varphi g^{-1}, x) - \la_{\theta}(\varphi)} < \epsilon.
$$
Therefore, the inclusion for the second component follows.

\medskip

Now by the above claim and disjointness of $U_{i(l)}$'s,  we have
 \[\begin{aligned}
\mu( O \times (\textrm{$\epsilon$-neighborhood of $I+\la_{\theta}(\varphi)$})) &\ge \mu\left( \bigcup_{l} F_l (C_{l} \times I) \right) \\
&= \sum_{l} \mu \left( F_l (C_l \times I)\right) \\
&= \sum_{l} \mu (C_{l} \times I) \\
& \ge \mu(K \times I).
\end{aligned}
\]
Note that $\mu( O \times (\textrm{$\epsilon$-neighborhood of $I+\la_{\theta}(\varphi)$})) < + \infty$ since $\mu$ is Radon.
Since $\epsilon > 0$ and an open set $O \supset K$ are arbitrary, we have 
$$
(T_{\varphi}^* \mu)(K \times I) = \mu(K \times (I + \la_{\theta}({\varphi}))) \ge \mu(K \times I). 
$$
This finishes the proof of the main statement.

For the last assertion, note that we have shown 
$$
\Spec_{\theta}(\Ga) \subset \left\{ u \in \fa_{\theta} : T_u^* \mu \text{ and } \mu \text{ are equivalent} \right\}.
$$
Since the right hand side is a closed additive subgroup of $\fa_{\theta}$ \cite[0.1 Basic Lemma]{aaronson2002invariant}, if $\Spec_{\theta}(\Ga)$ is non-arithmetic, then $\mu$ is quasi-invariant under the $\fa_{\theta}$-action.
\end{proof}

\subsection{Proof of the rigidity}

Theorem \ref{thm:uniqueRadon} is now deduced from Theorem \ref{thm:trbytrlengthqi} and Theorem \ref{thm:CZZHTS}. This deduction is originally due to (\cite[0.1 Basic Lemma]{aaronson2002invariant}, \cite[Lemma 1]{Sarig_abelian}).

To elaborate more, if $\mu$ is a $\Ga$-invariant ergodic Radon measure on $\Hor_{\theta}$ supported on $\Rc_{\Ga, \theta}$, then  $\mu$ is quasi-invariant under the $\fa_{\theta}$-action, applied to the case $\Ga = \Ga_0$. 
By (\cite[0.1 Basic Lemma]{aaronson2002invariant}, \cite[Lemma 1]{Sarig_abelian}), this implies that there exists $\psi \in \fa_{\theta}^*$, $\delta > 0$, and $c > 0$ so that
$$
d \mu (x, u) = c \cdot e^{\delta \psi(u)} d \nu(x) du
$$
for a $\delta$-dimensional $\psi$-Patterson--Sullivan measure $\nu$ of $\Ga$, where $du$ is the Lebesgue measure on $\fa_{\theta}$. Since $\mu$ is supported on $\Rc_{\Ga, \theta} = \La_{\theta}^{\rm con}(\Ga) \times \fa_{\theta}$, it follows from  Theorem \ref{thm:CZZHTS} that $\delta = \delta_{\psi}(\Ga)$ and $\nu$ is of divergence type. This finishes the proof.
\qed

\subsection{Strengthening the Hopf--Tsuji--Sullivan dichotomy} \label{subsection:fullGuided}

One important portion of the Hopf--Tsuji--Sullivan dichotomy is that divergence-type Patterson--Sullivan measures are supported on the conical limit sets, as in Theorem \ref{thm:CZZHTS} (\cite{CZZ_transverse}, \cite{KOW_PD}).

As a corollary of Theorem \ref{thm:radonCharge}, we present a strenghtned version of this, by showing that such Patterson--Sullivan measures are in fact supported on guided limit sets. We also refer to (\cite{BLLO}, \cite{sambarino2022report}, \cite{KOW_SF}) for different type of the strengthening, in terms of so-called directional limit sets.

\begin{corollary}
    \label{cor:pattersonSqueeze}
    
    Let $\Ga < \Gsf$ be a non-elementary $\Psf_{\theta}$-hypertransverse subgroup and $\nu$ a divergence-type Patterson--Sullivan measure of $\Ga$.  Let $\varphi \in \Ga$ be loxodromic  and let $C = C(\varphi) > 0$ be as in Lemma \ref{lem:BGIPHeredi}. Then
    $$
    \nu ( \La_{\theta}^{\varphi, C}(\Ga)) = 1.
    $$
\end{corollary}

\begin{proof}

    By \cite{kim2024conformal}, the associated Burger--Roblin measure $\mu_{\nu}^{\BR}$ on $\Hor_{\theta}$ is $\Ga$-ergodic. Moreover, by the Hopf--Tsuji--Sullivan dichotomy for transverse subgroups (Theorem \ref{thm:CZZHTS}), $\mu_{\nu}^{\BR}$ is supported on $\La_{\theta}^{\rm con}(\Ga) \times \fa_{\theta}$. Then applying Theorem \ref{thm:radonCharge} finishes the proof.
\end{proof}

One can also deduce the normal subgroup version of Corollary \ref{cor:pattersonSqueeze}, as \cite{kim2024conformal} and Theorem \ref{thm:radonCharge} work for normal subgroups.

\subsection{Ergodic decomposition}

When $\theta = \Delta$, the $\Ga$-action on $\Hor_{\theta}$ is dual to the $\Nsf \Msf$-action on $\Ga \ba \Gsf$ (cf. \cite[Proposition 10.25]{LO_invariant}). 
For a Zariski dense Borel Anosov subgroup $\Ga < \Gsf$, the $\Nsf$-ergodic decomposition of a Burger--Roblin measure of $\Ga$ on $\Ga \ba \Gsf$ was proved Lee--Oh in \cite{LO_ergodic}, which was extended to $\Psf_{\Delta}$-hypertransverse case by the second author in \cite{kim2024conformal}.
Abusing notations, for a Patterson--Sullivan measure $\nu$ of $\Ga$, we write $\mu_{\nu}^{\BR}$ the associated Burger--Roblin measure on $\Ga \ba \Gsf$, defined in Equation \eqref{def:BRonhomogeneous}.

Recall that
$$
\Dc_{\Ga}
$$
is the finite collection of $\Psf^{\circ}$-minimal sets of $\Ga \ba \Gsf$. Then for a divergence-type Patterson--Sullivan measure $\nu$ of $\Ga$,
\begin{equation} \label{eqn:Nergodicdecomp}
\mu_{\nu}^{\BR} = \sum_{\E \in \Dc_{\Ga}} \mu_{\nu}^{\BR}|_{\E}
\end{equation}
is the $\Nsf$-ergodic decomposition proved in (\cite{LO_ergodic}, \cite{kim2024conformal}).

Using this, we deduce the $\Nsf$-ergodic measure classification from $\Nsf \Msf$-ergodic measure classification we obtained (Theorem \ref{thm:uniqueRadon}). We set
$$
\Rc_{\Ga} := \{ [g] \in \Ga \ba \Gsf : g \Psf \in \La_{\Delta}^{\rm con}(\Ga) \},
$$
which corresponds to $\Rc_{\Ga, \Delta} \subset \Hor_{\Delta}$.

\begin{theorem} \label{thm:hypertransverseNergodic}
Let $\Ga < \Gsf$ be a Zariski dense $\Psf_{\Delta}$-hypertransverse subgroup. If $\mu$ is an $\Nsf$-invariant ergodic Radon measure on $\Ga \ba \Gsf$ supported on $\Rc_{\Ga}$, then
$$
\mu \text{ is a constant multiple of } \mu_{\nu}^{\BR}|_{\E}
$$
for some divergence-type Patterson--Sullivan measure $\nu$ of $\Ga$ and $\E \in \Dc_{\Ga}$.
\end{theorem}

\begin{proof}
Consider the right $\Msf$-action on $\Ga \ba \Gsf$ and the measure
$$
\tilde \mu := \int_{\Msf} m_* \mu \ dm
$$
where $dm$ is the Haar probability measure on $\Msf$, which is an $\Nsf \Msf$-invariant Radon measure on $\Rc_{\Ga}$. For an $\Nsf \Msf$-invariant $E \subset \Ga \ba \Gsf$, we have
$$
\tilde \mu(E) = \int_{\Msf} \mu(E m^{-1}) dm = \mu(E).
$$
Hence, by the $\Nsf$-ergodicity and the $\Nsf$-invariance of $E$, we have that $E$ is either $\tilde \mu$-null or $\tilde \mu$-conull. Therefore, $\tilde \mu$ is $\Nsf \Msf$-ergodic.

Note that $\Spec_{\Delta}(\Ga)$ is non-arithmetic when $\Ga$ is Zariski dense, by the result of Benoist \cite{Benoist2000proprietes} (Theorem \ref{thm:Benoist_nonarithmetic}). It then follows from Theorem \ref{thm:uniqueRadon} that
$$\tilde \mu = c \cdot \mu_{\nu}^{\BR}$$
for some divergence-type Patterson--Sullivan measure $\nu$ of $\Ga$ and $c > 0$. 

Since $\Msf$ normalizes $\Nsf$, each $m^* \mu$ is $\Nsf$-invariant and $\Nsf$-ergodic as well. Hence, by the $\Nsf$-ergodic decomposition in Equation \eqref{eqn:Nergodicdecomp}, we have
$$
m^* \mu = c \cdot \mu_{\nu}^{\BR}|_{\E}
$$
for some $m \in \Msf$ and $\E \in \Dc_{\Ga}$.
In other words, $m^* \mu$ an $\Nsf$-ergodic component of $\mu_{\nu}^{\BR}$. Since $\mu_{\nu}^{\BR}$ is $\Msf$-invariant, $\frac{1}{c} \mu = m^{-1}_* \mu_{\nu}^{\BR}|_{\E}$ is also an $\Nsf$-ergodic component of $\mu_{\nu}^{\BR}$. This completes the proof.
\end{proof}

\subsection{Theorems in the introduction}

We now deduce theorems stated in the introduction from Theorem \ref{thm:uniqueRadon}. First, note that $\Spec_{\theta}(\Ga)$ is non-arithmetic when $\Ga$ is Zariski dense, by the result of Benoist \cite{Benoist2000proprietes} (Theorem \ref{thm:Benoist_nonarithmetic}). Then Theorem \ref{thm:uniqueRadon} implies the inclusion $(2) \subset (1)$ in  Theorem \ref{thm:maintransverse}, and the inclusion $(1) \subset (2)$ there is due to \cite{kim2024conformal} (see \cite{LO_invariant} for Borel Anosov cases). Corollary \ref{cor:mainrelAnosov} is Corollary \ref{cor:relAnosovcasemeasureclassification}.
Theorem \ref{thm:mainrelAnosov} is a dual statement of Corollary \ref{cor:relAnosovcasemeasureclassification}, and Theorem \ref{thm:mainAnosov} is a particular case of Theorem~\ref{thm:mainrelAnosov}.
Theorem \ref{thm:mainAnosovN} and Theorem \ref{thm:mainrelAnosovN} are special cases of Theorem \ref{thm:hypertransverseNergodic}.

\bibliographystyle{alpha} 
\bibliography{ML}

@article{Canary_ICM,
  title={Kleinian viewpoints on higher rank worlds},
  author={Canary, Richard D.},
  journal={arXiv preprint arXiv:2510.21334},
  year={2025}
}

@book{benoist2016random,
	author = {Benoist, Yves and Quint, Jean-Fran\c{c}ois},
	isbn = {978-3-319-47719-0; 978-3-319-47721-3},
	mrclass = {60B15 (22E30 22E46 37C30 37H15 60B20 60F05 60G50)},
	mrnumber = {3560700},
	mrreviewer = {Radhakrishnan Nair},
	pages = {xi+323},
	publisher = {Springer, Cham},
	series = {Ergebnisse der Mathematik und ihrer Grenzgebiete. 3. Folge. A Series of Modern Surveys in Mathematics [Results in Mathematics and Related Areas. 3rd Series. A Series of Modern Surveys in Mathematics]},
	title = {Random walks on reductive groups},
	url = {https://mathscinet.ams.org/mathscinet-getitem?mr=3560700},
	volume = {62},
	year = {2016}}

@article{CZZ_cusped,
	author = {Canary, Richard and Zhang, Tengren and Zimmer, Andrew},
	doi = {10.1016/j.aim.2022.108439},
	fjournal = {Advances in Mathematics},
	issn = {0001-8708,1090-2082},
	journal = {Adv. Math.},
	mrclass = {22E40 (20H10 30F35 37F32 57M60)},
	mrnumber = {4418884},
	mrreviewer = {Herbert\ Abels},
	pages = {Paper No. 108439, 67},
	title = {Cusped {H}itchin representations and {A}nosov representations of geometrically finite {F}uchsian groups},
	url = {https://doi.org/10.1016/j.aim.2022.108439},
	volume = {404},
	year = {2022}}

@article{Tits_representations,
	author = {Tits, J.},
	doi = {10.1515/crll.1971.247.196},
	fjournal = {Journal f\"{u}r die Reine und Angewandte Mathematik. [Crelle's Journal]},
	issn = {0075-4102,1435-5345},
	journal = {J. Reine Angew. Math.},
	mrclass = {14.50},
	mrnumber = {277536},
	mrreviewer = {J.\ Dieudonn\'{e}},
	pages = {196--220},
	title = {Repr\'{e}sentations lin\'{e}aires irr\'{e}ductibles d'un groupe r\'{e}ductif sur un corps quelconque},
	url = {https://doi.org/10.1515/crll.1971.247.196},
	volume = {247},
	year = {1971}}

@article{oh2025dynamics,
	author = {Oh, Hee},
	journal = {arXiv preprint arXiv:2510.10771},
	title = {Dynamics and {R}igidity through the {L}ens of {C}ircles},
	year = {2025}}

@article{KLP_Anosov,
	author = {Kapovich, Michael and Leeb, Bernhard and Porti, Joan},
	doi = {10.1007/s40879-017-0192-y},
	fjournal = {European Journal of Mathematics},
	issn = {2199-675X,2199-6768},
	journal = {Eur. J. Math.},
	mrclass = {22E40 (20F65 53C35)},
	mrnumber = {3736790},
	mrreviewer = {Herbert\ Abels},
	number = {4},
	pages = {808--898},
	title = {Anosov subgroups: dynamical and geometric characterizations},
	url = {https://doi.org/10.1007/s40879-017-0192-y},
	volume = {3},
	year = {2017}}

@incollection{gromov1987hyperbolic,
	author = {Gromov, M.},
	booktitle = {Essays in group theory},
	doi = {10.1007/978-1-4613-9586-7\_3},
	mrclass = {20F32 (20F06 20F10 22E40 53C20 57R75 58F17)},
	mrnumber = {919829},
	mrreviewer = {Christopher W. Stark},
	pages = {75--263},
	publisher = {Springer, New York},
	series = {Math. Sci. Res. Inst. Publ.},
	title = {Hyperbolic groups},
	url = {https://mathscinet.ams.org/mathscinet-getitem?mr=919829},
	volume = {8},
	year = {1987}}

@book{bowen2008equilibrium,
	author = {Bowen, Rufus},
	edition = {revised},
	isbn = {978-3-540-77605-5},
	mrclass = {37C40 (28D05 37A25 37D20)},
	mrnumber = {2423393},
	note = {With a preface by David Ruelle, Edited by Jean-Ren\'{e} Chazottes},
	pages = {viii+75},
	publisher = {Springer-Verlag, Berlin},
	series = {Lecture Notes in Mathematics},
	title = {Equilibrium states and the ergodic theory of {A}nosov diffeomorphisms},
	url = {https://mathscinet.ams.org/mathscinet-getitem?mr=2423393},
	volume = {470},
	year = {2008}}

@incollection{ghys1990bord,
	author = {Ghys, \'{E}tienne and de la Harpe, Pierre},
	booktitle = {Sur les groupes hyperboliques d'apr{\`e}s {M}ikhael {G}romov ({B}ern, 1988)},
	doi = {10.1007/978-1-4684-9167-8\_7},
	mrclass = {53C23 (53C22)},
	mrnumber = {1086655},
	pages = {117--134},
	publisher = {Birkh\"{a}user Boston, Boston, MA},
	series = {Progr. Math.},
	title = {Le bord d'un espace hyperbolique},
	url = {https://mathscinet.ams.org/mathscinet-getitem?mr=1086655},
	volume = {83},
	year = {1990}}

@article{BCZZ_2,
	author = {Blayac, Pierre-Louis and Canary, Richard and Zhu, Feng and Zimmer, Andrew},
	journal = {arXiv preprint arXiv:2404.09718},
	title = {Counting, mixing and equidistribution for {GPS} systems with applications to relatively Anosov groups},
	year = {2024}}

@article{CK_ML,
	author = {Choi, Inhyeok and Kim, Dongryul M.},
	journal = {arXiv preprint arXiv:2510.23256},
	title = {Invariant measures on the space of measured laminations for subgroups of mapping class group},
	year = {2025}}

@article{CK_pradak,
	author = {Choi, Inhyeok and Kim, Dongryul M.},
	journal = {arXiv preprint arXiv:2510.23365},
	title = {Classification of horospherical invariant measures in higher rank: Teaser},
	year = {2025}}

@book{coornaert1990geometrie,
	author = {Coornaert, M. and Delzant, T. and Papadopoulos, A.},
	isbn = {3-540-52977-2},
	mrclass = {57M07 (20F32)},
	mrnumber = {1075994},
	mrreviewer = {John Meier},
	note = {Les groupes hyperboliques de Gromov. [Gromov hyperbolic groups], With an English summary},
	pages = {x+165},
	publisher = {Springer-Verlag, Berlin},
	series = {Lecture Notes in Mathematics},
	title = {G\'{e}om\'{e}trie et th\'{e}orie des groupes},
	url = {https://mathscinet.ams.org/mathscinet-getitem?mr=1075994},
	volume = {1441},
	year = {1990}}

@article{BCZZ_PS,
	author = {Blayac, Pierre-Louis and Canary, Richard and Zhu, Feng and Zimmer, Andrew},
	journal = {arXiv preprint arXiv:2404.09713},
	title = {Patterson-{S}ullivan theory for coarse cocycles},
	year = {2024}}

@article{LLLO_Horospherical,
	author = {Landesberg, Or and Lee, Minju and Lindenstrauss, Elon and Oh, Hee},
	doi = {10.3934/jmd.2023009},
	fjournal = {Journal of Modern Dynamics},
	issn = {1930-5311,1930-532X},
	journal = {J. Mod. Dyn.},
	mrclass = {37A40 (11F30 22E40 37A17)},
	mrnumber = {4588420},
	pages = {331--362},
	title = {Horospherical invariant measures and a rank dichotomy for {A}nosov groups},
	url = {https://doi.org/10.3934/jmd.2023009},
	volume = {19},
	year = {2023}}

@article{LL_Radon,
	author = {Landesberg, Or and Lindenstrauss, Elon},
	doi = {10.1093/imrn/rnab024},
	fjournal = {International Mathematics Research Notices. IMRN},
	issn = {1073-7928,1687-0247},
	journal = {Int. Math. Res. Not. IMRN},
	mrclass = {22E40 (37C85 37D40)},
	mrnumber = {4458560},
	mrreviewer = {Boris\ Hasselblatt},
	number = {15},
	pages = {11602--11641},
	title = {On {R}adon measures invariant under horospherical flows on geometrically infinite quotients},
	url = {https://doi.org/10.1093/imrn/rnab024},
	year = {2022}}

@article{minsky1996quasi-projections,
	author = {Minsky, Yair N.},
	doi = {10.1515/crll.1995.473.121},
	fjournal = {Journal f\"{u}r die Reine und Angewandte Mathematik. [Crelle's Journal]},
	issn = {0075-4102},
	journal = {J. Reine Angew. Math.},
	mrclass = {32G15},
	mrnumber = {1390685},
	mrreviewer = {Pablo Ar\'{e}s Gastesi},
	pages = {121--136},
	title = {Quasi-projections in {T}eichm\"{u}ller space},
	url = {https://mathscinet.ams.org/mathscinet-getitem?mr=1390685},
	volume = {473},
	year = {1996}}

@article{Winter_mixing,
	author = {Winter, Dale},
	doi = {10.1007/s11856-015-1258-5},
	fjournal = {Israel Journal of Mathematics},
	issn = {0021-2172,1565-8511},
	journal = {Israel J. Math.},
	mrclass = {37A25 (22E40 37A17)},
	mrnumber = {3430281},
	mrreviewer = {Adrien\ Boyer},
	number = {1},
	pages = {467--507},
	title = {Mixing of frame flow for rank one locally symmetric spaces and measure classification},
	url = {https://doi.org/10.1007/s11856-015-1258-5},
	volume = {210},
	year = {2015}}

@article{sisto2018contracting,
	author = {Sisto, Alessandro},
	bdsk-color = {7},
	doi = {10.1515/crelle-2015-0093},
	fjournal = {Journal f\"{u}r die Reine und Angewandte Mathematik. [Crelle's Journal]},
	issn = {0075-4102},
	journal = {J. Reine Angew. Math.},
	mrclass = {20F65 (20F67 60G50)},
	mrnumber = {3849623},
	mrreviewer = {Enric Ventura Capell},
	pages = {79--114},
	title = {Contracting elements and random walks},
	url = {https://mathscinet.ams.org/mathscinet-getitem?mr=3849623},
	volume = {742},
	year = {2018}}

@article{yang2014growth,
	author = {Yang, Wen-yuan},
	doi = {10.1017/S0305004114000322},
	fjournal = {Mathematical Proceedings of the Cambridge Philosophical Society},
	issn = {0305-0041},
	journal = {Math. Proc. Cambridge Philos. Soc.},
	mrclass = {20F65},
	mrnumber = {3254594},
	mrreviewer = {Xiangdong Xie},
	number = {2},
	pages = {297--319},
	title = {Growth tightness for groups with contracting elements},
	url = {https://mathscinet.ams.org/mathscinet-getitem?mr=3254594},
	volume = {157},
	year = {2014}}

@article{sisto2013projections,
	author = {Sisto, Alessandro},
	doi = {10.4171/LEM/59-1-6},
	fjournal = {L'Enseignement Math\'{e}matique. Revue Internationale. 2e S\'{e}rie},
	issn = {0013-8584},
	journal = {Enseign. Math. (2)},
	mrclass = {20F67},
	mrnumber = {3113603},
	mrreviewer = {David Hume},
	number = {1-2},
	pages = {165--181},
	title = {Projections and relative hyperbolicity},
	url = {https://mathscinet.ams.org/mathscinet-getitem?mr=3113603},
	volume = {59},
	year = {2013}}

@article{arzhantseva2015growth,
	author = {Arzhantseva, Goulnara N. and Cashen, Christopher H. and Tao, Jing},
	doi = {10.2140/pjm.2015.278.1},
	fjournal = {Pacific Journal of Mathematics},
	issn = {0030-8730},
	journal = {Pacific J. Math.},
	mrclass = {20F65 (20E06 20F67 37C35 57M07)},
	mrnumber = {3404665},
	mrreviewer = {\L ukasz Garncarek},
	number = {1},
	pages = {1--49},
	title = {Growth tight actions},
	url = {https://mathscinet.ams.org/mathscinet-getitem?mr=3404665},
	volume = {278},
	year = {2015}}

@article{yang2019statistically,
	author = {Yang, Wen-yuan},
	doi = {10.1093/imrn/rny001},
	fjournal = {International Mathematics Research Notices. IMRN},
	issn = {1073-7928},
	journal = {Int. Math. Res. Not. IMRN},
	mrclass = {20F65 (22E40 37C85 57M07)},
	mrnumber = {4039013},
	mrreviewer = {Camille Horbez},
	number = {23},
	pages = {7259--7323},
	title = {Statistically convex-cocompact actions of groups with contracting elements},
	url = {https://mathscinet.ams.org/mathscinet-getitem?mr=4039013},
	year = {2019}}

@article{chawla2023genericity,
	author = {Chawla, Kunal and Choi, Inhyeok and Tiozzo, Giulio},
	journal = {Groups, Geometry, and Dynamics},
	title = {Genericity of contracting geodesics in groups},
	year = {2025}}

@article{aaronson2002invariant,
	author = {Aaronson, Jon and Nakada, Hitoshi and Sarig, Omri and Solomyak, Rita},
	doi = {10.1007/BF02785420},
	fjournal = {Israel Journal of Mathematics},
	issn = {0021-2172},
	journal = {Israel J. Math.},
	mrclass = {37A40 (28D05 37A15)},
	mrnumber = {1910377},
	mrreviewer = {G. R. Goodson},
	pages = {93--134},
	title = {Invariant measures and asymptotics for some skew products},
	url = {https://doi.org/10.1007/BF02785420},
	volume = {128},
	year = {2002}}

@article{yang2024conformal,
	author = {Yang, Wen-yuan},
	journal = {arXiv preprint arXiv:2208.04861},
	title = {Conformal dynamics at infinity for groups with contracting elements},
	year = {2024}}

@article{dani1981invariant,
	author = {Dani, S. G.},
	doi = {10.1007/BF01389173},
	fjournal = {Inventiones Mathematicae},
	issn = {0020-9910},
	journal = {Invent. Math.},
	mrclass = {22D40 (22E40 28D10 57S20 58F17)},
	mrnumber = {629475},
	mrreviewer = {Caroline Series},
	number = {2},
	pages = {357--385},
	title = {Invariant measures and minimal sets of horospherical flows},
	url = {https://doi.org/10.1007/BF01389173},
	volume = {64},
	year = {1981}}

@article{coulon2024patterson-sullivan,
	author = {Coulon, R\'{e}mi},
	doi = {10.1017/etds.2024.10},
	fjournal = {Ergodic Theory and Dynamical Systems},
	issn = {0143-3857},
	journal = {Ergodic Theory Dynam. Systems},
	mrclass = {20F65 (37A15 37A35 37D40 57S30)},
	mrnumber = {4803663},
	number = {11},
	pages = {3216--3271},
	title = {Patterson-{S}ullivan theory for groups with a strongly contracting element},
	url = {https://doi.org/10.1017/etds.2024.10},
	volume = {44},
	year = {2024}}

@article{veech1977unique,
	author = {Veech, William A.},
	doi = {10.2307/2373868},
	fjournal = {American Journal of Mathematics},
	issn = {0002-9327},
	journal = {Amer. J. Math.},
	mrclass = {22E40 (28A65 58F99)},
	mrnumber = {447476},
	mrreviewer = {S. G. Dani},
	number = {4},
	pages = {827--859},
	title = {Unique ergodicity of horospherical flows},
	url = {https://doi.org/10.2307/2373868},
	volume = {99},
	year = {1977}}

@inproceedings{furstenberg1973the-unique,
	author = {Furstenberg, Harry},
	booktitle = {Recent advances in topological dynamics ({P}roc. {C}onf. {T}opological {D}ynamics, {Y}ale {U}niv., {N}ew {H}aven, {C}onn., 1972; in honor of {G}ustav {A}rnold {H}edlund)},
	mrclass = {22D40 (58F15)},
	mrnumber = {393339},
	mrreviewer = {Leonard F. Richardson},
	pages = {95--115},
	publisher = {Springer, Berlin-New York},
	series = {Lecture Notes in Math., Vol. 318},
	title = {The unique ergodicity of the horocycle flow},
	url = {https://mathscinet.ams.org/mathscinet-getitem?mr=393339},
	year = {1973}}

@article{dani1978invariant,
	author = {Dani, S. G.},
	doi = {10.1007/BF01578067},
	fjournal = {Inventiones Mathematicae},
	issn = {0020-9910},
	journal = {Invent. Math.},
	mrclass = {22D40 (58F15)},
	mrnumber = {578655},
	mrreviewer = {William Perrizo},
	number = {2},
	pages = {101--138},
	title = {Invariant measures of horospherical flows on noncompact homogeneous spaces},
	url = {https://doi.org/10.1007/BF01578067},
	volume = {47},
	year = {1978}}

@book{Bridson1999metric,
	author = {Bridson, Martin R. and Haefliger, Andr\'e},
	doi = {10.1007/978-3-662-12494-9},
	isbn = {3-540-64324-9},
	mrclass = {53C23 (20F65 53C70 57M07)},
	mrnumber = {1744486},
	mrreviewer = {Athanase\ Papadopoulos},
	pages = {xxii+643},
	publisher = {Springer-Verlag, Berlin},
	series = {Grundlehren der mathematischen Wissenschaften [Fundamental Principles of Mathematical Sciences]},
	title = {Metric spaces of non-positive curvature},
	url = {https://doi.org/10.1007/978-3-662-12494-9},
	volume = {319},
	year = {1999}}

@article{kim2024conformal,
	author = {Kim, Dongryul M.},
	journal = {arXiv preprint arXiv:2404.13727},
	title = {Conformal measure rigidity and ergodicity of horospherical foliations},
	year = {2024}}

@article{Burger_horoc,
	author = {Burger, Marc},
	doi = {10.1215/S0012-7094-90-06129-0},
	fjournal = {Duke Mathematical Journal},
	issn = {0012-7094,1547-7398},
	journal = {Duke Math. J.},
	mrclass = {58F17},
	mrnumber = {1084459},
	mrreviewer = {Robert\ Brooks},
	number = {3},
	pages = {779--803},
	title = {Horocycle flow on geometrically finite surfaces},
	url = {https://doi.org/10.1215/S0012-7094-90-06129-0},
	volume = {61},
	year = {1990}}

@article{KOW_SF,
	author = {Kim, Dongryul M. and Oh, Hee and Wang, Yahui},
	doi = {10.1090/cams/43},
	fjournal = {Communications of the American Mathematical Society},
	issn = {2692-3688},
	journal = {Commun. Am. Math. Soc.},
	mrclass = {37A17 (22E40 37A40)},
	mrnumber = {4867107},
	mrreviewer = {Cheng\ Zheng},
	pages = {1--47},
	title = {Ergodic dichotomy for subspace flows in higher rank},
	url = {https://doi.org/10.1090/cams/43},
	volume = {5},
	year = {2025}}

@article{CZZ_relative,
	author = {Canary, Richard and Zhang, Tengren and Zimmer, Andrew},
	doi = {10.1007/s00208-025-03137-2},
	fjournal = {Mathematische Annalen},
	issn = {0025-5831,1432-1807},
	journal = {Math. Ann.},
	mrclass = {22F10 (22E40)},
	mrnumber = {4906322},
	mrreviewer = {Boris\ Hasselblatt},
	number = {2},
	pages = {2309--2363},
	title = {Patterson-{S}ullivan measures for relatively {A}nosov groups},
	url = {https://doi.org/10.1007/s00208-025-03137-2},
	volume = {392},
	year = {2025}}

@article{KOW_PD,
	author = {Kim, Dongryul M. and Oh, Hee and Wang, Yahui},
	doi = {10.1007/s00208-025-03292-6},
	fjournal = {Mathematische Annalen},
	issn = {0025-5831,1432-1807},
	journal = {Math. Ann.},
	mrclass = {22E40 (14L15 20H20 37A25 37C85)},
	mrnumber = {4984347},
	number = {2},
	pages = {2391--2450},
	title = {Properly discontinuous actions, growth indicators, and conformal measures for transverse subgroups},
	url = {https://doi.org/10.1007/s00208-025-03292-6},
	volume = {393},
	year = {2025}}

@article{CZZ_transverse,
	author = {Canary, Richard and Zhang, Tengren and Zimmer, Andrew},
	doi = {10.3934/jmd.2024009},
	fjournal = {Journal of Modern Dynamics},
	issn = {1930-5311,1930-532X},
	journal = {J. Mod. Dyn.},
	mrclass = {22E40 (20H10 37A05)},
	mrnumber = {4799467},
	mrreviewer = {Cheng\ Zheng},
	pages = {319--377},
	title = {Patterson-{S}ullivan measures for transverse subgroups},
	url = {https://doi.org/10.3934/jmd.2024009},
	volume = {20},
	year = {2024}}

@article{sambarino2022report,
	author = {Sambarino, A.},
	doi = {10.1017/etds.2023.13},
	fjournal = {Ergodic Theory and Dynamical Systems},
	issn = {0143-3857,1469-4417},
	journal = {Ergodic Theory Dynam. Systems},
	mrclass = {22E40 (37Axx 37Dxx)},
	mrnumber = {4676211},
	number = {1},
	pages = {236--289},
	title = {A report on an ergodic dichotomy},
	url = {https://doi.org/10.1017/etds.2023.13},
	volume = {44},
	year = {2024}}

@article{GGKW2017,
	author = {Gu\'eritaud, Fran\c{c}ois and Guichard, Olivier and Kassel, Fanny and Wienhard, Anna},
	doi = {10.2140/gt.2017.21.485},
	fjournal = {Geometry \& Topology},
	issn = {1465-3060,1364-0380},
	journal = {Geom. Topol.},
	mrclass = {37D40 (20F67 22E40 57S30)},
	mrnumber = {3608719},
	number = {1},
	pages = {485--584},
	title = {Anosov representations and proper actions},
	url = {https://doi-org.ezproxy.library.wisc.edu/10.2140/gt.2017.21.485},
	volume = {21},
	year = {2017}}

@article{Ratner_measure,
	author = {Ratner, Marina},
	doi = {10.2307/2944357},
	fjournal = {Annals of Mathematics. Second Series},
	issn = {0003-486X,1939-8980},
	journal = {Ann. of Math. (2)},
	mrclass = {22E40 (58F11 58F17)},
	mrnumber = {1135878},
	mrreviewer = {S.\ G.\ Dani},
	number = {3},
	pages = {545--607},
	title = {On {R}aghunathan's measure conjecture},
	url = {https://doi.org/10.2307/2944357},
	volume = {134},
	year = {1991}}

@article{Landesberg_horospherically,
	author = {Landesberg, Or},
	doi = {10.3934/jmd.2021012},
	fjournal = {Journal of Modern Dynamics},
	issn = {1930-5311,1930-532X},
	journal = {J. Mod. Dyn.},
	mrclass = {37D40 (30F40 37A17 37A40 57K32)},
	mrnumber = {4312138},
	mrreviewer = {Jack\ O.\ Button},
	pages = {337--352},
	title = {Horospherically invariant measures and finitely generated {K}leinian groups},
	url = {https://doi.org/10.3934/jmd.2021012},
	volume = {17},
	year = {2021}}

@article{OP_local,
	author = {Oh, Hee and Pan, Wenyu},
	doi = {10.1093/imrn/rnx292},
	fjournal = {International Mathematics Research Notices. IMRN},
	issn = {1073-7928,1687-0247},
	journal = {Int. Math. Res. Not. IMRN},
	mrclass = {37D40 (22E40 37A25)},
	mrnumber = {4016891},
	mrreviewer = {Jonas\ Der\'e},
	number = {19},
	pages = {6036--6088},
	title = {Local mixing and invariant measures for horospherical subgroups on abelian covers},
	url = {https://doi.org/10.1093/imrn/rnx292},
	year = {2019}}

@article{LS_periodic,
	author = {Ledrappier, Fran\c{c}ois and Sarig, Omri},
	doi = {10.1007/s11856-007-0064-0},
	fjournal = {Israel Journal of Mathematics},
	issn = {0021-2172,1565-8511},
	journal = {Israel J. Math.},
	mrclass = {37D40 (37A25 37C40)},
	mrnumber = {2342499},
	mrreviewer = {Bachir\ Bekka},
	pages = {281--315},
	title = {Invariant measures for the horocycle flow on periodic hyperbolic surfaces},
	url = {https://doi.org/10.1007/s11856-007-0064-0},
	volume = {160},
	year = {2007}}

@article{Ledrappier_invariant,
	author = {Ledrappier, Fran\c{c}ois},
	doi = {10.5802/aif.2345},
	fjournal = {Universit\'e{} de Grenoble. Annales de l'Institut Fourier},
	issn = {0373-0956,1777-5310},
	journal = {Ann. Inst. Fourier (Grenoble)},
	mrclass = {37D40 (37A40 53C12 53D25)},
	mrnumber = {2401217},
	mrreviewer = {Athanase\ Papadopoulos},
	number = {1},
	pages = {85--105},
	title = {Invariant measures for the stable foliation on negatively curved periodic manifolds},
	url = {https://doi.org/10.5802/aif.2345},
	volume = {58},
	year = {2008}}

@article{Sarig_genus,
	author = {Sarig, Omri},
	doi = {10.1007/s00039-010-0048-9},
	fjournal = {Geometric and Functional Analysis},
	issn = {1016-443X,1420-8970},
	journal = {Geom. Funct. Anal.},
	mrclass = {37D40 (31C12 37A17 37A40)},
	mrnumber = {2594621},
	mrreviewer = {S.\ G.\ Dani},
	number = {6},
	pages = {1757--1812},
	title = {The horocyclic flow and the {L}aplacian on hyperbolic surfaces of infinite genus},
	url = {https://doi.org/10.1007/s00039-010-0048-9},
	volume = {19},
	year = {2010}}

@article{BLLO,
	author = {Burger, Marc and Landesberg, Or and Lee, Minju and Oh, Hee},
	doi = {10.3934/jmd.2023008},
	fjournal = {Journal of Modern Dynamics},
	issn = {1930-5311,1930-532X},
	journal = {J. Mod. Dyn.},
	mrclass = {37A40 (22E40 37A17)},
	mrnumber = {4588419},
	pages = {301--330},
	title = {The {H}opf-{T}suji-{S}ullivan dichotomy in higher rank and applications to {A}nosov subgroups},
	url = {https://doi.org/10.3934/jmd.2023008},
	volume = {19},
	year = {2023}}

@article{LO_invariant,
	author = {Lee, Minju and Oh, Hee},
	doi = {10.1093/imrn/rnac262},
	fjournal = {International Mathematics Research Notices. IMRN},
	issn = {1073-7928,1687-0247},
	journal = {Int. Math. Res. Not. IMRN},
	mrclass = {37D40 (37A99)},
	mrnumber = {4651889},
	mrreviewer = {S.\ G.\ Dani},
	number = {19},
	pages = {16226--16295},
	title = {Invariant measures for horospherical actions and {A}nosov groups},
	url = {https://doi.org/10.1093/imrn/rnac262},
	year = {2023}}

@article{LO_ergodic,
	author = {Lee, Minju and Oh, Hee},
	doi = {10.1007/s11856-023-2560-2},
	fjournal = {Israel Journal of Mathematics},
	issn = {0021-2172,1565-8511},
	journal = {Israel J. Math.},
	mrclass = {37A17 (22E40 22F30 37A25)},
	mrnumber = {4756999},
	mrreviewer = {Boris\ Hasselblatt},
	number = {1},
	pages = {195--234},
	title = {Ergodic decompositions of geometric measures on {A}nosov homogeneous spaces},
	url = {https://doi.org/10.1007/s11856-023-2560-2},
	volume = {260},
	year = {2024}}

@article{Edwards2020anosov,
	author = {Edwards, Sam and Lee, Minju and Oh, Hee},
	doi = {10.2140/gt.2023.27.513},
	fjournal = {Geometry \& Topology},
	issn = {1465-3060,1364-0380},
	journal = {Geom. Topol.},
	mrclass = {22E40 (37A17 37A25 37A44)},
	mrnumber = {4589560},
	number = {2},
	pages = {513--573},
	title = {Anosov groups: local mixing, counting and equidistribution},
	url = {https://doi.org/10.2140/gt.2023.27.513},
	volume = {27},
	year = {2023}}

@article{Labourie2006anosov,
	author = {Labourie, Fran\c{c}ois},
	doi = {10.1007/s00222-005-0487-3},
	fjournal = {Inventiones Mathematicae},
	issn = {0020-9910,1432-1297},
	journal = {Invent. Math.},
	mrclass = {20F65 (37D20 37F30)},
	mrnumber = {2221137},
	mrreviewer = {Richard\ Kenyon},
	number = {1},
	pages = {51--114},
	title = {Anosov flows, surface groups and curves in projective space},
	url = {https://doi.org/10.1007/s00222-005-0487-3},
	volume = {165},
	year = {2006}}

@article{Guichard2012anosov,
	author = {Guichard, Olivier and Wienhard, Anna},
	doi = {10.1007/s00222-012-0382-7},
	fjournal = {Inventiones Mathematicae},
	issn = {0020-9910,1432-1297},
	journal = {Invent. Math.},
	mrclass = {22F30 (32G15 53C30 53D25)},
	mrnumber = {2981818},
	mrreviewer = {Pablo\ Su\'arez-Serrato},
	number = {2},
	pages = {357--438},
	title = {Anosov representations: domains of discontinuity and applications},
	url = {https://doi.org/10.1007/s00222-012-0382-7},
	volume = {190},
	year = {2012}}

@incollection{Benoist2000proprietes,
	author = {Benoist, Yves},
	booktitle = {Analysis on homogeneous spaces and representation theory of {L}ie groups, {O}kayama--{K}yoto (1997)},
	doi = {10.2969/aspm/02610033},
	isbn = {4-314-10138-5},
	mrclass = {22E10 (22E40)},
	mrnumber = {1770716},
	mrreviewer = {F.\ E. A. Johnson},
	pages = {33--48},
	publisher = {Math. Soc. Japan, Tokyo},
	series = {Adv. Stud. Pure Math.},
	title = {Propri\'et\'es asymptotiques des groupes lin\'eaires. {II}},
	url = {https://doi.org/10.2969/aspm/02610033},
	volume = {26},
	year = {2000}}

@article{Sarig_abelian,
	author = {Sarig, Omri},
	doi = {10.1007/s00222-004-0357-4},
	fjournal = {Inventiones Mathematicae},
	issn = {0020-9910,1432-1297},
	journal = {Invent. Math.},
	mrclass = {37D40 (37A05 37A20)},
	mrnumber = {2092768},
	mrreviewer = {Boris\ Hasselblatt},
	number = {3},
	pages = {519--551},
	title = {Invariant {R}adon measures for horocycle flows on abelian covers},
	url = {https://doi.org/10.1007/s00222-004-0357-4},
	volume = {157},
	year = {2004}}

@incollection{Babillot_nilpotent,
	author = {Babillot, Martine},
	booktitle = {Random walks and geometry},
	isbn = {3-11-017237-2},
	mrclass = {37D40 (37A05)},
	mrnumber = {2087786},
	mrreviewer = {Boris\ Hasselblatt},
	pages = {319--335},
	publisher = {Walter de Gruyter, Berlin},
	title = {On the classification of invariant measures for horosphere foliations on nilpotent covers of negatively curved manifolds},
	year = {2004}}

@incollection{BL_covers,
	author = {Babillot, Martine and Ledrappier, Fran\c{c}ois},
	booktitle = {Lie groups and ergodic theory ({M}umbai, 1996)},
	isbn = {81-7319-235-9},
	mrclass = {37D40 (37C40)},
	mrnumber = {1699356},
	mrreviewer = {Gerhard\ Knieper},
	pages = {1--32},
	publisher = {Tata Inst. Fund. Res., Bombay},
	series = {Tata Inst. Fund. Res. Stud. Math.},
	title = {Geodesic paths and horocycle flow on abelian covers},
	volume = {14},
	year = {1998}}

@article{Quint2002divergence,
	author = {Quint, J.-F.},
	doi = {10.1007/s00014-002-8352-0},
	fjournal = {Commentarii Mathematici Helvetici},
	issn = {0010-2571},
	journal = {Comment. Math. Helv.},
	mrclass = {22E40},
	mrnumber = {1933790},
	mrreviewer = {Herbert Abels},
	number = {3},
	pages = {563--608},
	title = {Divergence exponentielle des sous-groupes discrets en rang sup\'{e}rieur},
	url = {https://doi.org/10.1007/s00014-002-8352-0},
	volume = {77},
	year = {2002}}

@article{Roblin2003ergodicite,
	author = {Roblin, Thomas},
	doi = {10.24033/msmf.408},
	fjournal = {M\'emoires de la Soci\'et\'e{} Math\'ematique de France. Nouvelle S\'erie},
	issn = {0249-633X,2275-3230},
	journal = {M\'em. Soc. Math. Fr. (N.S.)},
	mrclass = {37D40 (28D10 30F40 37A25 37C40 53C22 53D25 57R30)},
	mrnumber = {2057305},
	mrreviewer = {Marc\ Peign\'e},
	number = {95},
	pages = {vi+96},
	title = {Ergodicit\'e{} et \'equidistribution en courbure n\'egative},
	url = {https://doi.org/10.24033/msmf.408},
	year = {2003}}

@article{Quint2002Mesures,
	author = {Quint, J.-F.},
	doi = {10.1007/s00039-002-8266-4},
	fjournal = {Geometric and Functional Analysis},
	issn = {1016-443X},
	journal = {Geom. Funct. Anal.},
	mrclass = {22E40 (37C35)},
	mrnumber = {1935549},
	mrreviewer = {Michel Coornaert},
	number = {4},
	pages = {776--809},
	title = {Mesures de {P}atterson-{S}ullivan en rang sup\'{e}rieur},
	url = {https://doi.org/10.1007/s00039-002-8266-4},
	volume = {12},
	year = {2002}}

@article{Benoist1997proprietes,
	author = {Benoist, Y.},
	doi = {10.1007/PL00001613},
	fjournal = {Geometric and Functional Analysis},
	issn = {1016-443X},
	journal = {Geom. Funct. Anal.},
	mrclass = {22E15},
	mrnumber = {1437472},
	mrreviewer = {Scot Adams},
	number = {1},
	pages = {1--47},
	title = {Propri\'{e}t\'{e}s asymptotiques des groupes lin\'{e}aires},
	url = {https://doi.org/10.1007/PL00001613},
	volume = {7},
	year = {1997}}

\end{document}